\documentclass[11pt]{article}
\usepackage{amsmath}
\usepackage{amssymb}
\usepackage{amsthm}
\usepackage{enumerate}
\usepackage{setspace}
\usepackage{graphicx,color}
\usepackage{enumitem}
\usepackage{wrapfig}
\usepackage[margin=1in]{geometry}
\parskip \medskipamount
\setlist[enumerate]{noitemsep,topsep=0pt,label=(\roman*)}

\newtheorem{theorem}{Theorem}[section]
\newtheorem{lemma}[theorem]{Lemma}
\newtheorem{proposition}[theorem]{Proposition}
\newtheorem{corollary}[theorem]{Corollary}

\theoremstyle{remark}
\newtheorem{remark}[theorem]{Remark}

\newcommand{\ind}[1]{{\text{\Large $\mathfrak 1$}}\left(#1\right)}
\newcommand{\PR}{\mathbb{P}}
\newcommand{\E}{\mathbb{E}}
\newcommand{\pr}[1]{\operatorname{\mathbf{P}}\left(#1\right)}

\newcommand{\bc}{\mathrm{bc}}
\newcommand{\finc}{F_\mathrm{inc}}
\newcommand{\fdec}{F_\mathrm{dec}}
\newcommand{\fdiag}{F_\mathrm{diag}}

\newcommand{\leaves}{\mathrm{leaves}}
\newcommand{\good}{\mathrm{good}}
\newcommand{\sleaves}{\mathrm{sl}}
\newcommand{\gs}{\mathrm{gs}}
\newcommand{\compl}{\textrm{c}}

\renewcommand{\tilde}{\widetilde}
\renewcommand{\hat}{\widehat}
\begin{document}

\title{A Lyapunov function for Glauber dynamics on lattice triangulations}

\author{Alexandre Stauffer\thanks{University of Bath, Bath, UK; a.stauffer@bath.ac.uk. Supported in part by a Marie Curie Career Integration
Grant PCIG13-GA-2013-618588 DSRELIS.}}
\date{}
\maketitle

\begin{abstract}
   We study random triangulations of the integer points $[0,n]^2 \cap\mathbb{Z}^2$, where each triangulation $\sigma$
   has probability measure $\lambda^{|\sigma|}$ with $|\sigma|$ denoting the sum of the length of the edges in $\sigma$.
   Such triangulations are called \emph{lattice triangulations}.
   We construct a height function on lattice triangulations and prove that, in the whole subcritical regime $\lambda<1$, 
   the function behaves as a \emph{Lyapunov function} with respect to Glauber dynamics; that is, the function is a supermartingale.
   We show the applicability of the above result by establishing several features of lattice triangulations, 
   such as tightness of local measures, exponential tail of edge lengths, crossings of small triangles, and decay of correlations in thin rectangles. 
   These are the first results on lattice triangulations that are valid in the whole subcritical regime $\lambda<1$.
   In a very recent work with Caputo, Martinelli and Sinclair, we apply this Lyapunov function to establish tight bounds on the mixing time of 
   Glauber dynamics in thin rectangles that hold for all $\lambda<1$.   
   The Lyapunov function result here holds in great generality; it holds for triangulations of general lattice polygons (instead of the $[0,n]^2$ square) and 
   also in the presence of arbitrary constraint edges.
\newline
MSC 2010 \emph{subject classifications.}
Primary 60J10; % discrete markov chains 
Secondary 60K35, % statistical mechanics
          52C20, % tessellation and tilings
          05C10. % planar graphs
\end{abstract}

%############################################################################################
%############################################################################################
%############################################################################################
\section{Introduction}\label{sec:intro}
Consider the set of integer points $\Lambda^0_{n}=\{0,1,\ldots,n\}^2$ in the plane.
A triangulation $\sigma$ of $\Lambda^0_{n}$ is a maximal collection of edges (straight line segments) such that each edge has its endpoints in 
$\Lambda^0_{n}$ and, aside from its endpoints, intersects no other edge of $\sigma$ and no point of $\Lambda^0_{n}$. Figure~\ref{fig:triang} illustrates a
triangulation with $n=50$.
\begin{figure}[htbp]
   \begin{center}
      \includegraphics[scale=.1]{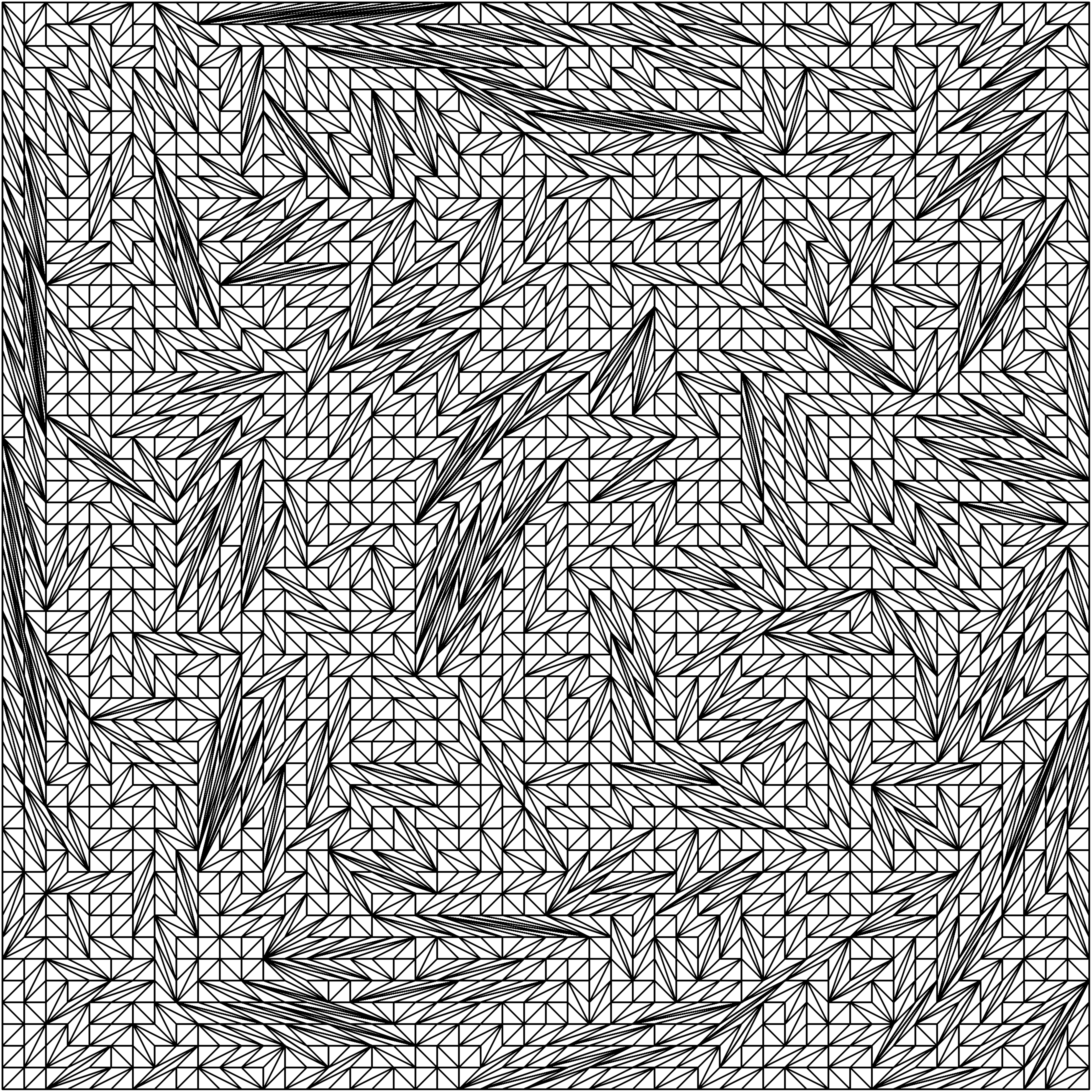}
   \end{center}\vspace{-.5cm}
   \caption{A $50\times 50$ lattice triangulation.}
   \label{fig:triang}
\end{figure}

Our goal is to study properties of \emph{random} triangulations.
Let $\Omega_n$ be the set of all triangulations of $\Lambda^0_n$. 
It is known that, for any triangulation $\sigma \in \Omega_{n}$,
every triangle in $\sigma$ has area exactly $1/2$, and the set of midpoints 
of the edges of $\sigma$ does \emph{not} depend on $\sigma$. 
In particular, this set is 
$$
   \Lambda_{n} = 
   \big\{0, \tfrac{1}{2}, 1, \tfrac{3}{2},\ldots, n-\tfrac{1}{2},n\big\}^2 \setminus \Lambda_{n}^0,
$$
which is the set of half-integer points in $[0,n]^2$ excluding $\Lambda_{n}^0$. 
This allows 
us to regard random lattice triangulations as a \emph{spin system} since 
a lattice triangulation $\sigma \in \Omega_{n}$ can be seen as a collection of variables $\{\sigma_x \colon x \in \Lambda_{n}\}$, where $\sigma_x$ denotes
the edge (representing the spin) of the midpoint $x$ in $\sigma$. 
However, many challenges arise when trying to make use of the vast literature on spin systems to study lattice triangulations. 
For example, lattice triangulations have unbounded dependences as long edges affect far away regions, the interaction between the spins depends on the triangulation,
and some useful properties in the study of spin systems do not hold, one example being the FKG inequality, see Appendix~\ref{sec:nofkg}.

There is a natural Markov chain (or Glauber dynamics) over $\Omega_n$, where transitions are given by 
\emph{flips} of uniformly random edges~\cite{KZ,Welzl}. More precisely, if $\sigma\in\Omega_n$ is the current state of the Markov chain, 
a transition consists of picking a non-boundary edge $e$ of $\sigma$ uniformly at random, and if the two triangles containing $e$ in $\sigma$ form
a strictly convex quadrilateral (in which case they actually form a parallelogram), then with probability $1/2$ 
we remove $e$ and replace it by the opposite diagonal of that 
quadrilateral. Otherwise, the Markov chain stays put; see Figure~\ref{fig:flips}.
\begin{figure}[htbp]
   \begin{center}
      \includegraphics[scale=1.5]{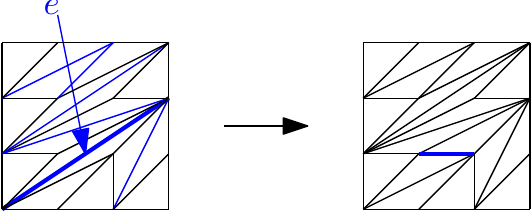}
   \end{center}\vspace{-.5cm}
   \caption{A flip of edge $e$ in a $3\times 3$ triangulation. In the triangulation on the left, the flippable edges are marked in blue, while unflippable edges are in black.}
   \label{fig:flips}
\end{figure}

The graph on $\Omega_n$ induced by the edge-flipping operation above is usually referred to as the \emph{flip graph}, and is 
known to be connected~\cite{Lawson}. In addition, since the transition matrix is symmetric and aperiodic, 
this Markov chain converges to the uniform distribution on $\Omega_n$. 
Very little is currently known regarding the dynamic properties of this Markov chain, 
in particular no non-trivial bound on its mixing time 
(the time until the Markov chain is close enough to its stationary distribution) is known.

In~\cite{CMSS15} we introduced a real parameter $\lambda>0$ and considered \emph{weighted} triangulations: 
each triangulation $\sigma\in\Omega_n$ has weight $\lambda^{\sum_{x\in\Lambda_n}|\sigma_x|}$, where
$|\sigma_x|$ denotes the $\ell_1$ norm of the edge $\sigma_x$. 
Adapting the Markov chain above using the so-called \emph{heat-bath dynamics} gives a 
Markov chain whose stationary distribution is proportional to the weights. 
Simulation suggests that this Markov chain has intriguing behavior, undergoing a phase transition at $\lambda=1$; see Figure~\ref{fig:weights}. 
\begin{figure}[htbp]
   \begin{center}
      \hspace{\stretch{1}}
      \includegraphics[scale=.1]{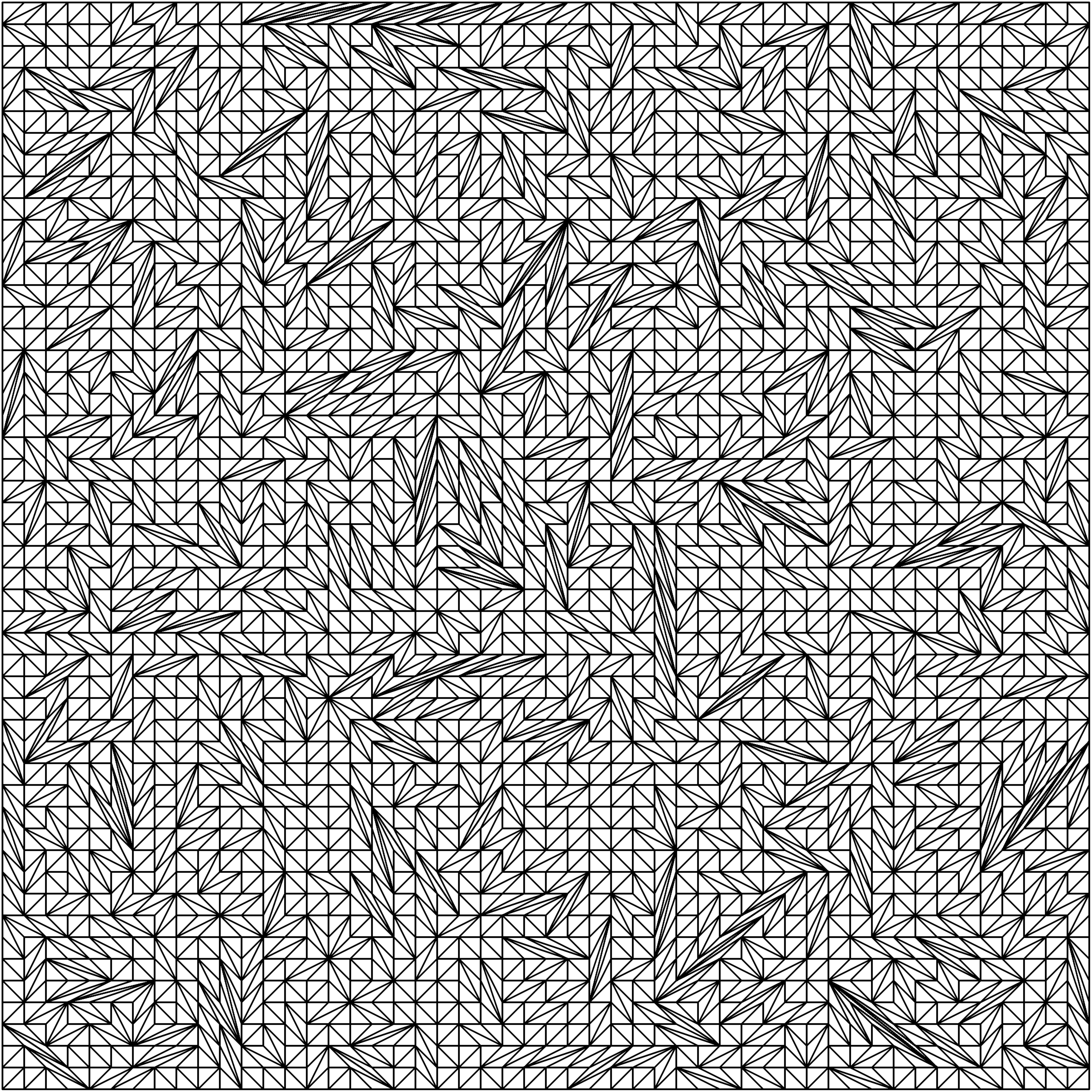}
      \hspace{\stretch{1}}
      \includegraphics[scale=.1]{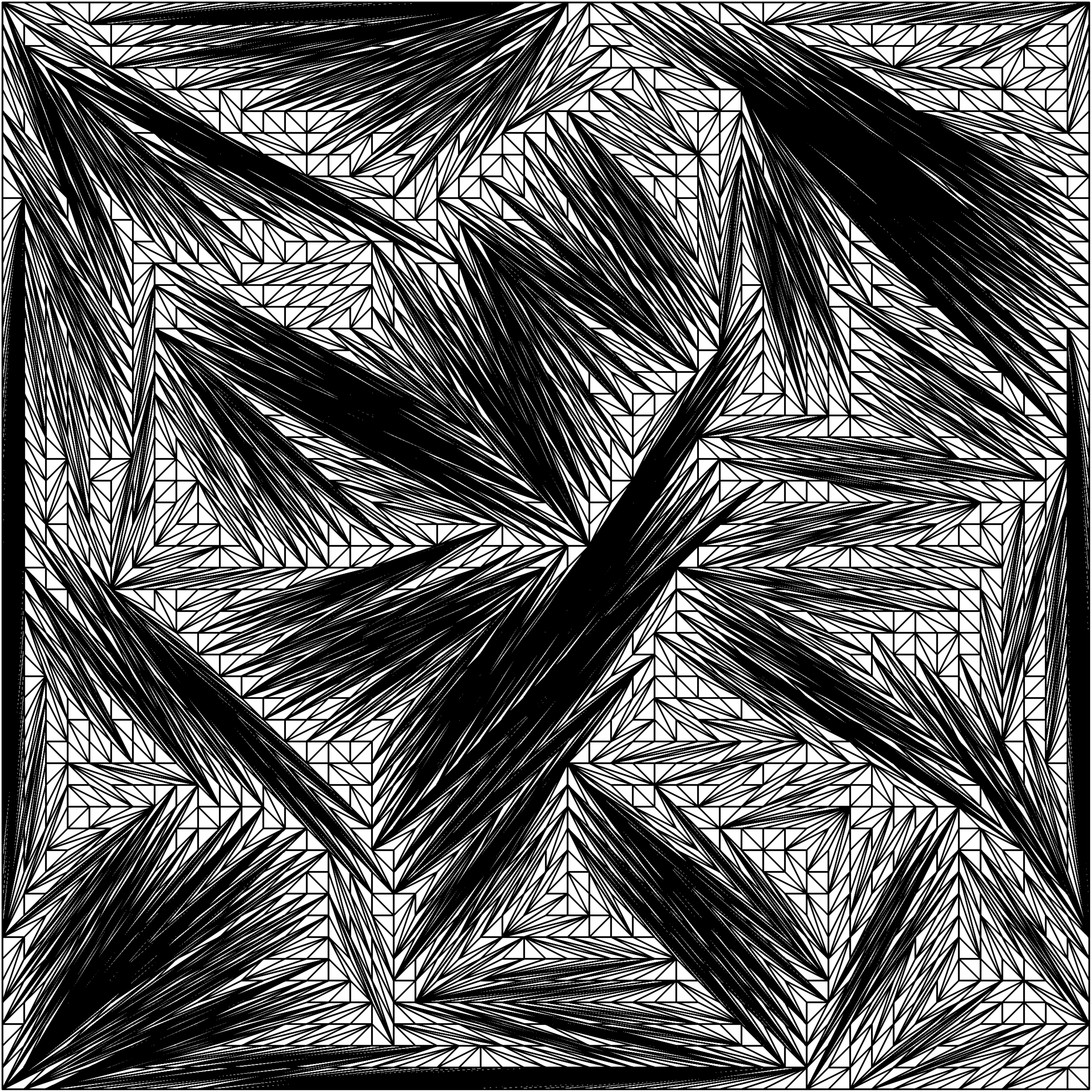}
      \hspace{\stretch{1}}~
   \end{center}\vspace{-.5cm}
   \caption{$50\times 50$ lattice triangulations produced by the edge-flipping Markov chain 
      with $\lambda=0.9$ (left) and $\lambda=1.1$ (right).
      The triangulation in Figure~\ref{fig:triang} was obtained by the edge-flipping Markov chain with $\lambda=1$.}
   \label{fig:weights}
\end{figure}
It is believed that for any $\lambda<1$, which we call the \emph{subcritical regime}, 
regions far from one another in the triangulation evolve roughly independently. 
This suggests the presence of decay of correlations and small mixing time. 
On the other hand, for $\lambda>1$, which we call the \emph{supercritical regime}, 
the Markov chain faces a rigidity phenomenon: long edges give rise to rigid regions of aligned edges. 
This suggests the presence of ``bottlenecks'' in the Markov chain, giving rise to exponential mixing time.
Finally, in the $\lambda=1$ case, which is the case of uniformly random triangulations, relatively long edges appear 
but simulation suggests that the regions of aligned edges are not as rigid as in the supercritical regime.

\paragraph{Our contribution.}
In this paper we construct a \emph{height function} on lattice triangulations: 
given a triangulation $\sigma\in\Omega_n$, the function attributes 
a positive real value to each midpoint in $\Lambda_n$. 
Our main result (Theorem~\ref{thm:lyapunov}) establishes that this function behaves as a \emph{Lyapunov function} with respect to Glauber dynamics.
This means that the value of the function at any midpoint $x\in\Lambda_n$ behaves like a supermartingale. 
Theorem~\ref{thm:lyapunov} holds for any $\lambda\in(0,1)$, and gives the first result on the dynamics of lattice triangulations that are valid in the whole
subcritical regime.
A crucial feature of Theorem~\ref{thm:lyapunov} is that it holds in great generality:
also in the case of triangulations of general lattice polygons\footnote{The set of vertices $\Lambda_n^0$ does not need to be the $n\times n$ square, but 
can be the set of integer points inside any, even non-convex, lattice polygon (a polygon whose vertices are points of $\mathbb{Z}^2$).} 
and in the presence of arbitrary constraint edge\footnote{Triangulations where some given set of edges are forced to be present, 
see Section~\ref{sec:results} for precise definitions.}.

The definition of the height function is a bit involved, so we defer it, as well as the statement of our main result 
(Theorem~\ref{thm:lyapunov}), to Section~\ref{sec:results}.
In particular, the height function is defined in terms of a novel type of geometric crossings, and uses a new partition on the edges of a triangulation in terms
of what we call regions of influence. We believe these two new concepts 
(which we introduce and analyze in Sections~~\ref{sec:treeinfluence} and~\ref{sec:groundstatecrossings}) are of independent interest.
The proof of Theorem~\ref{thm:lyapunov} is given in Section~\ref{sec:proof}.

Our main result has a large range of consequences and applications, which we discuss in Sections~\ref{sec:consequence}--\ref{sec:thinrectangle}. 
For example, we apply it to establish
that the length of an edge of a triangulation has an exponential tail (Corollary~\ref{cor:tailedge}), 
that local measures are tight (Theorem~\ref{thm:tightness}), 
and that there are crossings of triangles of constant size (Theorem~\ref{thm:smalltriangles}).
In the particular case 
of triangulations of $n\times k$ rectangles, where $k$ is a fixed integer independent of $n$, 
our technique yields the existence of local limits (Theorem~\ref{thm:locallimits}), 
decay of correlations (Theorem~\ref{thm:correlations}), and 
recurrence of random walks on the induced graph (Corollary~\ref{cor:recurrence}). 
In a very recent work with Caputo, Martinelli and Sinclair~\cite{CMSS}, 
we apply Theorem~\ref{thm:lyapunov}, as well as a number of other results from this paper, to 
establish tight bounds on the mixing time of $n\times k$ triangulations. 
%These are the first results on lattice triangulations that hold in the whole subcritical regime.

\paragraph{Motivation and previous works.}
Lattice triangulations have appeared in a broad range of contexts. 
A number of beautiful combinatorial arguments have been recently developed to estimate the number of lattice triangulations~\cite{Anclin,KZ,Welzl}.
For example, a very elegant argument by Anclin~\cite{Anclin} shows that the cardinality of $\Omega_n$ is at most $8^{n^2+2n}$. 
Despite not being the best known upper bound, 
Anclin's argument is quite general and applies to lattice triangulations of general lattice polygons; we state 
and use it 
later, see Lemma~\ref{lem:anclin}. The best known bounds on $|\Omega_n|$ are 
$|\Omega_n|\geq 4.13^{n^2}$~\cite{KZ} and $|\Omega_n|\leq c\, 6.86^{n^2}$~\cite{Welzl} for some positive constant $c$.

Lattice triangulations have also been studied in other areas. For example, in algebraic geometry,
they play a key role in the famous construction of 
plane algebraic curves by Viro~\cite{Viro}, which has connections with Hilbert's Sixteenth problem, and 
also appeared in 
the theory of discriminants~\cite{GKZ} and toric varieties~\cite{Dais}. 
Lattice triangulations have also been studied in the contexts of
discretization of random surface models~\cite{Frohlich} and 
two-dimensional quantum gravity~\cite{MYZ}. 
Several other applications of triangulations are discussed in~\cite{tr_book}.

Much less is known about \emph{random} triangulations.
The only result on the mixing time of the above edge-flipping Markov chain is~\cite{CMSS15}.
There we showed that, for any $\lambda>1$, the mixing time is at least $e^{cn}$ for some constant $c>0$. We also showed that, for all \emph{sufficiently small} $\lambda$, 
the mixing time is of order $n^3$, and a random triangulation has decay of correlations. 
Extending these results to the whole subcritical regime turned out to be quite challenging, 
especially since similar results for other spin systems make use of fundamental properties that do not hold on lattice triangulations.
This led us to look for new geometric properties of lattice triangulations and to develop new techniques. 
%The case $\lambda=1$, in turn, is still completely open.

Concurrently to~\cite{CMSS15}, a similar model of random lattice triangulations has independently appeared in the statistical
physics literature~\cite{KKM,KSM}. 
Following~\cite{CMSS15}, a similar model has been introduced to study the mixing time of random rectangular dissections and 
dyadic tilings~\cite{CMR}.

%############################################################################################
%############################################################################################
%############################################################################################
\section{Notation and statement of main result}\label{sec:results}
Unless stated otherwise, henceforth we let $\Lambda^0$ be any subset of $\mathbb{Z}^2$ such that 
$\Lambda^0$ contains all points of $\mathbb{Z}^2$ that lie inside some lattice polygon, including the vertices of the polygon.
A lattice polygon is defined as a polygon whose vertex set only contains points of $\mathbb{Z}^2$, 
and whose edges do not intersect one another (aside from their endpoints) and do not contain points of $\mathbb{Z}^2$ in their interior, 
see Figure~\ref{fig:midpoint}(a).
(In a first reading the reader can consider $\Lambda^0=[0,n]^2\cap\mathbb{Z}^2$ as mentioned in Section~\ref{sec:intro}.)
Let $\Omega$ be the set of triangulations of $\Lambda^0$, 
and $\Lambda$ be the set of midpoints of edges of a triangulation of $\Lambda^0$.
\begin{figure}[htbp]
   \begin{center}
      \includegraphics[scale=.6]{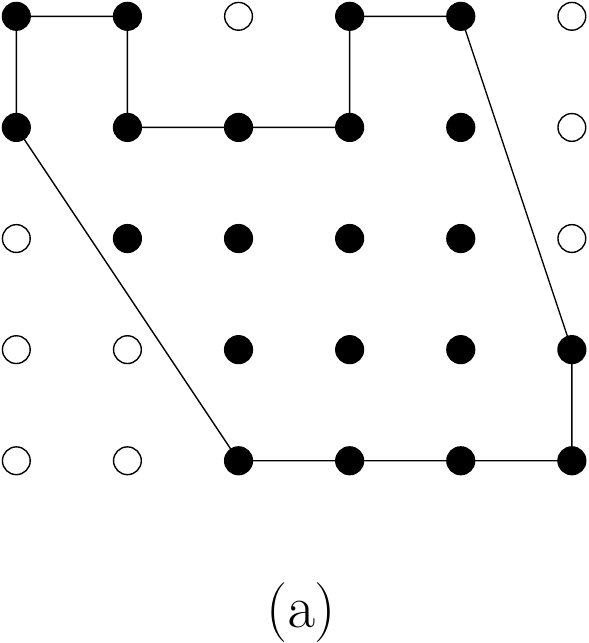}
      \hspace{\stretch{1}}
      \includegraphics[scale=.6]{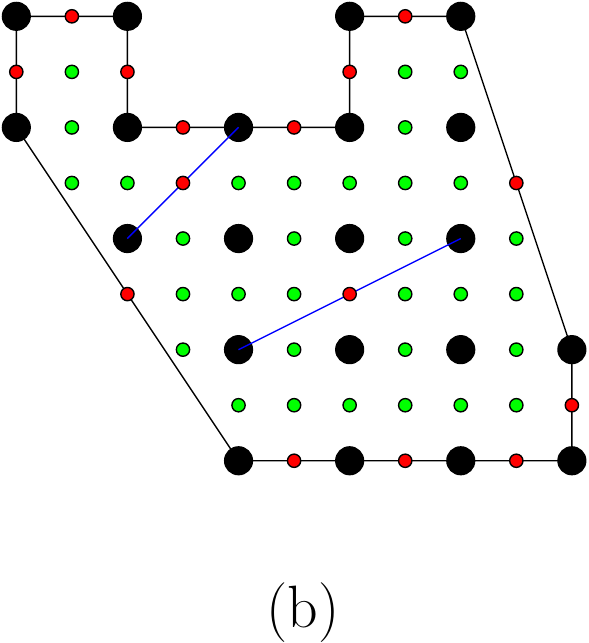}
      \hspace{\stretch{1}}
      \includegraphics[scale=.6]{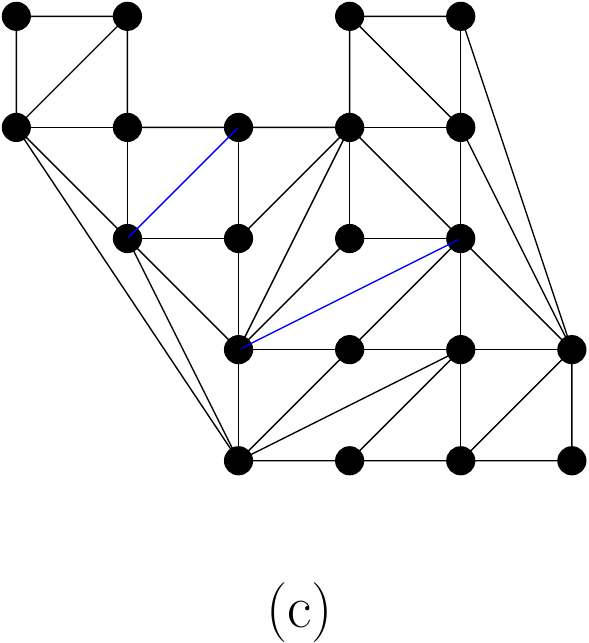}
      \hspace{\stretch{1}}
      \includegraphics[scale=.6]{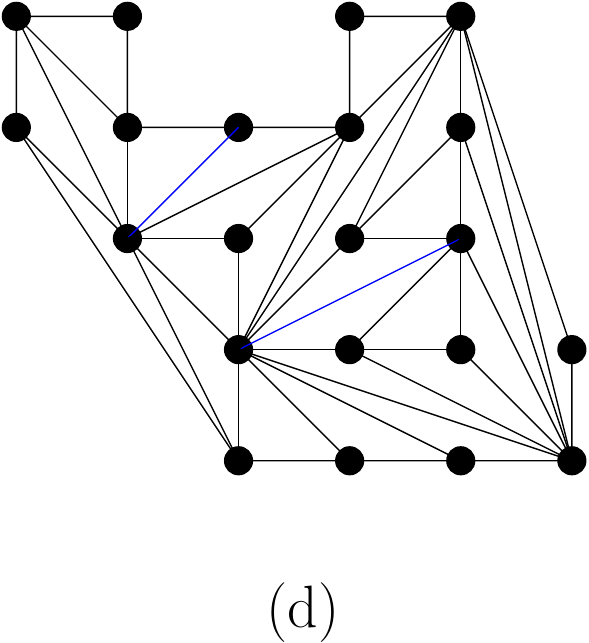}
   \end{center}\vspace{-.5cm}
   \caption{(a) A lattice polygon (black edges): the black points form the induced set $\Lambda^0$ and the white points belong to $\mathbb{Z}^2\setminus \Lambda^0$.
     (b) Two constraint edges (blue edges): the green and red half-integer points form the set of midpoints $\Lambda$, 
        the black and blue edges form the boundary condition $\xi$,
        and the red points form the midpoints $\Lambda^\bc=\xi\cap \Lambda$ of the boundary condition. 
    (c,d) Two triangulations consistent with the boundary condition $\xi$ of part (b).}
   \label{fig:midpoint}
\end{figure}

Now we formally define the notion of \emph{boundary conditions} via constraint edges. 
Consider a collection of edges $\xi=\{\xi_x \colon x\in\Lambda^\bc\}$ for some $\Lambda^\bc\subset\Lambda$ such that each edge of $\xi$ 
has endpoints in
$\Lambda^0$ and, aside from its endpoints, does not intersect other edges of $\xi$ or points in $\Lambda^0$. 
We say that 
% If a triangulation $\sigma\in\Omega$ is such that for all $x\in\Lambda^\bc$ we have $\sigma_x=\xi_x$, we 
% say that $\sigma$ is \emph{compatible} with $\xi$. 
\begin{align*}
   \text{a triangulation $\sigma\in\Omega$ is \emph{compatible} with $\xi$
   if $\sigma_x=\xi_x$ for all $x\in\Lambda^\bc$.}
\end{align*}
see Figure~\ref{fig:midpoint}(b--d).
We refer to $\xi$ as a \emph{boundary condition} and let 
$$
   \Omega^\xi = \{\sigma\in \Omega \colon \sigma \text{ is compatible with } \xi\}.
$$
For convenience, we assume that $\xi$ (resp., $\Lambda^\bc$) 
always contains the boundary edges (resp., the midpoints of the boundary edges) of the lattice polygon induced by $\Lambda^0$; for example, in the case of 
$\Lambda^0=[0,n]^2\cap\mathbb{Z}^2$, we have that $\xi$ contains the $4n$ edges of length $1$ connecting consecutive points 
on the boundary of $[0,n]^2$. See Figure~\ref{fig:midpoint}(b) for another example.
Define 
$$
   \text{$\Xi(\Lambda^0)$ to be the collection of all possible sets of constraint edges with endpoints $\Lambda^0$.}
$$
Consequently, for any boundary conditions $\xi\in \Xi(\Lambda^0)$, $\xi$ contains the boundary of the aforementioned lattice polygon induced by $\Lambda^0$.
Henceforth, for any $\xi\in\Xi(\Lambda^0)$, we denote by $\Lambda^\bc=\Lambda^\bc(\xi)=\Lambda \cap \xi$ the set of midpoints of the edges in $\xi$.

Given any $\lambda>0$, any set of constraint edges $\xi\in\Xi(\Lambda^0)$ and any $\sigma\in\Omega^\xi$, let 
\begin{align*}
   \mathcal{M}_\sigma^\lambda(\Omega^\xi) \text{ be the edge-flipping Markov chain obtained by Glauber dynamics on $\Omega^\xi$}\\
   \text{with parameter $\lambda$ and starting state $\sigma$.}
\end{align*}
(When the starting state is not important, we will simply denote the above Markov chain by $\mathcal{M}^\lambda(\Omega^\xi)$.)
Given a initial triangulation $\sigma\in\Omega^\xi$, the Markov chain $\mathcal{M}_\sigma^\lambda(\Omega^\xi)$ evolves as follows. 
Pick a uniformly random midpoint $x\in\Lambda$.
If $\sigma_x$ is a constraint edge (i.e., $\sigma_x\in\xi$) or $\sigma_x$ is unflippable, do nothing.
Otherwise, let $\sigma^x$ denote the triangulation obtained by flipping $\sigma_x$ in $\sigma$.
Then with probability $\frac{\lambda^{|\sigma_x^x|}}{\lambda^{|\sigma_x^x|}+\lambda^{|\sigma_x|}}$, flip $\sigma_x$ in $\sigma$, otherwise do nothing.
For any edge $e$, we denote by $|e|$ the $\ell_1$ length of $e$.

The stationary measure of $\mathcal{M}^\lambda(\Omega^\xi)$ is denoted by $\pi^\xi$, and is given by
\begin{equation}
   \pi^\xi(\sigma) = \lambda^{\sum_{x\in\Lambda}|\sigma_x|}/Z^\xi(\lambda),
   \label{eq:defpi}
\end{equation}
where
$
   Z^\xi(\lambda) = \sum\nolimits_{\sigma\in\Omega^\xi}\lambda^{\sum_{x\in\Lambda}|\sigma_x|}
   \text{ is a normalizing constant.}
$
We omit the dependence on $\lambda$ from $\pi^\xi$ to simplify the notation.

Given any midpoint $x\in\Lambda$, define
$E_x^\xi$ as the set of edges of midpoint $x$ that are compatible with $\xi$. In symbols, 
\begin{equation}
   E_x^\xi = \{\sigma_x \colon \sigma \in \Omega^\xi\}.
   \label{eq:ex}
\end{equation}
Despite not being the best known upper bound on the number of triangulations, we mention the following upper bound due to Anclin 
as it holds for arbitrary 
boundary conditions $\xi$.
Anclin showed that 
if we order the midpoints in $\Lambda\setminus \Lambda^\bc$ from top to bottom and left to right, and we construct the 
triangulation by sampling edges one by one following this order, then for each midpoint $x\in\Lambda\setminus\Lambda^\bc$ there are at most two edges of
$E_x^\xi$ that are compatible with all previously sampled edges. This immediately implies the following upper bound.
\begin{lemma}[{Anclin's bound,~\cite{Anclin}}]\label{lem:anclin}
   Given any set of constraint edges $\xi\in\Xi(\Lambda^0)$, we have 
   $$
      |\Omega^\xi|\leq 2^{|\Lambda\setminus\Lambda^\bc|}.
   $$
\end{lemma}

We refer to the edges of $E_x^\xi$ of smallest length as the \emph{ground state} edges of $x$ given $\xi$. 
The ground state edges of $x$ are either composed of a single edge or are the two opposite unit diagonals (i.e., the diagonals of a square of side length 1).
Let 
\begin{equation}
   \mathbb{G}^\xi = \bigcup_{x\in\Lambda}\{g \in E_x^\xi \colon g \text{ is a ground state edge of $x$ given $\xi$}\}
   \label{eq:defg}
\end{equation}
be the set of ground state edges given $\xi$. Also, define the set of all possible edges as 
$$
   E^\xi = \bigcup_{x\in\Lambda} E_x^\xi.
$$
We consider that the edges in $E^\xi$ are \emph{open} line segments. Hence, two edges $e,f\in E^\xi$ 
that intersect only at their endpoints
are considered to be disjoint. 
The \emph{ground state triangulation} is defined as the triangulation with smallest total edge length. 
The following lemma from~\cite{CMSS15} gives that a ground state triangulation can be easily constructed 
by independently adding the smallest edge of each midpoint that
is compatible with the boundary condition.
\begin{lemma}[{Ground State Lemma, from~\cite[Lemma~3.4]{CMSS15}}]\label{lem:groundstate}
   Given any boundary condition $\xi\in\Xi(\Lambda^0)$, 
   the ground state triangulation given $\xi$ 
   is unique (up to possible flips of unit diagonals), and can be constructed by placing each edge in its minimal length configuration consistent
   with $\xi$, independent of the other edges.
\end{lemma}

The flip operation induces a natural partial order on $E_x^\xi$. 
It is known that
for any non-ground-state edge $f\in E_x^\xi\setminus \mathbb{G}^\xi$ there is a unique edge $e\in E_x^\xi$ such that $e$ can be obtained from $f$ via a 
decreasing flip; see, for example,~\cite[Section 2.2]{CMSS15}. 
In this case we say that $e$ is the \emph{parent} of $f$.
When $f$ belongs to a triangulation where $f$ can be flipped to a shorter edge (which necessarily is $e$), 
in this triangulation $f$ is the largest diagonal of a parallelogram, which is referred to as the \emph{minimal parallelogram of $f$}.
Then for two distinct edges $e,f\in E_x^\xi$, we say that 
\begin{align}
   \text{$e\prec f$ iff there is a sequence $e=e_0,e_1,e_2,\ldots,e_k=f\in E_x^\xi$ such that} \nonumber\\
   \text{$e_i$ is the parent of $e_{i+1}$ for all $i=0,1,2,\ldots,k-1$}. 
   \label{eq:edgeposet}
\end{align}
In other words, if $e \prec f$, 
there is a length-increasing sequence of edges $e=e_1,e_2,\ldots,e_k=f$ 
with $e_i\in E_x^\xi$ for all $i$, 
and such that for each $i=1,2,\ldots,k-1$ 
there exists two triangulations $\sigma,\eta$ adjacent in the flip graph 
satisfying $\sigma_y=\eta_y$ for all $y\in\Lambda\setminus\{x\}$, 
$\sigma_x=e_i$ and 
$\eta_x=e_{i+1}$. 
Hence, the ground state edges of midpoint $x$ are the edges $g\in E_x^\xi$ such that 
there exists no $e\in E_x^\xi\setminus \{g\}$ with $e\prec g$.
We say that $e\preceq f$ if either $e=f$ or $e\prec f$.

Given a boundary condition $\xi\in\Xi(\Lambda^0)$, 
a triangulation $\sigma$, a midpoint $x\in\Lambda$ and a ground state edge $g\in\mathbb{G}^\xi$ (whose midpoint is not necessarily $x$), define the set 
\begin{equation}
   E_x^\xi(\sigma,g) = \{e \in E_x^\xi \colon e\cap g \neq \emptyset \text{ and } e \preceq \sigma_x\}.
   \label{eq:defe}
\end{equation}
We will show later in Proposition~\ref{pro:monotonicity} that $E_x^\xi(\sigma,g)=\emptyset$ if and only if $\sigma_x$ does not intersect $g$.
For any $e\in E^\xi$, let $|e|$ denote the $\ell_1$ length of $e$.
Given a parameter $\alpha>1$, 
define the function $\Psi^\xi_g \colon \Omega^\xi \to \mathbb{R}$ as 
\begin{equation}
   \Psi_g^\xi(\sigma) = \sum_{x\in\Lambda}\sum_{e \in E_x(\sigma,g)}\alpha^{|e|-|g|}.
   \label{eq:lyapunov}
\end{equation}
Note that, for any $g\in\mathbb{G}^\xi$ and $\sigma\in\Omega^\xi$, 
letting $x\in\Lambda$ be the midpoint of $g$, we obtain 
$$
   \Psi_g^\xi(\sigma)\geq \alpha^{|\sigma_x|-|g|}\geq 1.
$$
If $\sigma$ is the ground state triangulation, then $\Psi_g^\xi(\sigma)=1$ for all $g\in\mathbb{G}^\xi$.
We can regard $\Psi^\xi$ as a height function for triangulations: given any triangulation $\sigma\in\Omega^\xi$, 
$\Psi^\xi_g(\sigma)$ can be seen as a height value to the midpoint of $g$.

The theorem below shows that there are values of $\alpha$ for which $\Psi_g^\xi$ is a \emph{Lyapunov function}.
For this reason, in many parts of the paper we will refer to $\Psi^\xi_g$ as the \emph{Lyapunov function}.
Let $\PR_\sigma=\PR_\sigma^\xi$ denote the probability measure induced by $\mathcal{M}_\sigma^\lambda(\Omega^\xi)$, and let 
$\E_\sigma=\E_\sigma^\xi$ be the corresponding expectation.
\begin{theorem}\label{thm:lyapunov}
   For any $\lambda\in(0,1)$, there exists $\alpha\in (1,\lambda^{-1/2})$, $\psi_0>1$ and $\epsilon>0$, each depending only on $\lambda$, 
   for which the following holds. 
   Let $\xi\in\Xi(\Lambda^0)$ be any boundary condition, 
   $\sigma\in\Omega^\xi$ be any triangulation, 
   and $\sigma'$ be a random triangulation
   obtained by applying one step of $\mathcal{M}_\sigma^\lambda(\Omega^\xi)$. 
   For any $g\in\mathbb{G}^\xi$, if $\Psi_g^\xi(\sigma)\geq \psi_0$, then
   $$
      \E_\sigma\big(\Psi_g^\xi(\sigma')\big) \leq \left(1-\frac{\epsilon}{|\Lambda|}\right)\Psi_g^\xi(\sigma).
   $$
\end{theorem}

% \begin{remark}
%    The power of the above result comes from two characteristics of $\Psi_g^\xi$. First, the theorem above holds for \emph{any} boundary condition 
%    $\xi$. Second, for any $g$, $\Psi_g^\xi$ depends only on the ``local'' configuration of the triangulation around the midpoint of $g$ and establishes a 
%    certain measure between this local configuration and the ground state configuration given $\xi$.
% \end{remark}
%############################################################################################
%############################################################################################
%############################################################################################
\section{Partition of triangulations and trees of influence}\label{sec:treeinfluence}
Fix any boundary condition $\xi\in\Xi(\Lambda^0)$.
Given a triangulation $\sigma\in\Omega^\xi$ and a midpoint $x\in\Lambda$, 
we say that $\sigma_x$ is \emph{increasing} if it is a flippable edge of $\sigma$ and after flipping $\sigma_x$ we 
obtain a (strictly) larger edge. We could define \emph{decreasing} edges in a similar way, however for technical reasons we need 
to include some constraint edges in the set of decreasing edges, namely the constraint edges which would be flippable and decreasing if they were not in $\xi$.
We do this by calling $\sigma_x$ \emph{decreasing} if it is \emph{not} a unit diagonal 
and it is the largest edge of all triangles of $\sigma$ containing $\sigma_x$. 
Note that if $\sigma_x$ is decreasing according to the above definition and $\sigma_x\not\in\xi$, then 
$\sigma_x$ is flippable and after flipping $\sigma_x$ we obtain a (strictly) smaller edge. 
For any $\ell>\mathbb{R}_+$ and triangulation $\sigma$, define the following subsets of $\Lambda$:
\begin{align}
  F_{\ell}(\sigma) &= \big\{x \in\Lambda \colon |\sigma_x|\leq \ell\big\} \nonumber\\
  \finc(\sigma) &= \big\{x \in\Lambda \colon\text{$\sigma_x$ is an increasing edge}\big\} \nonumber\\
  \fdec(\sigma) &= \big\{x \in\Lambda \colon\text{$\sigma_x$ is a decreasing edge}\big\}\nonumber\\
  \fdiag(\sigma) &= \big\{x \in\Lambda \colon\text{$\sigma_x$ is a unit diagonal and the largest edge of all triangles of $\sigma$ containing $\sigma_x$}\big\}.
  \label{eq:deff}
\end{align}
Note that for any triangulation $\sigma\in\Omega^\xi$ 
and midpoint $x\in \fdiag(\sigma)$, we have that $\sigma_x$ is flippable but does not change its length after being flipped.

Given a triangulation $\sigma\in\Omega^\xi$, we define a collection of trees whose vertices are elements of $\Lambda$.
Each tree is rooted at a midpoint in $\fdec(\sigma)\cup \fdiag(\sigma)$, and 
there will be two trees for each $x\in \fdec(\sigma)\cup \fdiag(\sigma)$. We denote these trees by
$\tau^{(1)}(\sigma,x)$ and $\tau^{(2)}(\sigma,x)$. 
To define $\tau^{(1)}(\sigma,x)$ take one of the triangles of $\sigma$ containing $\sigma_x$. Denote this triangle by $\Delta$. 
The tree $\tau^{(2)}(\sigma,x)$ will
be defined analogously by considering the other triangle of $\sigma$ containing $\sigma_x$. 
The root of $\tau^{(1)}(\sigma,x)$ is $x$.
The children of $x$ in $\tau^{(1)}(\sigma,x)$ are the midpoints of the other two edges of $\Delta$.
Then we proceed inductively. The children of a midpoint $y$ with parent $z$ in $\tau^{(1)}(\sigma,x)$
are obtained by considering the triangle $\Delta'$ of $\sigma$ containing $\sigma_y$ but not containing $\sigma_z$. 
If $\sigma_y$ is not the largest edge of $\Delta'$, then $y$ has no child
in $\tau^{(1)}(\sigma,x)$; otherwise the children of $y$ are the midpoints 
of the other edges of $\Delta'$ (see Figure~\ref{fig:treeinf} for a reference).
Note that, for any two midpoints $y,z$ with $y$ being a child of $z$ in
$\tau^{(1)}(\sigma,x)$, we have that $|\sigma_z|>|\sigma_y|$. This guarantees that the construction above ends. 
Define $\tau(\sigma,x)$ as a tree rooted at $x$ obtained by the union of $\tau^{(1)}(\sigma,x)$ and $\tau^{(2)}(\sigma,x)$. 
We call $\tau(\sigma,x)$ the \emph{tree of influence} of $x$.
\begin{figure}[htbp]
   \begin{center}
      \includegraphics[scale=.7]{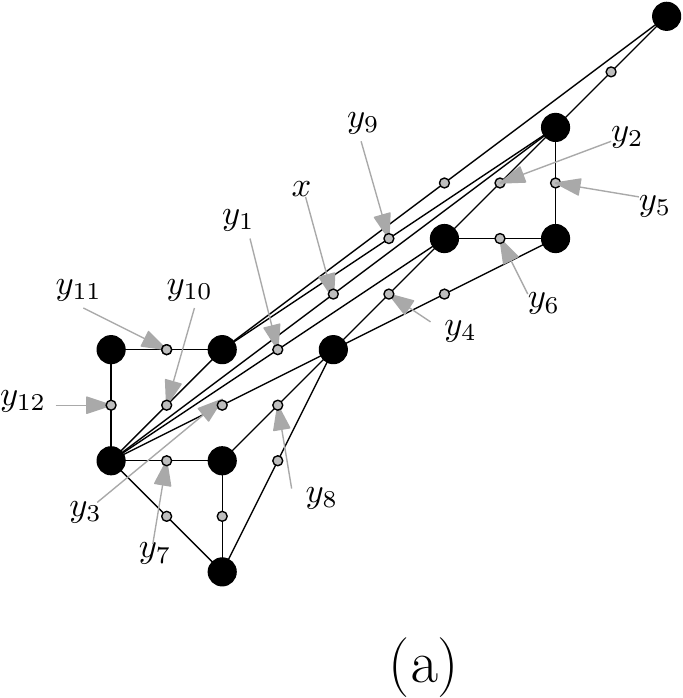}
      \hspace{\stretch{1}}
      \includegraphics[scale=.7]{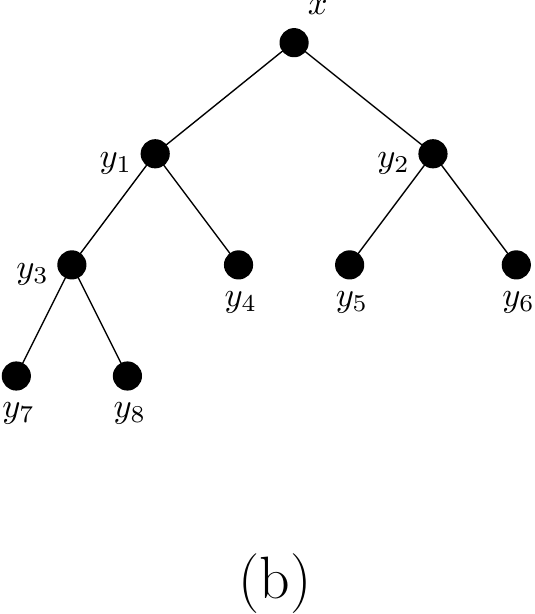}
      \hspace{\stretch{1}}
      \includegraphics[scale=.7]{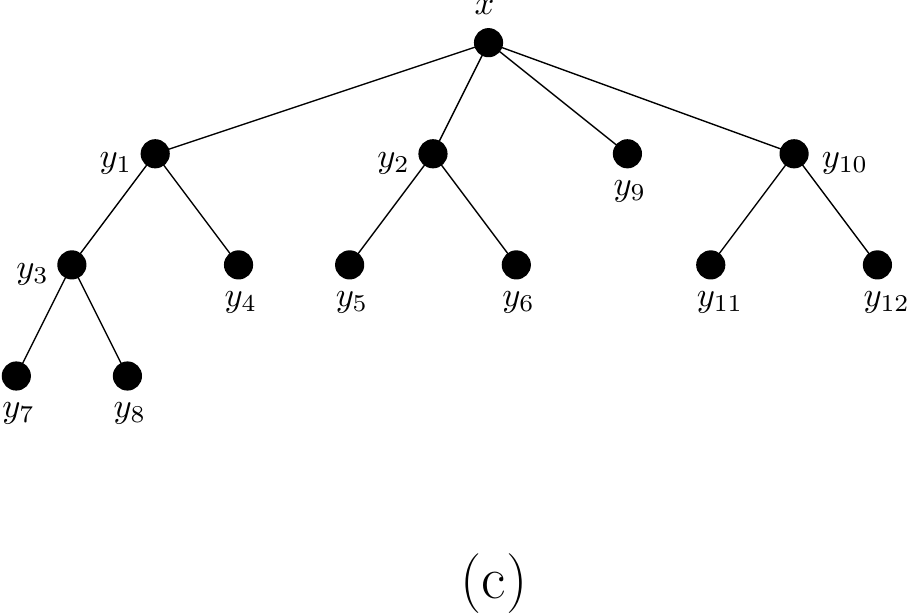}
   \end{center}\vspace{-.5cm}
   \caption{(a) A triangulation $\sigma$, with midpoints illustrated by gray points. 
      (b) The tree $\tau^{(1)}(\sigma,x)$ constructed from the triangle $\sigma_x,\sigma_{y_1},\sigma_{y_2}$. 
      (c) The tree $\tau(\sigma,x)$.}
   \label{fig:treeinf}
\end{figure}

Although we used the term tree, it is not explicit from the construction above that $\tau^{(1)}(\sigma,x)$, $\tau^{(2)}(\sigma,x)$ and $\tau(\sigma,x)$ are actually trees. 
However, if we orient the edges from parents to children, since parents are associated to strictly larger edges than their children,
the construction above is at least guaranteed to produce a directed acyclic graph.
But we have not ruled out the case that a midpoint $y$ is reached from two distinct paths from $x$ (i.e., some vertices may have two parents). 
Proposition~\ref{pro:tautree} below shows that this does not happen, hence the construction described above indeed produces trees.

Given two midpoints $y,z\in\tau(\sigma,x)$, 
we will use standard terminology to say that $y$ is an \emph{ancestor} (resp., \emph{descendant})
of $z$ in $\tau(\sigma,x)$ if there exists a directed path in $\tau(\sigma,x)$ from $y$ to $z$ (resp., from $z$ to $y$)
using the orientation of edges described above.
We will need one more definition. 
Partition $\mathbb{Z}^2$ into $1\times 1$ squares whose edges are parallel to the axes (i.e., the faces of the square lattice). 
Let $\mathbb{S}$ be the set of these $1\times 1$ squares.
Given any edge $e\in E^\xi$, let 
\begin{equation}
   S(e) = \{Q \in \mathbb{S} \colon \text{the interior of $Q$ intersects $e$}\}.
   \label{eq:defs}
\end{equation}

\begin{proposition}\label{pro:tautree}
   Consider any boundary condition $\xi\in\Xi(\Lambda^0)$, any triangulation $\sigma\in\Omega^\xi$, and 
   any $x\in \fdec(\sigma)\cup \fdiag(\sigma)$.
   The following statements hold:
   \begin{enumerate}
      \item\label{it:tree} $\tau(\sigma,x)$ is a tree. 
      \item\label{it:treesquare} For any $y,z\in \tau(\sigma,x)$ with $y$ being an ancestor of $z$, we have $S(\sigma_y)\supset S(\sigma_z)$.
      \item\label{it:leaves} For $i=1,2$, we have 
         $\sum_{y \in \tau^{(i)}_\leaves(\sigma,x)} |\sigma_y|= |\sigma_x|$,
         where $\tau^{(i)}_\leaves(\sigma,x)$ are the set of leaves of $\tau^{(i)}(\sigma,x)$, which are 
         the vertices without children.
   \end{enumerate}   
\end{proposition}
\begin{proof}
%    If $x\in \fdiag$, then $\tau^{(i)}(\sigma,x)$ contains only the root $x$ and two leaves, which are the midpoints of the edges in the same 
%    triangle as $\sigma_x$, which clearly forms a tree. 
%    
%    Now consider the case $x \in \fdec$.
   Consider the set of squares $S(\sigma_x)$. 
	Note that $\sigma_x$ partitions this set into two identical regions, which we denote by $Q_1$ and $Q_2$.
	See Figure~\ref{fig:treeexp}(a) for a reference.
	\begin{figure}[tbp]
	   \begin{center}
	      \includegraphics[scale=.7]{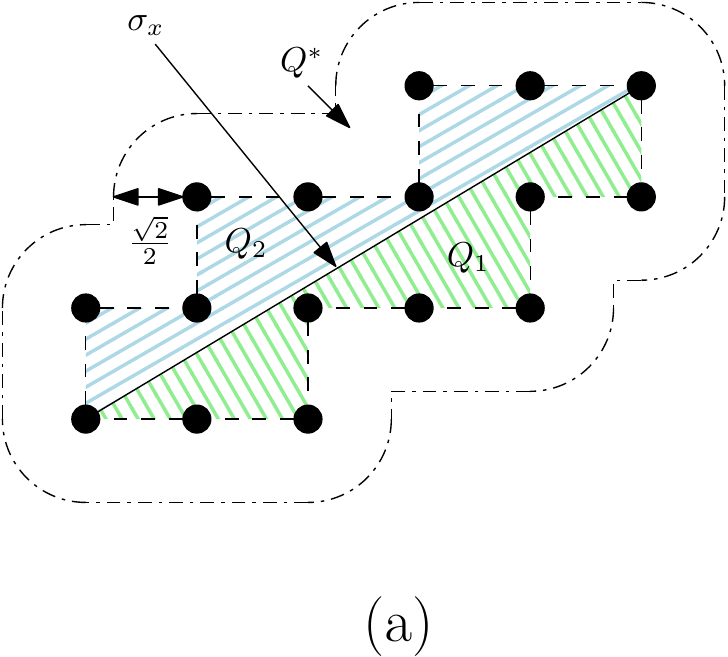}
	      \hspace{\stretch{1}}
	      \includegraphics[scale=.7]{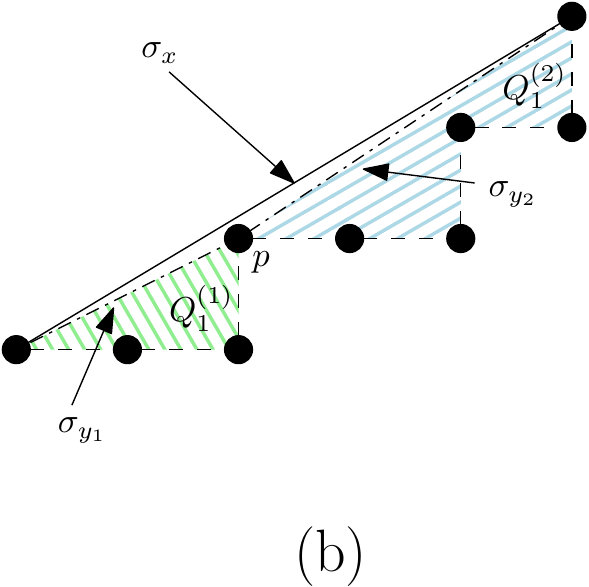}
	      \hspace{\stretch{1}}
	      \includegraphics[scale=.7]{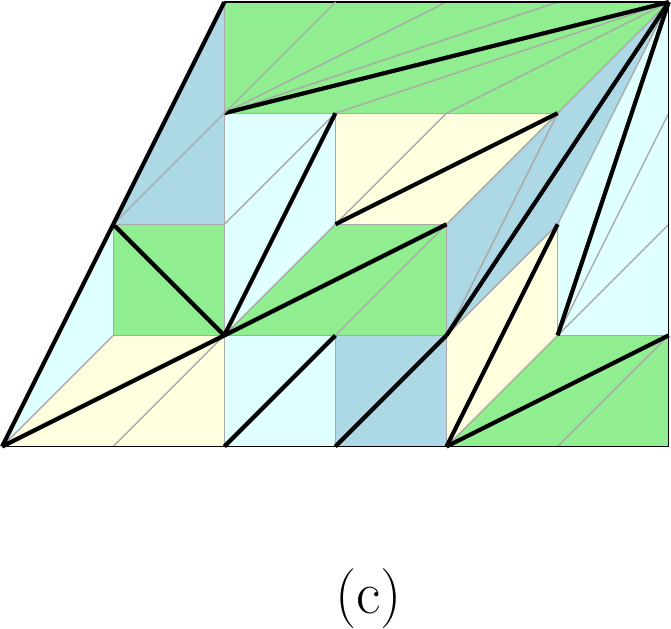}
	   \end{center}\vspace{-.5cm}
	   \caption{(a) The set of squares $S(\sigma_x)=Q_1\cup Q_2$, the two identical regions $Q_1$ and $Q_2$, and the enlarged region $Q^*$.
	            (b) The children $y_1$ and $y_2$ of $x$ decompose $Q_1$ into three disconnected regions: $Q_1^{(1)}$, $Q_1^{(2)}$ and the triangle $(\sigma_x,\sigma_{y_1},\sigma_{y_2})$.
	            (c) A partition (cf.\ Proposition~\ref{pro:partition}) of a triangulation $\sigma$ of the trapezoid into the regions 
	               $\{T(\sigma,x)\colon x\in \fdec(\sigma)\cup\fdiag(\sigma)\}$. 
	               The bold edges represent the edges of midpoint in $\fdec(\sigma)\cup\fdiag(\sigma)$, which are the roots of the trees.}
	   \label{fig:treeexp}
	\end{figure}	
	Note also that $Q_1$ and $Q_2$ are lattice polygons.
	Let $Q^*$ be all points of $\mathbb{R}^2$ within distance $\sqrt{2}/2$ from $Q_1 \cup Q_2$ (including $Q_1 \cup Q_2$).
	We obtain that $Q^*$ contains the same points of $\mathbb{Z}^2$ as $Q_1\cup Q_2$.
	Let $(\sigma_x,\sigma_{y_1}, \sigma_{y_2})$ be one of the triangles containing $\sigma_x$ (say, the one intersecting $Q_1$), 
	and assume that $y_1$ and $y_2$ are the 
	children of $x$ in $\tau^{(1)}(\sigma,x)$.
	We claim that 
	\begin{equation} 
	   \text{all descendants of $x$ in $\tau^{(1)}(\sigma,x)$ are contained
	    in $Q_1$.}
	    \label{eq:treeclaim}
	\end{equation}
	In the discussion below, refer to Figure~\ref{fig:treeexp}(b).
	Since $y_1,y_2$ are children of $x$, we have that $\sigma_{y_1}$ and $\sigma_{y_2}$ have size smaller than $\sigma_x$.
	This implies that the $\ell_2$ length of $\sigma_x$ is at least $\sqrt{2}$.
   Let $p$ denote the vertex of the triangle $(\sigma_x,\sigma_{y_1},\sigma_{y_2})$ that is not an endpoint of $\sigma_x$.
	Since the area of each triangle is equal to $1/2$ and $\sigma_x$ has $\ell_2$ length at least $\sqrt{2}$,
	the distance between $p$ and $\sigma_x$ is at most $\frac{\sqrt{2}}{2}$.
	Thus $p$ must be inside $Q^*$ and, therefore, must be one of the vertices on the boundary of $Q_1$.
	We can use $p$ to partition $Q_1$ into three regions: $Q_1^{(1)}$, $Q_1^{(2)}$ and the triangle $(\sigma_x,\sigma_{y_1},\sigma_{y_2})$. 
	Since the triangle $(\sigma_x,\sigma_{y_1},\sigma_{y_2})$ cannot contain any integer point aside from its three vertices, we have that 
	$\sigma_{y_1}$ and $\sigma_{y_2}$ are entirely contained in $Q_1$.
	Doing this construction inductively for $\sigma_{y_1}$ and $\sigma_{y_2}$, we establish that
	the descendants of $y_1$ are contained in $Q_1^{(1)}$ and the descendants of $y_2$ are contained in $Q_1^{(2)}$, which 
	establishes~\eqref{eq:treeclaim}. Since $Q_1^{(1)}$ and $Q_1^{(2)}$ have disjoint interior, 
	we obtain that $\tau(\sigma,x)$ does not contain any cycle, proving part~\ref{it:tree}.
   This also gives that 
	\begin{equation}
	   |\sigma_{y_1}|+|\sigma_{y_2}|=|\sigma_x|. 
	   \label{eq:sumofchildren}
	\end{equation}
	
	For part~\ref{it:treesquare}, let $S_1$ be the squares of $\mathbb{S}$ whose interior intersects $Q_1^{(1)}$, and let 
	$S_2$ be the squares of $\mathbb{S}$ whose interior intersects $Q_1^{(2)}$. 
	Note that $S_1\cap S_2=\emptyset$ and $S_1\cup S_2=S(\sigma_{y})$. 
	Also, 
	$S(\sigma_{y_1})=S_1$ and $S(\sigma_{y_2})=S_2$. 
	Since the descendants of $y_1$ are contained in $Q_1^{(1)}\subset S_1$, we obtain part~\ref{it:treesquare}. 
	
	For part~\ref{it:leaves}, 
	applying~\eqref{eq:sumofchildren} inductively we have 
	$$
	   \sum_{y\in \tau^{(1)}_\leaves(\sigma,x)} |\sigma_y|
	   = |\sigma_{y_1}|+|\sigma_{y_2}|
	   = |\sigma_x|.
	$$
	The same reasoning applies to $\tau^{(2)}(\sigma,x)$.
\end{proof}

For any triangulation $\sigma\in\Omega^\xi$ and any $x\in\Lambda$, define the set 
$$
   \tau^{-1}(\sigma,x) = \{z\in \fdec(\sigma) \cup \fdiag(\sigma)\colon x\in\tau(\sigma,z)\}.
$$
In words, $\tau^{-1}(\sigma,x)$ is the set of midpoints $z$ such that $x$ is in the tree rooted at $z$.
Note that in any tree containing $x$, the parent of $x$ in the tree is a midpoint $y$ such that $\sigma_y$ is the largest edge in the 
triangle containing both $\sigma_x$ and $\sigma_y$.
\begin{lemma}\label{lem:childparentdefine}
   For any boundary condition $\xi\in\Xi(\Lambda^0)$, any triangulation $\sigma\in\Omega^\xi$ and any $x,y\in \Lambda$ 
   such that $\sigma_x,\sigma_y$ are in the same triangle and $\sigma_y$ is the largest edge of this triangle,
   there exists exactly one tree containing both $x$ and $y$, and $y$ is the parent of $x$ in that tree.
\end{lemma}
\begin{proof}
   We will show that we can construct a path $z_0,z_1,z_2,\ldots$ of adjacent midpoints in $\sigma$ (i.e., midpoints of edges sharing a triangle in $\sigma$) 
   from $z_0=x$ and $z_1=y$ until the root of the tree 
   containing both $x$ and $y$. This path will have the property that $z_i$ is the parent of $z_{i-1}$ in the tree, for all $i$.
   Assume, inductively, that we have defined $z_0=x,z_1=y,z_2,z_3,\ldots,z_i$ with the property that for all $j=1,2,\ldots,i$ we have that 
   $\sigma_{z_j}$ and $\sigma_{z_{j-1}}$ share the same triangle, for which the largest edge is $\sigma_{z_j}$. 
   Let $w$ be the midpoint of the largest edge in the triangle containing $\sigma_{z_i}$ but not $\sigma_{z_{i-1}}$.
   If $w=z_i$ or $z_i$ is contained in only one triangle in $\sigma$ 
   (the later implies that $z_i\in\xi$ as $\sigma_{z_i}$ is at the boundary of the smallest lattice polygon containing $\Lambda^0$), 
   then $\sigma_{z_i}$ is the largest edge in all triangles of $\sigma$ containing $\sigma_{z_i}$ and, consequently, 
   $z_i\in\fdec(\sigma)\cup\fdiag(\sigma)$ is the root of a tree. 
   This gives that 
   $z_i \in \tau^{-1}(\sigma,x)\cap \tau^{-1}(\sigma,y)$.
   Otherwise, let $z_{i+1}=w$, and repeat this procedure. 
   Note that $|\sigma_{z_{i+1}}|>|\sigma_{z_{i}}|$, which implies that this procedure eventually 
   ends, yielding the root of a tree containing $x$ and $y$. 
   It remains to show that this is the unique tree containing $x$ and $y$.
   Since for each $i\geq 1$ in the path $z_0,z_1,\ldots$, we have that $z_i$ is the largest edge 
   in the triangle containing $\sigma_{z_i}$ and $\sigma_{z_{i-1}}$, 
   we obtain that $\sigma_{z_i}$ cannot be a leaf in any tree and 
   the only midpoint that can be a parent of $z_{i}$ in any tree is $z_{i+1}$. This establishes that, for all $i\geq2$, 
   $z_i$ is an ancestor of $y$ in any tree containing $y$, which implies that the root of the tree obtained by the above construction 
   is the root of any tree containing $\sigma_{y}$, completing the proof.
\end{proof}

\begin{proposition}\label{pro:twotrees}
   Given any boundary condition $\xi\in\Xi(\Lambda^0)$, any triangulation $\sigma\in\Omega^\xi$ and any midpoint $x\in\Lambda$, 
   the following holds:
   \begin{enumerate}
      \item\label{it:card} The cardinality of $\tau^{-1}(\sigma,x)$ is either 1 or 2.
      \item\label{it:leaf} If $\tau^{-1}(\sigma,x)=\{y_1,y_2\}$ contains two midpoints, then $x$ is a leaf in both $\tau(\sigma,y_1)$ and $\tau(\sigma,y_2)$.
      \item\label{it:inc} If $x\in \finc(\sigma)$, then $\tau^{-1}(\sigma,x)$ contains two midpoints.
      \item\label{it:large} If $x\in \Lambda$ is such that $\sigma_x$ is the largest edge in some triangle in $\sigma$, 
         then $\tau^{-1}(\sigma,x)$ contains only one midpoint.
   \end{enumerate}
\end{proposition}
\begin{proof}
   Lemma~\ref{lem:childparentdefine} implies~\ref{it:card} since for any $x$ there exists at least one and at most two midpoints $z_1,z_2\in\Lambda$,
   not necessarily distinct from $x$, such that $\sigma_{z_1}$ and $\sigma_{z_2}$ are the longest edges in a triangle containing $\sigma_x$.
   For~\ref{it:leaf}, note that the cardinality of $\tau^{-1}(\sigma,x)$ being two 
   implies that $x$ is not the root of a tree, and there are two midpoints $z_1,z_2$ such that $z_1$ is the parent of $x$ in one tree 
   and $z_2$ is the parent of $x$ in the other tree.
   Therefore, $\sigma_x$ has two distinct parents, one in each tree, implying that $\sigma_x$ cannot be the largest edge in any triangle of $\sigma$;
   hence $x$ cannot be the parent of any midpoint in any tree. This gives that 
   $x$ is a leaf in all trees containing $x$. 
   For~\ref{it:inc}, note that if $x\in\finc(\sigma)$, then there are two distinct midpoints $z_1,z_2\in\Lambda$ such that $\sigma_{z_1}$ and $\sigma_{z_2}$ are the 
   largest edges in triangles containing $\sigma_x$. Therefore, using Lemma~\ref{lem:childparentdefine},
   we have that $z_1$ and $z_2$ are the parents of $x$ in the trees containing $x$, implying that $x$ is contained 
   in two trees.
   For~\ref{it:large}, note that if $\sigma_x,\sigma_y,\sigma_z$ is a triangle such that $\sigma_x$ is the largest edge,
   then there exists at most one midpoint that can be the parent of $x$ in a tree: namely, the midpoint of the largest edge contained in a triangle with $\sigma_x$, 
   if that midpoint exists and is different than $x$. Therefore $x$ can belong to only one tree.
\end{proof}

For each $x\in\fdec(\sigma)\cup\fdiag(\sigma)$, consider the following subset of $\mathbb{Z}^2$:
$$
   T(\sigma,x) = \text{union of all triangles of $\sigma$ containing only edges of midpoint in $\tau(\sigma,x)$}.
$$
\begin{proposition}\label{pro:partition}
   For any boundary condition $\xi\in\Xi(\Lambda^0)$ and any triangulation $\sigma\in\Omega^\xi$, 
   the set $\{T(\sigma,z)\colon z\in \fdec(\sigma)\cup\fdiag(\sigma)\}$ partitions the lattice polygon with vertices in $\Lambda^0$.
\end{proposition}
\begin{proof}
   Proposition~\ref{pro:twotrees}~\ref{it:large} gives that for any triangle $\sigma_x,\sigma_y,\sigma_z$ of $\sigma$, where 
   $\sigma_x$ is the largest edge of this triangle, there exists only one tree containing $x$. Let $\tau(\sigma,w)$ be this tree. We have that 
   $x$ is the parent of both $y$ and $z$ in $\tau(\sigma,w)$, therefore 
   $T(\sigma,w)$ contains the triangle $\sigma_x,\sigma_y,\sigma_z$, and the proof is completed.
\end{proof}

We recall the notion of the \emph{minimal parallelogram} of an edge, which was introduced in~\cite{CMSS15} and appeared briefly in the paragraph preceding~\eqref{eq:edgeposet}.
For any edge $e\in E^\xi$, the \emph{minimal parallelogram} of $e$ is the unique parallelogram composed of two lattice triangles for which 
$e$ is the longest diagonal.
\begin{proposition}\label{pro:smallestedgehierarchy}
   Let $\xi\in\Xi(\Lambda^0)$ be any boundary condition and $\sigma\in\Omega^\xi$ be any triangulation. 
   Let $\Delta_1=(\sigma_{y_1},\sigma_{z_1},\sigma_{w_1})$ and 
   $\Delta_2=(\sigma_{y_2},\sigma_{z_2},\sigma_{w_2})$ be two triangles of $\sigma$ in the same tree $\tau(\sigma,x)$, for some $x\in\Lambda$.
   Assume that 
   $|\sigma_{y_1}|>|\sigma_{z_1}|\geq |\sigma_{w_1}|$, $|\sigma_{y_2}|>|\sigma_{z_2}|\geq |\sigma_{w_2}|$ and 
   $y_1$ is an ancestor of $y_2$ in $\tau(\sigma,x)$. 
   Then, $|\sigma_{w_1}|\geq |\sigma_{w_2}|$. 
\end{proposition}
\begin{proof}
   First consider the case of $y_1$ being the parent of $y_2$ in $\tau(\sigma,x)$, 
   which gives that $y_2\in \{z_1,w_1\}$.
	\begin{figure}[tbp]
	   \begin{center}
	      \hspace{\stretch{1}}
	      \includegraphics[scale=.7]{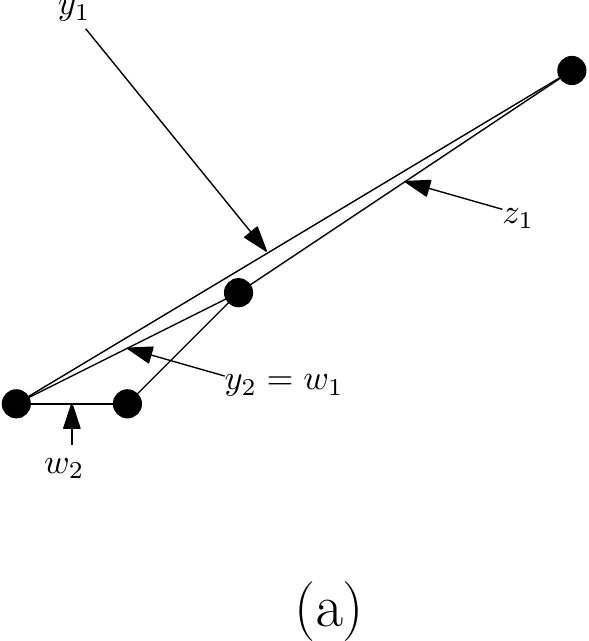}
	      \hspace{\stretch{1}}
	      \includegraphics[scale=.7]{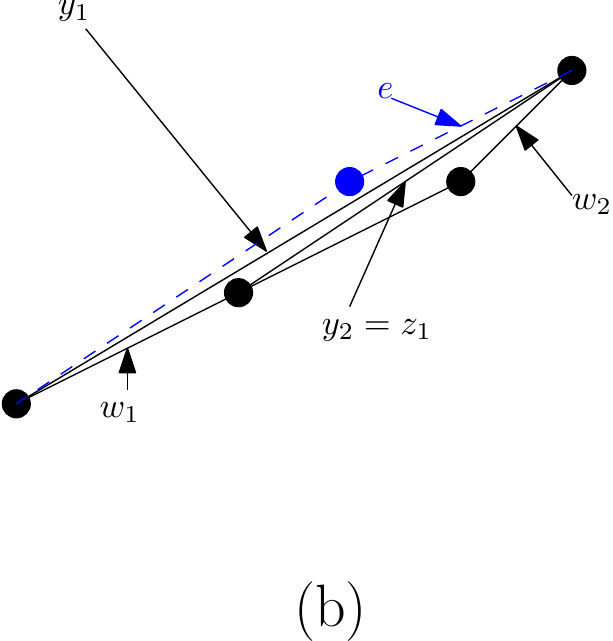}
	      \hspace{\stretch{1}}~
	   \end{center}\vspace{-.5cm}
	   \caption{Illustration for the proof of Proposition~\ref{pro:smallestedgehierarchy} when $y_2=w_1$ (a) and $y_2=z_1$ (b).}
	   \label{fig:smalledge}
	\end{figure}	
   If $y_2=w_1$ (see Figure~\ref{fig:smalledge}(a)), the lemma clearly holds since
   $$
      |\sigma_{w_2}|< |\sigma_{y_2}|=|\sigma_{w_1}|. 
   $$
   If $y_2=z_1$, then we use that 
   $\sigma_{z_1},\sigma_{w_1}$ are part of the minimal parallelogram of $\sigma_{y_1}$. 
   Refer to Figure~\ref{fig:smalledge}(b).
   Let $e$ be the edge opposite to $\sigma_{w_1}$ in the minimal paralellogram of $\sigma_{y_1}$. 
   Note that $e$ may not belong to $\sigma$, and $|e|=|\sigma_{w_1}|$.
   We claim that 
   \begin{equation} 
      \text{$e$ is in the minimal parallelogram of $\sigma_{z_1}$}.
      \label{eq:claimmp}
   \end{equation}
   Using this claim, since $\sigma_{w_2}$ is the smallest edge in the minimal parallelogram of $\sigma_{y_2}=\sigma_{z_1}$, we have
   $$
      |\sigma_{w_2}|\leq |e|=|\sigma_{w_1}|,
   $$
   and the proposition follows when $y_1$ is the parent of $y_2$.
   In the general case of 
   $y_1$ not being the parent of $y_2$, the proposition follows by applying
   the above reasoning inductively along the path from $y_1$ to $y_2$ in the tree $\tau(\sigma,x)$.
   
   It remains to establish~\eqref{eq:claimmp}. If $e$ were an edge of $\sigma$ and $\sigma_{y_1}$ were flippable in $\sigma$ (as illustrated in
   Figure~\ref{fig:smalledge}(b)), 
   then $\sigma_{y_1}$ would be a decreasing edge and, by flipping $\sigma_{y_1}$, we would obtain a triangulation in which 
   $\sigma_{z_1}$ and $e$ are in the same triangle, whose largest edge is 
   $\sigma_{z_1}$. This gives that $e$ is part of the minimal parallelogram of 
   $\sigma_{z_1}$, as claimed. 
\end{proof}

%############################################################################################
%############################################################################################
%############################################################################################
\section{Crossings of ground state edges}\label{sec:groundstatecrossings}
In this section we consider a given edge $g\in\mathbb{G}^\xi$ and establish 
geometric properties of the set of edges of a triangulation $\sigma$ that intersect $g$; recall the definition 
of $\mathbb{G}^\xi$ from~\eqref{eq:defg}. 
In particular, given one edge $\sigma_x$ intersecting $g$, 
one of our main results here gives that the edges of midpoint $\tau^{-1}(\sigma,x)$ 
also intersect $g$. 

% Then given two triangulations $\sigma,\eta$, we say that 
% \begin{equation}
%    \sigma \preceq \eta \text{ if } \sigma_x\preceq \eta_x \text{ for all } x\in\Lambda.
%    \label{eq:tpreceq}
% \end{equation}
% 

We will need the following useful facts from~\cite{CMSS15}.
Fix any boundary condition $\xi\in\Xi(\Lambda^0)$ and any midpoint $x\in\Lambda$.
Two edges $e,f\in E_x^\xi$ are said to be neighbors if we can obtain $e$ from $f$ via a single flip; 
more formally, if there are 
$\sigma,\sigma'\in\Omega^\xi$ such that 
$\sigma_x=e,\sigma'_x=f$ and $\sigma_y=\sigma'_y$ for all $y\neq x$. 
It is known that the graph with vertex set $E_x^\xi$ and the neighborhood relation described above 
is a tree.
This follows since, for any edge $e\in E_x^\xi$, there is at most one $f\in E_x^\xi$ such that 
$e$ and $f$ are neighbors satisfying $|f|< |e|$ (in which case we see $f$ as the parent of $e$ in the tree). 
We consider the ground state edges of $E_x^\xi$ as the root of the tree, and it is possible that the tree has two neighboring roots, which are 
opposite unit diagonals. 
We will call this tree the \emph{tree induced by $E_x^\xi$}.
% We will use the following technical lemma.
% \begin{lemma}[{\cite[Proposition~3.8]{CMSS15}}]\label{lem:distances}
%    Fix a boundary condition $\xi$. 
%    For any midpoint $x\in\Lambda$ and any two triangulations 
%    $\sigma,\sigma'\in\Omega^\xi$, the distance between $\sigma$ and $\sigma'$ in the flip graph is equal to 
%    $\sum_{x\in\Lambda} \kappa(\sigma_x,\sigma'_x)$, 
%    where $\kappa(\sigma_x,\sigma'_x)$ is the distance between $\sigma_x$ and $\sigma'_x$ in the tree induced by $E_x^\xi$. 
% \end{lemma}

Given a boundary condition $\xi\in\Xi(\Lambda^0)$ and a midpoint $x\in\Lambda$, 
we denote by $\bar\sigma_x$ the ground state edge of midpoint $x$ given $\xi$ (with an arbitrary choice among unit diagonals).
Since ground state edges of distinct midpoints are all compatible with one another, 
we have that $\bar \sigma=\{\bar\sigma_x \colon x\in\Lambda\}$ is a ground state triangulation.
In the lemma below we use the partial order on the set $E_x^\xi$, which is defined in~\eqref{eq:edgeposet}, and the set of midpoints of constraint edges
$\Lambda^\bc=\xi\cap\Lambda$.
\begin{proposition}\label{pro:monotoneflips}
   Given any boundary condition $\xi\in\Xi(\Lambda^0)$, any midpoint $x\in\Lambda\setminus \Lambda^\bc$, 
   any two edges $e,f\in E_x^\xi$ such that $e\prec f$, and any triangulation
   $\sigma$ containing $f$, one can obtain a triangulation containing $e$ by performing a sequence of decreasing flips from $\sigma$.
\end{proposition}
\begin{proof}
   Since the graph induced by $E_x^\xi$ is a tree, there is a unique path 
   $f=h_1 \succ h_2 \succ \cdots \succ h_k=e$ in this graph. 
   We claim that there exists a sequence of decreasing flips from $\sigma$ that produce a triangulation containing 
   $h_2$. With this the lemma follows since we can apply this claim repetitively for $h_2,h_3,\ldots$ until we obtain a triangulation 
   containing $e$. 
   
   Now we establish the above claim. If $f$ is decreasing in $\sigma$, the claim follows since we can flip $f$ to obtain $h_2$. 
   From now on assume that $f$ is not decreasing, and let $x$ be the midpoint of $f$. 
   Let $L_x(\sigma)$ be the sum of the $\ell_1$ length of the edges 
   of $\sigma$ that cross $S(f)$, where the set $S$ is defined in~\eqref{eq:defs}.
   Let $y\in \tau^{-1}(\sigma,x)$.
   We have that $\sigma_y$ is a decreasing and flippable edge. 
   Otherwise $y\in \tau^{-1}(\sigma,x)$ would imply that $\sigma_y$ is a constraint edge, which gives that 
   $f$ is a ground state edge, contradicting that $f\succ e$.
   Let $\sigma'$ be the triangulation obtained by flipping $\sigma_y$ in $\sigma$.
   By Proposition~\ref{pro:tautree}\ref{it:treesquare}, $S(f)\subset S(\sigma_y)$, hence $\sigma_y$ intersects $S(f)$. 
   Using this and the fact that $|\sigma'_y|\leq |\sigma_y|-2$
   %, together with the fact that $S(\sigma'_y)\subset S(\sigma_y)$, 
   we obtain   
   that $L_x(\sigma')\leq L_x(\sigma)-2$. 
   Repeating these steps we obtain a sequence of triangulations so that the $i$-th triangulation in the sequence 
   is obtained via a decreasing flip of an edge of the $(i-1)$-th triangulation, and the value of $L_x$ monotonically decreases along the sequence.
   Since for any triangulation $\sigma''$ containing $f$ we have that $L_x(\sigma'')\geq|f|$, we obtain that this procedure will eventually make $f$ be a flippable 
   and decreasing edge, establishing the claim.
\end{proof}

The lemma below gives the first property of crossings of ground state edges. We denote by $\ind{\cdot}$ the indicator 
function.
\begin{proposition}[Monotonicity]\label{pro:monotonicity}
   Given any boundary condition $\xi\in\Xi(\Lambda^0)$,
   any midpoint $x\in \Lambda$, 
   any ground state edge $g\in\mathbb{G}^\xi$, 
   and any two edges $e,f\in E_x^\xi$ such that 
   $\bar\sigma_x \preceq e \preceq f$ then 
   $$
      \ind{e \cap g \neq \emptyset} \leq \ind{f \cap g \neq \emptyset}.
   $$
\end{proposition}
\begin{proof}
   We show that if $f \cap g = \emptyset$ then $e \cap g = \emptyset$. 
   If $f$ does not intersect $g$, then there is a triangulation $\sigma\in\Omega^\xi$ so that $g\in\sigma$ and 
   $\sigma_x=f$. Proposition~\ref{pro:monotoneflips} gives that we can perform a sequence of decreasing flips 
   from $\sigma$ until obtaining a triangulation $\sigma'$ 
   with $\sigma'_x=e$ since $e \preceq f=\sigma_x$. 
   Since $g$ is in ground state, $g$ is not flipped in this sequence. 
   This implies that 
   $g$ is contained in $\sigma'$ and, consequently, cannot intersect $e$. 
\end{proof}

% \begin{lemma}[No intersection by ground state]\label{lem:groundstate}
%    For any $G\subset\mathbb{G}$ and any $x\in\Lambda$ such that $G\cap E_x =\emptyset$, then $\bar\sigma_x$ does not intersect any edge of $G$.
% \end{lemma}
% \begin{proof}
%    Since $\bar\sigma$ is a ground-state triangulation, $\bar\sigma_x$ cannot intersect any edge of $\mathbb{G}$ whose midpoint is not $x$. 
% \end{proof}
% 
The following is a simple geometric lemma that we will need later.
\begin{lemma}\label{lem:2ndangle}
  In any triangle of a lattice triangulation, the largest angle is at least $\pi/2$ and the other angles are 
  at most $\pi/4$.
\end{lemma}
\begin{proof}
  Without loss of generality, assume that $\Lambda^0=[-n,n]^2\cap\mathbb{Z}^2$ and ``empty'' boundary condition (i.e., $\xi$ contains only the unit horizontal and 
  vertical edges at the boundary of $\Lambda^0$).
  The lemma will follow for arbitrary choices of $\Lambda^0$ and $\xi$ since we can choose $n$ large enough so that 
  $\Lambda^0\subseteq [-n,n]^2\cap\mathbb{Z}^2$, which gives that the set of triangulations of $\Lambda^0$ with any boundary condition $\xi$ is 
  contained in the set of triangulations of $[-n,n]^2\cap\mathbb{Z}^2$ with empty boundary condition.
  Now this property clearly holds (with equality) for any ground state triangulation of $[-n,n]^2\cap\mathbb{Z}^2$. 
  Proposition~\ref{pro:monotoneflips} implies that any triangulation $\sigma\in\Omega$ can be obtained by 
  a sequence of increasing flips from some ground state triangulation.
  Hence it suffices to show that the property in the statement of the lemma is preserved under increasing flips.
  Let $\Delta$ and $\Delta'$ be two triangles sharing an edge $e$ such that $e$ is increasing. 
  So $e$ is the smallest diagonal of the parallelogram $\Delta\cup\Delta'$.
  Let $\tilde\Delta$ and $\tilde\Delta'$ be the two new triangles obtained after flipping $e$.
  Note that the largest angles of $\tilde\Delta$ and $\tilde\Delta'$ are larger than the largest angles of 
  $\Delta$ and $\Delta'$. Moreover, the other angles of $\tilde\Delta$ and $\tilde\Delta'$ are obtained by splitting 
  angles $\theta,\theta'$ of $\Delta,\Delta'$, respectively, where $\theta,\theta'$ are not the largest angle of $\Delta,\Delta'$. 
\end{proof}

The next proposition gives an upper bound on the number of small edges intersecting a given ground state edge. 
For any triangulation $\sigma\in\Omega^\xi$, any $g\in\mathbb{G}^\xi$, and any $\ell\in\mathbb{R}^+$, let 
$$
   I_g(\sigma,\ell)=\{\sigma_x \colon \sigma_x \cap g \neq \emptyset \text{ and } |\sigma_x|\leq |g|+\ell\}
$$
be the set edges of $\sigma$ that intersect $g$ and have length at most $|g|+\ell$.
Note that $I_g$ is a set of \emph{edges} (rather than a set of midpoints), 
and the midpoints of the edges in $I_g(\sigma,\ell)$ are given by $I_g(\sigma,\ell)\cap \Lambda$.
A crucial property of the lemma below is that the bounds do not depend on $|g|$.
\begin{proposition}\label{pro:smalltrees}
   Given any boundary condition $\xi\in\Xi(\Lambda^0)$, any $g\in\mathbb{G}^\xi$ and any triangulation $\sigma\in \Omega^\xi$, all the following statements hold:
   \begin{enumerate}
      \item\label{it:larger} If an edge $\sigma_x$ of $\sigma$ intersects $g$, then 
         $|\sigma_x|\geq |g|$, with strict inequality when the midpoint of $g$ is not $x$.
      \item\label{it:thistriang} For any $\ell\geq 1$, the midpoints of $I_g(\sigma,\ell)$ are contained in the ball of radius $2\ell$ centered at the midpoint of $g$.
      \item\label{it:alltriang} There exists a universal $c>0$ such that for any $\ell\geq 1$ we have 
         that the cardinality of $I_g(\sigma,\ell)$ is at most $c\ell^2$ 
         and the cardinality of $\bigcup\nolimits_{\sigma\in\Omega^\xi} I_g(\sigma,\ell)$ is at most $c\ell^4$.
   \end{enumerate}
\end{proposition}
\begin{proof}
   First we establish the lemma when $g$ is either a unit horizontal, a unit vertical or a unit diagonal. 
   Then~\ref{it:larger} holds trivially since any edge with the same midpoint as $g$ has length at least $|g|$, 
   and an edge with midpoint different than $g$ can only 
   intersect $g$ if its length is larger than $\sqrt{2}\geq|g|$. 
   Parts~\ref{it:thistriang} and~\ref{it:alltriang} 
   follows since any edge of length at most $\ell$ that intersects 
   $g$ must be completely contained inside a ball of radius $\frac{|g|}{2}+\ell\leq \frac{\sqrt{2}}{2}+\ell$ centered at the midpoint of $g$.
   
   Now let $g$ be a ground state edge that is not a unit vertical, unit horizontal or unit diagonal. 
   This means that $g$ is constrained by a constraint edge $e\in\xi$; that is, $g \subset S(e)$.
   The proof uses the concept of excluded regions introduced in~\cite{CMSS15}.
   The excluded region of an edge $g$ is obtained by taking its minimal parallelogram
   and considering the infinite strips between both pairs of 
   opposite sides of the parallelogram, as illustrated by the shaded area in Figure~\ref{fig:excludedregion}. 
   The interior of the excluded region contains no point of $\mathbb{Z}^2$, cf.~\cite[Proposition~3.3]{CMSS15}. 
   The endpoints of the constraint edge $e$ are in regions $X$ and $Y$, 
   which are the two components of the complement of the excluded region of $g$ that contain
   an endpoint of $g$ in their boundary, 
   as illustrated in Figure~\ref{fig:excludedregion}. 
   \begin{figure}[htbp]
      \begin{center}
         \includegraphics[scale=.9]{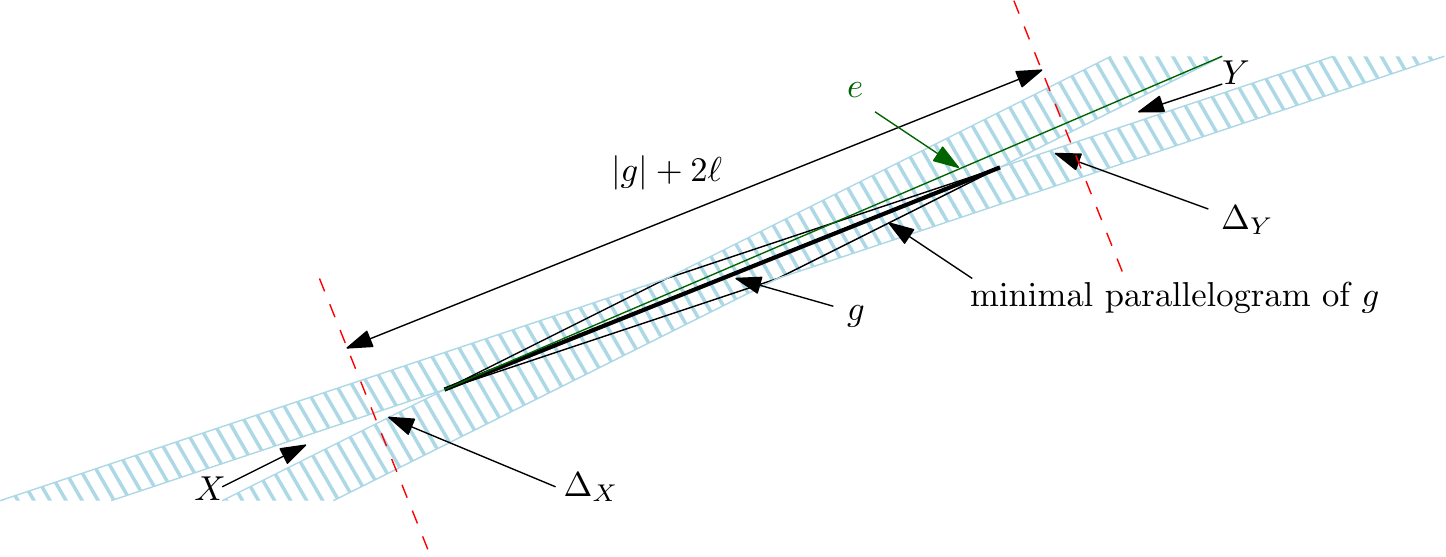}
      \end{center}\vspace{-.5cm}
      \caption{The excluded region (blue shaded area) of edge $g$.}
      \label{fig:excludedregion}
   \end{figure}
   Thus, all edges 
   that intersect $g$ must also have endpoint in $X$ and $Y$, not to intersect $e$. 
   This gives that any edge that intersects $g$ must have length larger than $|g|$, establishing part~\ref{it:larger}.
   
   Now we establish part~\ref{it:thistriang}.
   All edges in $\bigcup\nolimits_\sigma I_g(\sigma,\ell)$ must have endpoints inside the intersection of 
   $X\cup Y$ with the strip between the red dashed lines in Figure~\ref{fig:excludedregion}, which are the lines 
   perpendicular to $g$ and at distance $\ell$ from the endpoints of $g$. 
   The intersection of $X\cup Y$ and this strip forms two triangles $\Delta_X\subset X$ and $\Delta_Y\subset Y$.
   Since the triangles in the minimum parallelogram of $g$ have their two smallest angles of size at most $\pi/4$ (cf.\ Lemma~\ref{lem:2ndangle}), 
   the angle of $\Delta_X$ at the endpoint of $g$ is at most $\pi/2$.
   Therefore, $\Delta_X$ is contained 
   inside an isosceles right triangle whose right angle is at the vertex of $\Delta_X$ which is an endpoint of $g$, and 
   whose hypotenuse is in the red dashed line intersecting $\Delta_X$. The length of the hypotenuse of this isosceles right triangle is $2\ell$. 
   Therefore, the midpoints of an edge with endpoints in $\Delta_X$ and $\Delta_Y$ is contained inside a $\ell \times 2\ell$ rectangle 
   centered at the midpoint of $g$ and whose smallest edges are parallel to $g$. This rectangle is contained inside a ball of radius $2\ell$
   centered at the midpoint of $g$.
   
   Part~\ref{it:alltriang} follows since part~\ref{it:thistriang} implies that $|I_g(\sigma,\ell)|\leq c\ell^2$ for some universal constant $c>0$.
   Also, since the area of each $\Delta_X$ and $\Delta_Y$ is at most $\ell^2$, the number of points of $\Lambda^0$ inside each of $\Delta_X$ 
   and $\Delta_Y$ is at most $c_1 \ell^2$ for some positive constant $c_1$, which gives that 
   $\big|\bigcup\nolimits_{\sigma\in\Omega^\xi} I_g(\sigma,\ell)\big| \leq c_1\ell^4$.
\end{proof}

The geometric property below was the main inspiration for constructing the Lyapunov function~\eqref{eq:lyapunov}.
Roughly speaking, if an increasing edge $\sigma_x$ of $\sigma$ that is not in ground state intersects a ground state edge $g$, 
then the decreasing edges in the same tree as $\sigma_x$ (i.e., the edges $\sigma_y$ for all $y\in \tau^{-1}(\sigma,x)$) also intersect $g$.
Hence each increasing edge can be mapped to a decreasing edge of larger 
length. Since only flips of increasing edges can increase the value of the Lyapunov function, we are able to show that when taking the expectation over 
all possible flips, each flip that increases the Lyapunov function is ``compensated'' by flips that decrease the Lyapunov function.
Another fundamental property in the lemma below is that if $\sigma_x$ itself does not intersect $g$ but the edge obtained by flipping $\sigma_x$ does, 
then $g$ is also intersected by at least one of the decreasing edges in the same tree as $\sigma_x$ (i.e., there exists $y\in\tau^{-1}(\sigma,x)$ such 
that $\sigma_y$ intersects $g$).
\begin{proposition}\label{pro:squaregeom}
   Fix any boundary condition $\xi\in\Xi(\Lambda^0)$, any triangulation $\sigma\in\Omega^\xi$, any midpoint $x\in\Lambda$, and any ground state edge $g\in\mathbb{G}^\xi$. 
   Let $\{z_1,z_2\}=\tau^{-1}(\sigma,x)$; if $\tau^{-1}(\sigma,x)$ has cardinality one, set $z_1=z_2$.
   Then either $\sigma_x=g$ or we have 
   \begin{equation}
      \ind{\sigma_x \cap g \neq \emptyset} \leq \ind{\sigma_{z_1} \cap g \neq \emptyset} \ind{\sigma_{z_2} \cap g \neq \emptyset}.
      \label{eq:squaregeom1}
   \end{equation}
   Moreover, if $x\in \finc(\sigma)$ and $\eta$ is the triangulation obtained by flipping $\sigma_x$ in $\sigma$, we have that 
   \begin{equation}
      \ind{\sigma_x \cap g = \emptyset}\ind{\eta_x \cap g \neq \emptyset} \leq \ind{\sigma_{z_1} \cap g \neq \emptyset}+\ind{\sigma_{z_2} \cap g \neq \emptyset}.
      \label{eq:squaregeom2}
   \end{equation}
\end{proposition}
\begin{proof}
   We establish~\eqref{eq:squaregeom1} by contradiction. 
   Assume that $x\not\in \fdec(\sigma)\cup \fdiag(\sigma)$, otherwise $x=z_1=z_2$ and the lemma follows trivially.
   Assume also that $\sigma_x\neq g$, and that $\sigma_x$ intersects $g$ but $\sigma_{z_1}$ does not. 
   This implies that 
   there exists a triangulation that contains $g$ and $\sigma_{z_1}$. 
   Take $\zeta$ to be one such triangulation as following. 
   Remove from $\sigma$ all edges intersecting $g$, regard the edges of $\sigma$ that were not removed as 
   a new boundary condition, and define $\zeta$ to be a ground state triangulation containing $g$ given this new boundary condition.  
   Since a ground state triangulation (given any boundary condition) can be obtained by 
   the union of ground state edges by Lemma~\ref{lem:groundstate}, we obtain that $\zeta\prec \sigma$.
   Since $\sigma_x$ intersects $g$ and $\sigma_x\neq g$, we have $|\zeta_x|<|\sigma_x|$. 
   Also, since in $\sigma$ we have that $\sigma_{z_1}$ is the root of a tree containing $\sigma_x$, 
   we have that $\sigma_{x}$ is in ground state given $\sigma_{z_1}$. 
   Since $\zeta_{z_1}=\sigma_{z_1}$, then $|\zeta_{x}|\geq|\sigma_{x}|$, establishing a contradiction.
   The same reasoning applies to $z_2$.
   
   In order to establish~\eqref{eq:squaregeom2}, 
   we assume that $\sigma_x$ does not intersect $g$ but $\eta_x$ does, 
   and show that this implies that either $\sigma_{z_1}$ or $\sigma_{z_2}$ must intersect $g$.
   Let $w_1,y_1\in\tau(\sigma,z_1)$ and $w_2,y_2\in\tau(\sigma,z_2)$ be such that $\sigma_x\sigma_{w_1}\sigma_{y_1}$ and 
   $\sigma_x\sigma_{w_2}\sigma_{y_2}$ are triangles in $\sigma$. 
   Note that if $p_1$ is the common endpoint of $\sigma_{w_1}$ and $\sigma_{y_1}$, and $p_2$ is the common endpoint of 
   $\sigma_{w_2}$ and $\sigma_{y_2}$, then $\eta_x$ has endpoints $p_1,p_2$.
   Since $\eta_x$ intersects $g$ and $\eta_x\neq g$,
   it follows that $g$ intersects at least one of $\sigma_{w_1}$, $\sigma_{y_1}$, $\sigma_{w_2}$ and $\sigma_{y_2}$.
   Assume that $g$ intersects $\sigma_{w_1}$. 
   %From Proposition~\ref{pro:partition}, $w_1$ must belong to either $\tau(\sigma,z_1)$ or $\tau(\sigma,z_2)$; assume $\sigma_{w_1}$ belongs to $\tau(\sigma,z_1)$. 
   Applying the first part of the lemma with $x=w_1$ yields that $g$ intersects $\sigma_{z_1}$.
\end{proof}

\begin{proposition}\label{pro:producegroundstate}
   Fix any boundary condition $\xi\in\Xi(\Lambda^0)$, any triangulation $\sigma\in\Omega^\xi$ and any ground state edge $g\in\mathbb{G}^\xi$.
   There is a sequence of non-increasing flips from $\sigma$
   that produces a triangulation containing $g$. 
   Moreover, if $\Gamma\subset \Lambda$ are the midpoints of the edges of $\sigma$ that intersect $g$, 
   then in this sequence only the edges of midpoint in $\Gamma$ are flipped.
\end{proposition}
\begin{remark}\label{rem:pro:producegroundstate}
   In the sequence of flips above, all flips are (strictly) decreasing 
   unless when $g \not\prec \sigma_x$, where $x$ is the midpoint of $g$. 
   In this case, $g$ is a unit diagonal and the opposite unit diagonal $g'$ of midpoint $x$ also belongs to $\mathbb{G}^\xi$. 
   Then the sequence of flips consists of a sequence of decreasing flips that culminates in a triangulation containing $g'$ and its minimal parallelogram,  
   and then a length-preserving flip of $g'$ to obtain $g$. 
\end{remark}
\begin{proof}[Proof of Proposition~\ref{pro:producegroundstate}]
   Assume for the moment that $g\prec \sigma_x$, where $x$ is the midpoint of $g$. 
   We perform the same sequence of triangulations 
   $\sigma=\eta^0,\eta^1,\eta^2,\ldots$ as in the proof of Proposition~\ref{pro:monotoneflips} with $e=g$.
   In this sequence, $\eta^i$ is obtained from $\eta^{i-1}$ by performing a decreasing flip of an edge of midpoint in $\tau^{-1}(\eta^{i-1},x)$.
   Since $\eta^{i-1}_x$ intersects $g$ for all $i$, Proposition~\ref{pro:squaregeom} gives that all edges of midpoint in 
   $\tau^{-1}(\eta^{i-1},x)$ also intersect $g$. Therefore, all flipped edges in this sequence must intersect $g$.
   
   When $g\not\prec \sigma_x$, we have that $g$ is a unit diagonal and the opposite unit diagonal $g'$ of midpoint $x$ also belongs to
   $\mathbb{G}^\xi$; otherwise for all $f\in E_x^\xi$ we have $g\preceq f$. 
   This gives that $g'\preceq \sigma_x$.
   From the previous case we obtain a sequence of triangulations $\sigma=\eta^0,\eta^1,\eta^2,\ldots,\eta^k$
   such that $\eta^k_x=g'$ and, for all $i$, 
   $\eta^i$ is obtained by performing a decreasing flip to an edge of midpoint in $\tau^{-1}(\eta^{i-1},x)$. Since $\eta^{i-1}_x$ intersects $g$, 
   we have that only edges intersecting $g$ are flipped in this sequence. 
   Now we claim that we can perform a sequence of decreasing flips from $\eta^k$ to obtain triangulations $\eta^{k+1},\eta^{k+2},\ldots, \eta^\ell$ 
   such that $\eta^\ell$ contains $g'$ and its 
   minimal parallelogram. Using this the lemma follows since we can perform a length-preserving flip of $g'$ in $\eta^\ell$, which produces $g$. 
   To establish the claim, 
   let $w_1,w_2,w_3,w_4$ be the midpoints of the edges in the minimal parallelogram of $g$. 
   Note that there are exactly four edges $h_1,h_2,h_3,h_4\in\mathbb{G}^\xi$ (which are unit horizontal and vertical edges) 
   such that $w_i$ is the midpoint of $h_i$ for all $i\in\{1,2,3,4\}$.
   The sequence of flips is obtained by applying the previous case for each $h_i$; i.e., 
   at each step we perform a decreasing flip to an edge of midpoint in $\tau^{-1}(\cdot, w_i)$ for some $i\in\{1,2,3,4\}$ until obtaining 
   a triangulation contaning $h_1,h_2,h_3,$ and $h_4$. It remains to show that in this sequence we only flip edges that intersect $g$.
   Note that any edge of a triangulation that intersects $h_i$ for some $i$ must 
   intersect either $g$ or $g'$.    
   Since $g'$ belongs to all 
   triangulations $\eta^{k+1},\eta^{k+2},\ldots,\eta^\ell$ and any edge flipped in this sequence intersects $h_i$ for some $i$, we have that 
   all flipped edges intersect $g$, and the claim is established. 
\end{proof}

%############################################################################################
%############################################################################################
%############################################################################################
\section{Proof of the Lyapunov function (Theorem~\ref{thm:lyapunov})}\label{sec:proof}
During the proof of Theorem~\ref{thm:lyapunov}, we will need to treat small edges (edges smaller than some constant $C$) separately. 
We fix $\alpha\in(1,\lambda^{-1/2})$ and set $C>1$ large enough so that the following two conditions hold:
\begin{equation}
  \alpha^{-C/4}\leq \frac{1}{10}
  \quad\text{and}\quad
  C\alpha^{-C/2} \leq \frac{\alpha^2-1}{10\alpha^2}.
  \label{eq:condclarge}
\end{equation}
Also, we will need to handle ``small trees'' separately: trees whose root edge is smaller than some other constant $C'$. 
After $C$ has been fixed, set $C'$ large enough so that the following conditions hold:
\begin{equation}
   C'>\left(3+\frac{2}{\alpha-1}\right)C^2 
   \quad\text{and}\quad
   4x\alpha^{2C} \leq \alpha^x \text{ for all $x\geq C'$.}
   \label{eq:condcp}
\end{equation}

Throughout this section we fix an arbitrary boundary condition $\xi\in\Xi(\Lambda^0)$ and a triangulation $\sigma\in\Omega^\xi$.
We need to introduce some notation.
For any flippable edge $x\in\Lambda$ of $\sigma$, define
\begin{equation}
   \psi_x = \psi_x(\sigma) = \text{the length of the shortest edge different than $\sigma_x$ in a 
      triangle of $\sigma$ containing $\sigma_x$.}
   \label{eq:defpsi}
\end{equation}
Given any $g\in\mathbb{G}^\xi$, define 
$$
   \fdec^g(\sigma) = \{x\in \fdec(\sigma) \colon \sigma_x\cap g \neq \emptyset\}
   \quad\text{and}\quad
   \finc^g(\sigma) = \{x\in \finc(\sigma) \colon \sigma_x^x\cap g \neq \emptyset\},
$$
where 
$$
   \text{$\sigma^x$ stands for the triangulation obtained by flipping $\sigma_x$ in $\sigma$}.
$$
In words, $\fdec^g(\sigma)$ is the set of decreasing edges of $\sigma$ that intersect $g$ 
and $\finc^g(\sigma)$ is the set of increasing edges of $\sigma$ that either intersect $g$ or get to intersect $g$ after a flip.
(The fact that it is enough 
to define $\finc^g(\sigma)$ in terms of $\sigma^x$ only is a consequence of Proposition~\ref{pro:monotonicity}.)
Since the edges of $\mathbb{G}^\xi$ are all compatible with $\xi$, 
we obtain that, unlike $\fdec(\sigma)$, the set $\fdec^g(\sigma)$
contains no midpoint of $\lambda^\bc$. 
Let
$$
   \sigma' \text{ be the random triangulation obtained from $\sigma$ by one step of the Glauber dynamics,}
$$
and for $x\in\Lambda$ let  
$$
   \tilde \sigma^x = \text{the triangulation obtained by choosing $x$ to be flipped in $\sigma$.}
$$
The triangulation $\tilde\sigma^x$ differs from $\sigma^x$ since $\tilde\sigma^x$ is a \emph{random} triangulation 
(the probability that the flip actually occurs is $\frac{\lambda^{|\sigma^x_x|}}{\lambda^{|\sigma_x|}+\lambda^{|\sigma^x_x|}}$), 
while $\sigma^x$ is a deterministic triangulation.
Hence,
\begin{equation}
   \E_\sigma\big(\Psi_g(\sigma')-\Psi_g(\sigma)\big)
   =\sum_{x\in\Lambda}\frac{1}{|\Lambda|}\E_\sigma\big(\Psi_g(\tilde \sigma^x)-\Psi_g(\sigma)\big)
   =\sum_{x\in \fdec^g(\sigma)\cup\finc^g(\sigma)}\frac{1}{|\Lambda|}\E_\sigma\big(\Psi_g(\tilde \sigma^x)-\Psi_g(\sigma)\big).
   \label{eq:simpledecomp}
\end{equation}

We start estimating the expected change in $\Psi_g(\sigma)$ incurred by flipping a given edge. 
For $x\in\Lambda$, define 
\begin{align}
   \rho_g(\sigma,x) 
   = -\frac{\alpha^{|\sigma_x|-|g|}}{1+\lambda^{2\psi_x}}
   \quad \text{ for all $x\in \fdec^g(\sigma)$},\nonumber\\
   \rho_g(\sigma,x) 
   = \frac{\alpha^{|\sigma_x|-|g|}(\alpha\lambda)^{2\psi_x}}{1+\lambda^{2\psi_x}}
      \quad \text{ for all $x\in \finc^g(\sigma)$}
   \label{eq:defrho}
\end{align}
and 
$\rho_g(\sigma,x)=0$ for all other $x$.
\begin{lemma}\label{lem:individual}
   Fix any boundary condition $\xi\in\Xi(\Lambda^0)$, 
   any triangulation $\sigma\in\Omega^\xi$, 
   any ground state edge $g\in\mathbb{G}^\xi$, and any midpoint $x\in \Lambda$. 
   We have 
   $$
      \E_\sigma\big(\Psi_g(\tilde \sigma^x)-\Psi_g(\sigma)\big)
      = \rho_g(\sigma,x).
   $$
\end{lemma}
\begin{proof} 
   If $x\in \fdiag(\sigma)$, 
   then $\sigma_x\in\mathbb{G}^\xi$ and, consequently, $\sigma_x$ only intersects $g$ if the midpoint of $g$ is $x$; 
   hence, $\Psi_g(\sigma)=\Psi_g(\sigma^x)$ in this case.
   Assume henceforth that $x\not\in \fdiag(\sigma)$.
   Then for any $x$ such that $\sigma_x$ is a flippable edge we have that the absolute value of $|\sigma^x_x|-|\sigma_x|$ is $2\psi_x$. 
To see this, note that if $w,y\in\Lambda$ are such that $\sigma_w,\sigma_y,\sigma_x$ form a triangle of $\sigma$ with $|\sigma_w|\geq|\sigma_y|$, then $\psi_x=|\sigma_y|$. 
Thus if $\sigma_x$ is the largest edge of the triangle we have $|\sigma_x|=|\sigma_w|+|\sigma_y|$ and $|\sigma_x^x|=|\sigma_w|-|\sigma_y|$, 
otherwise we have $|\sigma_x|=|\sigma_w|-|\sigma_y|$ and $|\sigma_x^x|=|\sigma_w|+|\sigma_y|$.
   Therefore, the probability that $\sigma_x$ is actually flipped in $\tilde\sigma^x$ is 
   $$
      \frac{\lambda^{|\sigma^x_x|}}{\lambda^{|\sigma_x|}+\lambda^{|\sigma^x_x|}}
      = \frac{\lambda^{|\sigma^x_x|}}{\lambda^{|\sigma_x|\land |\sigma^x_x|}(1+\lambda^{2\psi_x})}
      =\frac{\lambda^{2\psi_x}\ind{x\in \finc(\sigma)}+\ind{x\in \fdec(\sigma)\setminus \lambda^\bc}}{1+\lambda^{2\psi_x}}. 
   $$
   Therefore, if $x\in \fdec^g(\sigma)$, we have
   $$
      \E_\sigma\big(\Psi_g(\tilde \sigma^x)-\Psi_g(\sigma)\big)
      = -\alpha^{|\sigma_x|-|g|}\frac{1}{1+\lambda^{2\psi_x}}
      = \rho_g(\sigma,x).
   $$
   If $x\in \finc^g(\sigma)$, we obtain
   $$
      \E_\sigma\big(\Psi_g(\tilde\sigma^x)-\Psi_g(\sigma)\big)
      = \alpha^{|\sigma_x|+2\psi_x-|g|}\frac{\lambda^{2\psi_x}}{1+\lambda^{2\psi_x}}
      = \frac{\alpha^{|\sigma_x|-|g|}(\alpha\lambda)^{2\psi_x}}{1+\lambda^{2\psi_x}}
      = \rho_g(\sigma,x).
   $$
\end{proof}

%############################################################################################
\subsection{Proof overview}\label{sec:overview}
Our goal is to show that $\sum_{x\in\finc^g(\sigma)}\rho_g(\sigma,x)$ can be bounded above by 
$-c \sum_{x\in\fdec^g(\sigma)}\rho_g(\sigma,x)$ for some constant $c<1$, and then apply Lemma~\ref{lem:individual} and~\eqref{eq:simpledecomp} to
establish Theorem~\ref{thm:lyapunov}. We will do this by comparing each $\rho_g(\sigma,x)$ with $x\in \finc^g(\sigma)$ with 
$\rho_g(\sigma,z)$ for $z$ being a root of a tree containing $x$ (i.e., $z\in\tau^{-1}(\sigma,x)$). 
Proposition~\ref{pro:squaregeom} guarantees that there exists
such a $z$ for which $z\in \fdec^g(\sigma)$. Proposition~\ref{pro:twotrees}~\ref{it:leaf} and~\ref{it:inc} gives that for any $x\in \finc^g(\sigma)$, $x$ is a leaf in all
trees containing $x$, so in our proof we will restrict our attention to the roots and leaves of the trees.

The proof is split into sections.
In Section~\ref{sec:sleaves} 
we bound above $\rho_g(\sigma,x)$ with $x\in \finc^g(\sigma)$ in terms of $\rho_g(\sigma,z)$ with $z\in\tau^{-1}(\sigma,x)$
for small leaves (leaves $\sigma_x$ that are smaller enough 
in comparison to $\sigma_z$). In Section~\ref{sec:1dconfig} we do the same for large leaves, which will require a more delicate proof. 
Then in Section~\ref{sec:decrease} we combine the result of the previous two sections with~\eqref{eq:simpledecomp} and establish that the 
expected change in the Lyapunov function can be written as a function of only the decreasing edges. 
In Section~\ref{sec:decdominate} we show that the value of the Lyapunov function can also be written in terms of the decreasing edges only.
Combining these two results together gives that the expected change in the Lyapunov function can be written in terms of the value 
of the Lyapunov function. This is established in Section~\ref{sec:finish}, completing the proof of Theorem~\ref{thm:lyapunov}.

%############################################################################################
\subsection{Handling small leaves}\label{sec:sleaves}
For any $z\in\fdec(\sigma)\cup\fdiag(\sigma)$, we will employ the following definition:
$$
   \tau_\sleaves(\sigma,z)=\big\{x \in \tau_\leaves(\sigma,z) \colon |\sigma_x|\leq |\sigma_z|-C\big\},
$$
where $\tau_\leaves(\sigma,z)$ are the leaves of $\tau(\sigma,z)$.
The subscript $\mathrm{sl}$ above stands for ``small leaves.''
In the lemma below, recall that $\rho_g(\sigma,z)<0$ for all $z\in\fdec^g(\sigma)$.
\begin{lemma}\label{lem:largetreesmalledges}
   Given any boundary condition $\xi\in\Xi(\Lambda^0)$, any triangulation $\sigma\in\Omega^\xi$, any $z\in \fdec^g(\sigma)$, and any $g\in\mathbb{G}^\xi$,
   we have
   $$
      \sum_{x\in \tau_\sleaves(\sigma,z)} \rho_g(\sigma,x)
      \leq 2|\sigma_z|\alpha^{C-|g|} - 4\alpha^{-C}\rho_g(\sigma,z).
   $$
   If in addition we have $|\sigma_z|\geq C'$, then the bound above can be simplified to
   $$
      \sum_{x\in \tau_\sleaves(\sigma,z)} \rho_g(\sigma,x)
      \leq -5\alpha^{-C}\rho_g(\sigma,z).
   $$
\end{lemma}
\begin{proof}
   By Proposition~\ref{pro:tautree}\ref{it:leaves} 
   we have that $\sum_{x\in \tau_\sleaves(\sigma,z)}|\sigma_x|\leq 2 |\sigma_z|$. Then, since $\alpha\lambda<1/\alpha$, we write
   $$
      \sum_{x\in \tau_\sleaves(\sigma,z)} \rho_g(\sigma,x)
      = \sum_{x\in \tau_\sleaves(\sigma,z)} \frac{\alpha^{|\sigma_x|-|g|}(\alpha\lambda)^{2\psi_x}}{1+\lambda^{2\psi_x}}
      \leq \sum_{x\in \tau_\sleaves(\sigma,z)}\frac{\alpha^{|\sigma_x|-|g|-2\psi_x}}{1+\lambda^{2\psi_x}}.
   $$
   Now let $\tau'_\sleaves(\sigma,z)\subseteq \tau_\sleaves(\sigma,z)$ be the set of midpoints $x\in\tau_\sleaves(\sigma,z)$ such that $|\sigma_x|-2\psi_x\leq C$. 
   Then for $\tau'_\sleaves(\sigma,z)$ we use the simple bound
   $$
      \sum_{x\in \tau'_\sleaves(\sigma,z)} \rho_g(\sigma,x)
      \leq |\tau'_\sleaves(\sigma,z)|\alpha^{C-|g|}
      \leq 2|\sigma_z|\alpha^{C-|g|}
      \leq \frac{4|\sigma_z|\alpha^{C-|g|}}{1+\lambda^{2\psi_z}}.
   $$
   When $|\sigma_z|\geq C'$, using the condition on $C'$ in~\eqref{eq:condcp} we obtain
   \begin{equation}
      \sum_{x\in \tau'_\sleaves(\sigma,z)} \rho_g(\sigma,x)
      \leq \frac{\alpha^{|\sigma_z|-|g|-C}}{1+\lambda^{2\psi_z}}
      = -\alpha^{-C}\rho_g(\sigma,z).
      \label{eq:smalledges}
   \end{equation}
   For the other edges, we use the fact that $|\sigma_x|> \psi_x$, which implies that $\psi_x$ is the size of the smallest edge in the triangle containing 
   $\sigma_x$ in $\tau(\sigma,z)$, and hence Proposition~\ref{pro:smallestedgehierarchy}
   gives that $\psi_x\leq \psi_z$. Using this, we obtain
   \begin{equation}
      \sum_{x\in \tau_\sleaves(\sigma,z)\setminus \tau'_\sleaves(\sigma,z)} \rho_g(\sigma,x)
      \leq \frac{1}{1+\lambda^{2\psi_z}} \sum_{x\in \tau_\sleaves(\sigma,z)\setminus \tau'_\sleaves(\sigma,z)} \alpha^{|\sigma_x|-|g|-2\psi_x}.
      \label{eq:slp}
   \end{equation}
   For the edges in $\tau_\sleaves(\sigma,z)\setminus \tau'_\sleaves(\sigma,z)$ we will also leverage on the fact that they are not small, 
   applying the following technical estimate.
   Given any positive $\ell_1 \geq \ell_2 \geq \cdots \geq \ell_k\in\mathbb{Z}$ and any $S\geq \sum_{i=1}^k\ell_i$ such that 
   $k\geq 2$ and $\ell_i\in [C,S-C]$ for all $i$, we have 
   \begin{align*}
      \sum_{i=1}^k \alpha^{\ell_i}
      = \alpha^{\ell_1+\ell_2}(\alpha^{-\ell_1}+\alpha^{-\ell_2}) + \sum_{i=3}^k \alpha^{\ell_i}
      \leq 2\alpha^{-C}\alpha^{\ell_1+\ell_2} + \sum_{i=3}^k \alpha^{\ell_i}.
   \end{align*}
   Using that $2\alpha^{-C}<1$, and proceeding in the same way as above, we obtain
   \begin{align*}
      \sum_{i=1}^k \alpha^{\ell_i}
      &\leq \alpha^{\ell_1+\ell_2} + \sum_{i=3}^k \alpha^{\ell_i}\\
      &\leq \alpha^{\ell_1+\ell_2+\cdots+\ell_{k-1}} + \alpha^{\ell_{k}}\\
      &\leq 2\alpha^{-C}\alpha^{\sum_{i=1}^k \ell_i}
      \leq 2\alpha^{S-C}.
   \end{align*}
   
   If $k=1$, then we have $\sum_{i=1}^k \alpha^{\ell_i}=\alpha^{\ell_1}\leq \alpha^{S-C}$, 
   and we can simply use the upper bound above.
   We apply this estimate twice, once for the elements of 
   $\tau_\sleaves(\sigma,z)\setminus \tau'_\sleaves(\sigma,z)$ that belong to $\tau^{(1)}(\sigma,z)$ 
   and another for the ones that belong to $\tau^{(2)}(\sigma,z)$. Since 
   we have that the sum of the $|\sigma_x|$ for $x$ in each of these sets is at most $S=|\sigma_z|$, 
   applying this to the right-hand side of~\eqref{eq:slp} yields
   \begin{align*}
      \sum_{x\in \tau_\sleaves(\sigma,z)\setminus \tau'_\sleaves(\sigma,z)} \rho_g(\sigma,x)
      &\leq \frac{1}{1+\lambda^{2\psi_z}} 4\alpha^{|\sigma_z|-|g|-C}
      = -4\alpha^{-C}\rho_g(\sigma,z).
   \end{align*}
   Summing this and~\eqref{eq:smalledges} establishes the lemma.
\end{proof}

%############################################################################################
\subsection{Large leaves and 1-dimensional configurations}\label{sec:1dconfig}
As mentioned above, the most delicate part of the proof will be to etablish an upper bound on $\rho_g(\sigma,x)$ when
$x\in \finc^g(\sigma)$ is such that $x$ belongs to 
a tree $\tau(\sigma,z)$ for which $\sigma_x$ and $\sigma_z$ have almost the same length. This is the case we treat in this section.

Here we only need to consider trees rooted at long edges.
Fix a boundary condition $\xi\in\Xi(\Lambda^0)$ and a ground state edge $g\in\mathbb{G}^\xi$.
Consider the increasing edges $\sigma_x$ for which either both trees containing $x$ have root in $\fdec^g(\sigma)$ 
and have length at least $|g|+C'$, 
or one of them has root satisfying the conditions above 
and the other has root outside $\fdec^g(\sigma)$. In addition, 
only consider $x$ that does not belong to $\tau_\sleaves(\sigma,z)$ for any $z$. 
More precisely, define
$$
   X=\big\{x\in \finc^g(\sigma) \colon \forall z\in \fdec^g(\sigma) \cap \tau^{-1}(\sigma,x) \text{ we have } |\sigma_z|> |g|+C' \text{ and }
      x\not\in \tau_\sleaves(\sigma,z)\big\}.
$$
Proposition~\ref{pro:squaregeom} gives that the set $\fdec^g(\sigma) \cap \tau^{-1}(\sigma,x)$ has at least one element.

We construct the following bipartite graph $H$ with vertex sets $X$ and $\fdec^g(\sigma)\setminus  F_{|g|+C'}(\sigma)$. 
(Recall the definition of $F_\ell(\sigma)$ from~\eqref{eq:deff}.)
To avoid ambiguity, we will refer to the connections between pairs of vertices of $H$ as \emph{links} instead of edges; 
we reserve the word 
\emph{edges} to the edges of a triangulation.
There is a link between $x\in X$ and $z\in \fdec^g(\sigma)\setminus F_{|g|+C'}(\sigma)$ in $H$
if $x\in\tau(\sigma,z)$. 
Since $\tau^{-1}(\sigma,x)$ has cardinality at most two (cf.~Proposition~\ref{pro:twotrees}\ref{it:card}), 
the degree of $x$ in $H$ is at most two.
Also each edge of midpoint $z\in \fdec^g(\sigma)\setminus  F_{|g|+C'}(\sigma)$ has length at least $|g|+C'>3C$, which gives that a leaf 
$x\in\tau_\leaves(\sigma,z)\setminus \tau_\sleaves(\sigma,z)$ must have size at least $|\sigma_z|-C>2|\sigma_z|/3$.
Since $\sum_{y\in \tau_\leaves(\sigma,z)}|\sigma_y|=2|\sigma_z|$, the set $\tau_\leaves(\sigma,z)\setminus \tau_\sleaves(\sigma,z)$ has at most two elements,
which gives that the degree of $z$ in $H$ is at most two.
Since all vertices of $H$ have degree at most two, $H$ is a graph formed by paths and cycles.\footnote{Actually, as it will be proved in Lemma~\ref{lem:1d}, 
there is no cycle in $H$. But we will not need this fact.}

We will treat each path $P$ of $H$ individually. 
Since $H$ is bipartite, the vertices of $P$ must alternate between midpoints in $X$ (which correspond to increasing edges of $\sigma$) and 
midpoints in $\fdec^g(\sigma)\setminus  F_{|g|+C'}(\sigma)$ (which are decreasing edges of $\sigma$). 
If the number of decreasing edges in $P$ is at least as large as the number of increasing edges in $P$, 
then we can construct a one-to-one mapping between increasing and decreasing edges of $P$, which allow us to show a contraction 
in the Lyapunov function. 
The main challenge is when the number of increasing edges in $P$ is larger than the number of decreasing edges (that is, the 
number of increasing edges is one plus the number of decreasing edges). 
In this case, we will show that the path must form a specific shape in $\sigma$, 
which implies that the path is long enough. Only with this we can establish a contraction in the Lyapunov function for this case. 
This is proved in Lemma~\ref{lem:1d}.
\begin{lemma}\label{lem:1d}
   Let $P=\{w_1,w_2,\ldots,w_\ell\}$ be a path of $H$
   such that $w_1,w_\ell\in X$.
   Then, 
   $$
      \ell \geq \frac{C'-C}{C^2}.
   $$
   Moreover, $\psi_{w_1}=\psi_{w_2}=\cdots=\psi_{w_\ell}$.
   We also obtain that $H$ has no cycles.
\end{lemma}
\begin{proof}
   Given $P$, we will construct a path of adjacent triangles $\Delta_1,\Delta_2,\ldots,\Delta_k$ in $\sigma$ starting from $\sigma_{w_1}$ until reaching 
   $\sigma_{w_\ell}$, and such that it contains all edges 
   $\sigma_{w_1},\sigma_{w_2},\ldots,\sigma_{w_\ell}$. 
   We will show that this path of triangles must have a certain 1-dimensional shape, which we illustrate in Figure~\ref{fig:1dconfig}(a).
   \begin{figure}[htbp]
      \begin{center}
         \includegraphics[scale=.9]{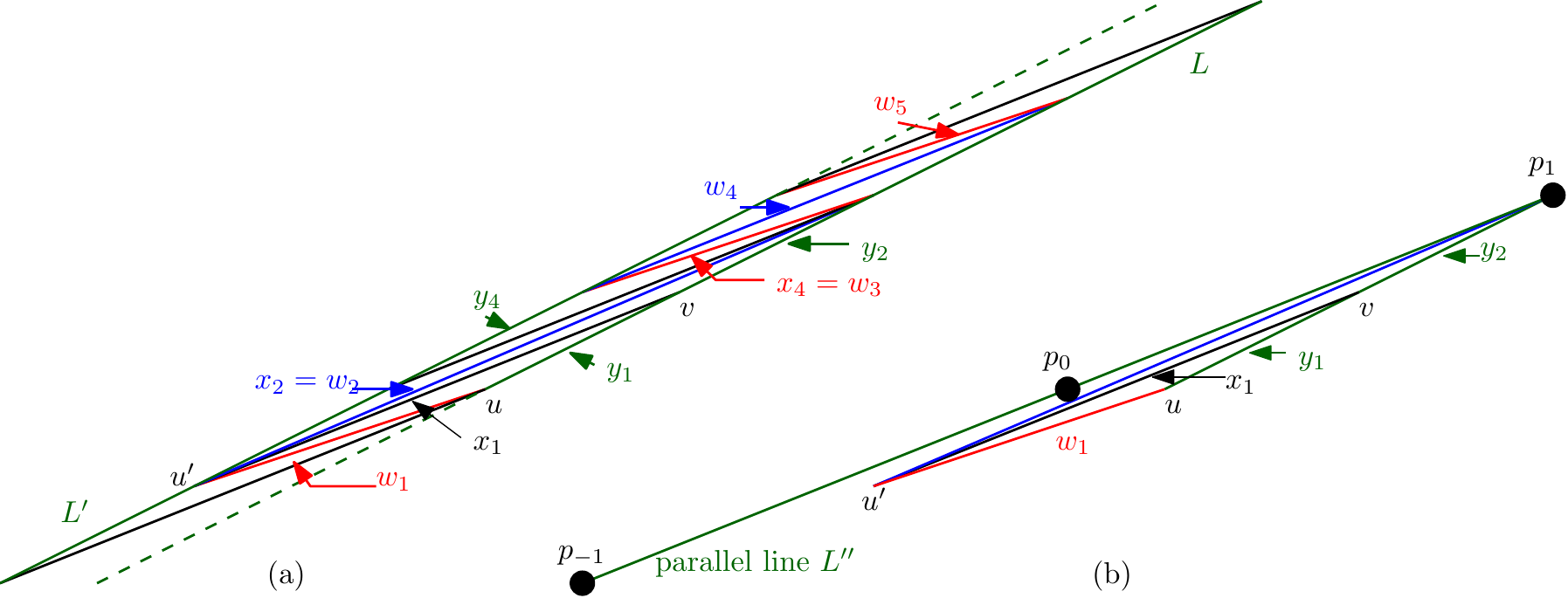}
      \end{center}\vspace{-.5cm}
      \caption{(a) A 1-dimensional configuration: blue edges are decreasing, red edges are increasing, 
         and green edges are the other edges that form the parallel lines $L$ and $L'$. (b) Details of the proof of Lemma~\ref{lem:1d}.}
      \label{fig:1dconfig}
   \end{figure}

   Now we construct the path of triangles. 
   Let $\Delta_1$ be the triangle of $\sigma$ containing $\sigma_{w_1}$ and formed of edges of midpoint in $\tau(\sigma,w_2)$. 
   Let $\sigma_{x_1}$ and $\sigma_{y_1}$ with $|\sigma_{x_1}|\geq|\sigma_{y_1}|$ be the other 
   edges of $\Delta_1$. Since $\sigma_{w_1}$ is increasing, we have $|\sigma_{x_1}|>|\sigma_{w_1}|$. 
   Let $\Delta_2$ be the other triangle containing $\sigma_{x_1}$ in $\sigma$, 
   and let $\sigma_{x_2},\sigma_{y_2}$ be the other edges of $\Delta_2$ with $\sigma_{x_2}$ larger than 
   $\sigma_{y_2}$. Then either $\sigma_{x_1}$ is a decreasing edge or $|\sigma_{x_2}|>|\sigma_{x_1}|$. 
   In the latter case, we look at the other triangle containing $\sigma_{x_2}$ and repeat the procedure above 
   until we reach a triangle $\Delta_j$ such that $\sigma_{x_{j}}$ is a decreasing edge. 
   Since $w_1\in X$, it holds that $x_{j}=w_2$; 
   this must happen since at each step we cross the largest edge of the triangle, traversing a path in $\tau(\sigma,w_2)$ from the leaf $w_1$ to
   the root $w_2$. This establishes a path of adjacent triangles from $\sigma_{w_1}$ to $\sigma_{w_2}$. 
   Similarly, we can find a path of adjacent triangles from the increasing edge $\sigma_{w_3}$ to the decreasing edge $\sigma_{w_2}$ 
   and concatenate the two paths to obtain a path from $\sigma_{w_1}$ to $\sigma_{w_3}$.
   Iterating this procedure we obtain a path of adjacent triangles from $\sigma_{w_1}$ to $\sigma_{w_\ell}$. 
   
   Now define $L$ to be the infinite line containing $\sigma_{y_1}$ and $L'$ to be the infinite line that is parallel to 
   $L$ and contains the other endpoint of $\sigma_{w_1}$.
   We show that the union of the $\sigma_{y_j}$ must lie on $L \cup L'$, 
   and that each triangle $\Delta_j$ has all its vertices on $L \cup L'$ (as illustrated in Figure~\ref{fig:1dconfig}(a)). 
   First assume that $x_1\neq w_2$, and let $\sigma_{y_1}=(u,v)$ 
   where $u,v\in\Lambda^0$ and $v=\sigma_{x_1}\cap \sigma_{y_1}$ (refer to Figure~\ref{fig:1dconfig}(b)). 
   We claim that 
   \begin{equation}
      \sigma_{y_2}=(v,v+v-u), 
      \label{eq:y2}
   \end{equation}
   which is an edge colinear with $\sigma_{y_1}$ and is illustrated by $(v,p_1)$ in Figure~\ref{fig:1dconfig}(b).
   The reason for~\eqref{eq:y2} is the following. Given $\sigma_{x_1}$, the third vertex of $\Delta_2$ must lie on a line $L''$ parallel to $\sigma_{x_1}$ since the area of 
   $\Delta_2$ is $1/2$; this is the line containing $p_{-1},p_0,p_1$ in 
   Figure~\ref{fig:1dconfig}(b). This line must pass through the vertex 
   $v+v-u$ since $v-(v-u)=u$ pass through a similar line on the other side of $\sigma_{x_1}$. 
   Let $\ldots,p_{-2},p_{-1},p_0,p_1,p_2,\ldots$ be the lattice vertices on $L''$ such that $p_0$ is the only such vertex 
   in the minimal parallelogram of $\sigma_{x_1}$.
   Since $x_1\neq w_2$, then $\sigma_{x_1}$ is not decreasing and $p_0$ is not a vertex of $\Delta_2$. 
   Let $p_1=v+v-u$ and define $u'\in\mathbb{Z}^2\cap L'$ such that $\sigma_{w_1}=(u,u')$.
   Note that for the same reason that $(v,p_1)$ is a translate of $\sigma_{y_1}$, 
   $(u',p_{-1})$ is a translate of $\sigma_{w_{1}}$. If $p_{-1}$ were a vertex of $\Delta_2$, then 
   $$
      |\sigma_{x_2}|=|\sigma_{x_1}|+|\sigma_{w_1}|
      \geq 2|\sigma_{w_1}|\geq 2(|\sigma_{w_2}|-C). 
   $$
   But since $|\sigma_{w_2}|\geq 2C$, we obtain $|\sigma_{x_2}|\geq |\sigma_{w_2}|$ 
   which is a contradiction since $x_2\in\tau(\sigma,w_2)$. 
   Similarly, any point $p_{-2},p_{-3},\ldots$ in $L''$ cannot be a vertex of $\Delta_2$; otherwise it makes $\sigma_{x_2}$ be too large. 
   For a similar reason, any point $p_2,p_3,\ldots$ in $L''$ cannot be a vertex of $\Delta_2$, otherwise we would have 
   $$
      |\sigma_{x_2}|\geq |\sigma_{x_1}|+\|v-p_2\|_1 = |\sigma_{x_1}|+|\sigma_{y_1}|+|\sigma_{x_1}|\geq 2|\sigma_{w_1}|. 
   $$
   Therefore, $p_1$ must be the vertex of $\Delta_2$ giving that $\sigma_{y_1}$ and $\sigma_{y_2}$ are colinear, and establishing~\eqref{eq:y2}.
   Now, if $x_1=w_2$, then $x_1$ is decreasing and $p_0$ is the vertex of $\Delta_2$. 
   In this case, $x_1$ is the diagonal of a paralellogram and, clearly, $(u',p_0)$ is parallel to and has the same length as $\sigma_{y_1}$. 
   Proceeding inductively, 
   we obtain that the path of triangles must be between the two (green) parallel lines $L$ and $L'$ in Figure~\ref{fig:1dconfig}(a), 
   which also implies that $H$ has no cycle. Also, it implies that $\psi_{w_i}=|\sigma_{y_1}|$ for all $i$.
   
   Now we compute a lower bound on $\ell$ (the size of the path $P$). 
   First notice that $\sigma_{w_1}$ and $\sigma_{w_\ell}$ do not intersect $g$, 
   otherwise their degree in $H$ would be two. Also, if $R$ is the region between $L$ and $L'$, and $\hat g$ is the closure of $g$ (i.e., $\hat g$ is 
   the union of $g$ and its endpoints), then $R\setminus \hat g$ is not simply connected
   (i.e., $\hat g$ intersects both $L$ and $L'$).
   The reason for this is that $R$ must intersect $g$ (because $\sigma_{w_2}$ intersects $g$ and is contained in $R$), 
   but $R$ does not contain any lattice point since $R$ is part of the excluded region of $\sigma_{w_2}$. (See~\cite[Proposition~3.3]{CMSS15} 
   for the proof that the excluded region of any edge does not contain lattice points.) 
   Let $s=\hat g \cap L$ and $s'=\hat g\cap L'$ be the points where $\hat g$ intersects $L$ and $L'$,
   and let $r,r'$ be the endpoints of $\sigma_{w_\ell}$ such that $r\in L$ and $r'\in L'$. 
   Clearly, $s$ is between $r$ and $u$ in $L$, and $s'$ is between $r'$ and $u'$ in $L'$.
   Recalling that $\sigma^{w_1}$ stands for the triangulation obtained from $\sigma$ by flipping $\sigma_{w_1}$, 
   and since $\sigma_{w_1}^{w_1}$ intersects $g$ and $|\sigma_{y_1}|\leq|\sigma_{w_2}|-|\sigma_{w_1}|\leq C$, we have that 
   $$
      \|s-u\|_1 \leq |\sigma_{y_1}| \leq C,
   $$
   and 
   $$
      \|s'-u'\|_1 \geq |\sigma_{w_1}| -|g| \geq |\sigma_{w_2}| - C - |g| \geq C'-C.
   $$
   Consequently, the number of edges $\sigma_{y_i}$ on $L'$ that belongs to triangles of the path $\Delta_1,\Delta_2,\ldots$ is at least 
   $\frac{C'-C}{C}$. Since for even $j$ we have $|\sigma_{w_j}|\leq |\sigma_{w_{j-1}}|+C$, there must be at most $C$ edges $\sigma_{y_i}$ between
   $\sigma_{w_j}$ and $\sigma_{w_{j-1}}$. Therefore, we have that 
   $$
      \ell \geq \frac{C'-C}{C^2}.
   $$
\end{proof}

\begin{lemma}\label{lem:largetree1d}
   Let $P=\{w_1,w_2,\ldots,w_\ell\}$ be a path of $H$.
   Then
   $$
      \sum_{x\in P\cap \finc^g(\sigma)} \rho_g(\sigma,x)
      \leq -\sum_{z\in P\cap \fdec^g(\sigma)} \alpha^{-2\psi_z}\rho_g(\sigma,z).
   $$
\end{lemma}
\begin{proof}
   Since $H$ is bipartite, the midpoints in $P$ must alternate between increasing and decreasing edges of $\sigma$. 
   Also, if $\sigma_{w_i}$ is increasing, then $|\sigma_{w_i}|\leq |\sigma_{w_{i-1}}| \land |\sigma_{w_{i+1}}|$.
   By Lemma~\ref{lem:1d} we have that all $\psi_{w_i}$ are the same; for simplicity we write $\psi=\psi_{w_i}$. 
   If $\sigma_{w_\ell}$ is decreasing, then the lemma follows
   since each increasing $\sigma_{w_i}$ can be associated with the decreasing edge $\sigma_{w_{i+1}}$, and we can write 
   \begin{equation}
      \rho_g(\sigma,w_i)
      = \frac{\alpha^{|\sigma_{w_{i}}|-|g|}(\alpha\lambda)^{2\psi}}{1+\lambda^{2\psi}}
      \leq \frac{\alpha^{|\sigma_{w_{i+1}}|-|g|-\psi}(\alpha\lambda)^{2\psi}}{1+\lambda^{2\psi}}
      \leq \frac{\alpha^{|\sigma_{w_{i+1}}|-|g|-3\psi}}{1+\lambda^{2\psi}}
      = -\alpha^{-3\psi}\rho_g(\sigma,w_{i+1}).
      \label{eq:nicepaths}
   \end{equation}
   Similarly, if $\sigma_{w_1}$ is decreasing but $\sigma_{w_\ell}$ is increasing, 
   then associate each increasing $\sigma_{w_i}$ with the decreasing edge $\sigma_{w_{i-1}}$, 
   and the lemma follows by an analogous argument as in~\eqref{eq:nicepaths}.
   
   It remains to establish the lemma when both $\sigma_{w_1}$ and $\sigma_{w_\ell}$ are increasing.
   In this case, let $j$ be such that $\sigma_{w_j}$ is the smallest edge of $\sigma$ with midpoint in $P$.
   Clearly, $\sigma_{w_j}$ must be an increasing edge. 
   The idea is to associate each increasing edge that is not $\sigma_{w_j}$ to a different decreasing edge, 
   and then split $\rho_g(\sigma,w_j)$ among all $\frac{\ell-1}{2}$ decreasing edges of $P$. 
   Since $\sigma_{w_j}$ is the smallest edge, and $\ell$ is large enough, this extra 
   addition can be controlled.
   Letting $\kappa=\frac{\ell-1}{2}\geq\frac{C'-C}{2C^2}-\frac{1}{2}$,
   we write
   \begin{align*}
      \sum_{x\in P\cap \finc^g(\sigma)} \rho_g(\sigma,x)
      &= \sum_{x\in P\cap \finc^g(\sigma)} \frac{\alpha^{|\sigma_x|-|g|}(\alpha\lambda)^{2\psi_x}}{1+\lambda^{2\psi_x}}\\
      &= \frac{\alpha^{|\sigma_{w_j}|-|g|}(\alpha\lambda)^{2\psi}}{1+\lambda^{2\psi}} 
         + \sum_{\text{odd $i\neq j$}} \frac{\alpha^{|\sigma_{w_i}|-|g|}(\alpha\lambda)^{2\psi}}{1+\lambda^{2\psi}}\\
      &\leq \sum_{\text{odd $i\neq j$}} \left(1+\frac{1}{\kappa}\right)\frac{\alpha^{|\sigma_{w_i}|-|g|}(\alpha\lambda)^{2\psi}}{1+\lambda^{2\psi}},
   \end{align*}
   where in the inequality we used that $\sigma_{w_j}$ is the smallest among all edges in $P$, and that there are $\kappa$ terms in the summation.
   Now for $i<j$, associate each increasing edge $\sigma_{w_i}$ to the decreasing edge $\sigma_{w_{i+1}}$, obtaining
   \begin{align}
      \left(1+\frac{1}{\kappa}\right)\frac{\alpha^{|\sigma_{w_{i}}|-|g|}(\alpha\lambda)^{2\psi}}{1+\lambda^{2\psi}}
      &\leq \left(1+\frac{1}{\kappa}\right)\frac{\alpha^{|\sigma_{w_{i+1}}|-\psi-|g|}(\alpha\lambda)^{2\psi}}{1+\lambda^{2\psi}}\nonumber\\
      &\leq \left(1+\frac{1}{\kappa}\right)\frac{\alpha^{|\sigma_{w_{i+1}}|-3\psi-|g|}}{1+\lambda^{2\psi}}\nonumber\\
      &\leq \frac{\alpha^{|\sigma_{w_{i+1}}|-2\psi-|g|}}{1+\lambda^{2\psi}}
      = -\alpha^{-2\psi}\rho_g(\sigma,w_{i+1}),
      \label{eq:toughpaths}
   \end{align}
   where in the last inequality we use that $1+1/\kappa\leq \alpha$ from the condition on $C'$ in~\eqref{eq:condcp}.
   For $i>j$, we associate each increasing edge $\sigma_{w_i}$ to the decreasing edge $\sigma_{w_{i-1}}$, and perform the same derivation as
   in~\eqref{eq:toughpaths}. This completes the proof of the lemma.
\end{proof}

%############################################################################################
\subsection{Expected change in terms of decreasing edges}\label{sec:decrease}
The next lemma puts together the results from Sections~\ref{sec:sleaves} and~\ref{sec:1dconfig} to show that 
the expected change in the Lyapunov function can be written as a function of only the decreasing edges, and only those that are large enough.
\begin{lemma}\label{lem:rootdominates}
   There exists a positive constant $c_1=c_1(\alpha,C,C')$ such that
   given any boundary condition $\xi\in\Xi(\Lambda^0)$, 
   any triangulation $\sigma\in\Omega^\xi$ and any ground state edge $g\in\mathbb{G}^\xi$, we have
   $$
      \E_\sigma\big(\Psi_g(\sigma')-\Psi_g(\sigma)\big)
      \leq \frac{c_1}{|\Lambda|}
           + \frac{(\alpha^2-1)}{2\alpha^2|\Lambda|} \sum_{z\in \fdec^g(\sigma)\setminus F_{|g|+C'}(\sigma)}\rho_g(\sigma,z).
   $$
   Consequently, if
   $
      \sum_{z\in \fdec^g(\sigma)\setminus F_{|g|+C'}(\sigma)}\rho_g(\sigma,z)
      \leq -\frac{4\alpha^2c_1}{\alpha^2-1},
   $
   we obtain
   $$
      \E_\sigma\big(\Psi_g(\sigma')-\Psi_g(\sigma)\big)
      \leq \frac{(\alpha^2-1)}{4\alpha^2|\Lambda|}\sum_{z\in \fdec^g(\sigma)\setminus F_{|g|+C'}(\sigma)}\rho_g(\sigma,z).
   $$
\end{lemma}
\begin{remark}\label{rem:rootdominates}
In Lemma~\ref{lem:rootdominates}, 
recall that $\rho_g(\sigma,z)<0$ for all $z\in\fdec^g(\sigma)$, which implies that 
$$
   \E_\sigma\big(\Psi_g(\sigma')-\Psi_g(\sigma)\big)
   \leq \frac{c_1}{|\Lambda|}
   \quad\text{for all $\sigma\in\Omega^\xi$}.
$$
\end{remark}
\begin{proof}[Proof of Lemma~\ref{lem:rootdominates}]
   Among the increasing edges, only the ones in $\finc^g(\sigma)$ can change the value of $\Psi_g(\sigma)$, cf.~\eqref{eq:simpledecomp}. 
   From Proposition~\ref{pro:squaregeom}, for any $x\in \finc^g(\sigma)$ that is not in ground state, we have that there is a 
   decreasing edge $\sigma_z$ 
   so that $x\in \tau(\sigma,z)$ and $\sigma_z$ intersects $g$. Consequently, $z\in \fdec^g(\sigma)$.
   Note that $\sigma_z$ is not a constraint edge and is not in ground state, 
   since no constraint edge can intersect $g$ and no decreasing edge can be in ground state.
   
   Let $K\subset \Lambda$ be defined as 
   $$
      K=\big\{x\in \fdec^g(\sigma) \cup \finc^g(\sigma) \colon \tau^{-1}(\sigma,x)\cap \fdec^g(\sigma)\neq \emptyset\big\}.
   $$
   By Proposition~\ref{pro:squaregeom}, $K$ includes all flippable edges that are not in ground state and either intersect $g$ 
   or will intersect $g$ after being flipped. 
   Let $w$ be the midpoint of $g$.
   For the ground state edges of $\sigma$, only the one with midpoint $w$ can intersect $g$.
   Therefore, from~\eqref{eq:simpledecomp}, we write
   \begin{align}
      &\E_\sigma\big(\Psi_g(\sigma')-\Psi_g(\sigma)\big)\nonumber\\
      &\leq \frac{\ind{\sigma_w \in\mathbb{G}^\xi, w\in \finc(\sigma)}}{|\Lambda|}\E_\sigma\big(\Psi_g(\tilde \sigma^w)-\Psi_g(\sigma)\big)
         + \frac{1}{|\Lambda|}\sum_{x\in K}\E_\sigma\big(\Psi_g(\tilde \sigma^x)-\Psi_g(\sigma)\big)\nonumber\\
      &\leq \frac{\ind{\sigma_w \in\mathbb{G}^\xi, w\in \finc(\sigma)}}{|\Lambda|} \alpha^{|\sigma_w|-|g|}(\alpha\lambda)^{2\psi_w} 
         + \frac{1}{|\Lambda|}\sum_{x\in K}\E_\sigma\big(\Psi_g(\tilde \sigma^x)-\Psi_g(\sigma)\big)\nonumber\\
      &\leq \frac{\ind{\sigma_w \in\mathbb{G}^\xi, w\in \finc(\sigma)}}{|\Lambda|} \alpha^{-2}
         + \frac{1}{|\Lambda|}\sum_{x\in K}\E_\sigma\big(\Psi_g(\tilde \sigma^x)-\Psi_g(\sigma)\big).
      \label{eq:condexp}
   \end{align}
   We split $K$ into two sets $K_\mathrm{big},K_\mathrm{small}$. 
   Define 
   $$
      K_\mathrm{big} 
      =\big\{x\in K \colon \forall z\in \tau^{-1}(\sigma,x) \cap \fdec^g(\sigma) \text{ we have } |\sigma_z|> |g|+C'\big\}.
   $$
   This set is related to the vertices of the graph $H$ in Section~\ref{sec:1dconfig}. Also define
   $$
      K_\mathrm{small} 
      = K \setminus K_\mathrm{big}.
   $$
   
   We start with $K_\mathrm{small}$. 
   Note that for all $x\in K_\mathrm{small}\cap \finc^g(\sigma)$, 
   there exists an edge $z\in \tau^{-1}(\sigma,x)\cap\fdec^g(\sigma)$ such that $|\sigma_z|\leq|g|+C'$. 
   Hence, for each $x\in K_\mathrm{small}\cap \finc^g(\sigma)$, we can associate one $z\in \tau^{-1}(\sigma,x)\cap \fdec^g(\sigma)$ such that 
   $|\sigma_x|< |\sigma_z|\leq |g|+C'$. 
   Thus,
   \begin{align}
      \frac{1}{|\Lambda|}\sum_{x\in K_\mathrm{small}}\E_\sigma\big(\Psi_g(\tilde \sigma^x)-\Psi_g(\sigma)\big)
      &\leq \frac{1}{|\Lambda|}\sum_{x\in K_\mathrm{small}\cap \finc^g(\sigma)}\E_\sigma\big(\Psi_g(\tilde \sigma^x)-\Psi_g(\sigma)\big)\nonumber\\
      &\leq \frac{1}{|\Lambda|}\sum_{x\in K_\mathrm{small}\cap \finc^g(\sigma)}\alpha^{|\sigma_x|-|g|-2\psi_x}
      \leq \frac{\alpha^{C'}}{|\Lambda|}\sum_{x\in K_\mathrm{small}\cap \finc^g(\sigma)}\alpha^{-2\psi_x}.\nonumber
      %&\leq \frac{2 c (|b_*|+C')^5|g|}{|\Lambda|}\alpha^{|b_*|+C'-2}.
      %\label{eq:small}
   \end{align}   
   Note that for each $x\in K_\mathrm{small}\cap \finc^g(\sigma)$, either $\sigma_x$ or $\sigma_x^x$ intersects $g$.
   In the first case, $\sigma_x$ belongs to the set $I_g(\sigma,C')$, which is the set of edges of $\sigma$ of size at most $|g|+C'$ that intersect $g$.
   By Proposition~\ref{pro:smalltrees}, we have $|I_g(\sigma,C')|\leq c_1{C'}^2$, for some positive constant $c_1$.
   For the case when $\sigma_x^x$ intersects $g$ but $\sigma_x$ does not, then one edge $\sigma_y$ in the same triangle as $\sigma_x$ must intersect $g$. Clearly, 
   $|\sigma_y|\leq |\sigma_x|+\psi_x \leq |g|+C'+\psi_x$, 
   giving that $\sigma_y\in I_g(\sigma,C'+\psi_x)$. 
   Since for each such edge $\sigma_y$ there are at most 
   four other edges in the same triangle as $\sigma_y$, we obtain
   \begin{align}
      \frac{1}{|\Lambda|}\sum_{x\in K_\mathrm{small}}\E_\sigma\big(\Psi_g(\tilde \sigma^x)-\Psi_g(\sigma)\big)
      &\leq \frac{\alpha^{C'}}{|\Lambda|}\left(\alpha^{-2}c_1{C'}^2+\sum_{i\geq 1}4|I_g(\sigma,C'+i)|\alpha^{-2i}\right) \nonumber\\
      &\leq \frac{\alpha^{C'}}{|\Lambda|}\left(\alpha^{-2}c_1{C'}^2+\sum_{i\geq 1}4c_1(C'+i)^2\alpha^{-2i}\right) \nonumber\\
      &\leq \frac{c_2}{|\Lambda|},
      \label{eq:small}
   \end{align}   
   for some positive constant $c_2=c_2(\alpha,C')$. An important feature of the bound above is that it does not depend on $|g|$.
   
   Now for $K_\mathrm{big}$ we have that 
   \begin{align}
      \frac{1}{|\Lambda|}\sum_{x\in K_\mathrm{big}}\E_\sigma\big(\Psi_g(\tilde \sigma^x)-\Psi_g(\sigma)\big)
      &\leq \frac{1}{|\Lambda|}\sum_{z \in \fdec^g(\sigma) \setminus F_{|g|+C'}(\sigma)}\E_\sigma\big(\Psi_g(\tilde \sigma^z)-\Psi_g(\sigma)\big)\nonumber\\
      &\quad+ \frac{1}{|\Lambda|}\sum_{z \in \fdec^g(\sigma) \setminus F_{|g|+C'}(\sigma)}\sum_{x\in\tau(\sigma,z)\cap \finc(\sigma)}
            \E_\sigma\big(\Psi_g(\tilde \sigma^x)-\Psi_g(\sigma)\big).
      \label{eq:kbig}
   \end{align}
   Using Lemmas~\ref{lem:largetreesmalledges} and~\ref{lem:largetree1d} we obtain
   \begin{align*}
      &\frac{1}{|\Lambda|}\sum_{z \in \fdec^g(\sigma) \setminus F_{|g|+C'}(\sigma)}\sum_{x\in\tau(\sigma,z)\cap \finc(\sigma)}
         \E_\sigma\big(\Psi_g(\tilde \sigma^x)-\Psi_g(\sigma)\big)\\
      &\leq -\left(5\alpha^{-C}+\frac{1}{\alpha^2}\right)\frac{1}{|\Lambda|}\sum_{z \in \fdec^g(\sigma) \setminus F_{|g|+C'}(\sigma)}
         \E_\sigma\big(\Psi_g(\tilde \sigma^z)-\Psi_g(\sigma)\big).
   \end{align*}
   Plugging this into~\eqref{eq:kbig}, and using that $1-5\alpha^{-C}-\alpha^{-2}=\frac{\alpha^2-1}{\alpha^2}-5\alpha^{-C}\leq\frac{\alpha^2-1}{2\alpha^2}$ we have
   \begin{align}
      \frac{1}{|\Lambda|}\sum_{x\in K_\mathrm{big}}\E_\sigma\big(\Psi_g(\tilde \sigma_x)-\Psi_g(\sigma)\big)
      &\leq \left(\frac{\alpha^2-1}{2\alpha^2}\right)\frac{1}{|\Lambda|}\sum_{z \in \fdec^g(\sigma) \setminus F_{|g|+C'}(\sigma)}
         \rho_g(\sigma,z).
      \label{eq:big}
   \end{align}
   Putting~\eqref{eq:small} and~\eqref{eq:big} together into~\eqref{eq:condexp} concludes the proof.
\end{proof}

%############################################################################################
\subsection{Long decreasing edges dominate the Lyapunov function}\label{sec:decdominate}
From Section~\ref{sec:decrease} we have that the change in the Lyapunov function can be written as a sum of $\rho_g(\sigma,x)$ over all $x$ such that 
$\sigma_x$ is decreasing and large enough. If this sum is small enough, then the Lyapunov function decreases in expectation. However, we want to write that the 
decrease in the Lyapunov function is a function of $\Psi_g(\sigma)$, the value of the function. 
The next two lemmas are used to establish this. They show that $\Psi_g(\sigma)$ can be written as a constant times a sum over decreasing edges.
\begin{lemma}\label{lem:treeroot}
   For any boundary condition $\xi\in\Xi(\Lambda^0)$, any triangulation $\sigma\in\Omega^\xi$ and any $z\in \fdec(\sigma)$, we have
   $$
      \sum_{x\in\tau(\sigma,z)} \alpha^{|\sigma_{x}|}
      \leq \left(\frac{\alpha+1}{\alpha-1}\right)\alpha^{|\sigma_z|}+10(C-1)\alpha^C|\sigma_z|.
   $$
   If $|\sigma_z|\geq C'$, the bound above simplifies to
   $$
      \sum_{x\in\tau(\sigma,z)} \alpha^{|\sigma_{x}|}
      \leq \left(\frac{2\alpha}{\alpha-1}\right)\alpha^{|\sigma_z|}.
   $$
\end{lemma}
\begin{proof}
   First, we decompose
   $$
       \sum_{x\in\tau(\sigma,z)} \alpha^{|\sigma_{x}|} = \sum_{x\in \tau(\sigma,z)\setminus F_C(\sigma)} \alpha^{|\sigma_x|}
         +\sum_{x\in \tau(\sigma,z)\cap F_C(\sigma)} \alpha^{|\sigma_x|}.
   $$
   For the edges that are not small (i.e., the first sum in the right-hand side above), we use the tree of influence. 
   If $\sigma_{w_1}, \sigma_{w_2}, \sigma_{w_3}$ form a triangle in $\sigma$ such that 
   $|\sigma_{w_1}|>|\sigma_{w_2}|\geq|\sigma_{w_3}|$,
   and we set $\delta=\frac{1}{\alpha-1}$, then 
   \begin{equation}
      \delta\alpha^{|\sigma_{w_1}|}
      \geq (1+\delta)\left(\alpha^{|\sigma_{w_2}|}\ind{w_2\not\in F_C(\sigma)}
         +\alpha^{|\sigma_{w_3}|}\ind{w_3\not\in F_C(\sigma)}\right).
      \label{eq:downtree}
   \end{equation}
   In order to see this, note that $|\sigma_{w_1}|=|\sigma_{w_2}|+|\sigma_{w_3}|$, which gives that
   \begin{align*}
      &(1+\delta)\left(\alpha^{|\sigma_{w_2}|}\ind{w_2\not\in F_C(\sigma)}
         +\alpha^{|\sigma_{w_3}|}\ind{w_3\not\in F_C(\sigma)}\right)\\
      &= (1+\delta)\alpha^{|\sigma_{w_1}|}\left(\alpha^{-|\sigma_{w_3}|}\ind{w_2\not\in F_C(\sigma)}
         +\alpha^{-|\sigma_{w_2}|}\ind{w_3\not\in F_C(\sigma)}\right)\\
      &\leq \alpha^{|\sigma_{w_1}|}\left(\frac{1+\delta}{\alpha}\right)
      = \delta \alpha^{|\sigma_{w_1}|}.
   \end{align*}
   Let $w_1,w_2,w_3,w_4$ be the children of $z$ in $\tau(\sigma,z)$. Iterating~\eqref{eq:downtree}, we obtain
   \begin{align}
      \sum_{x\in \tau(\sigma,z)\setminus F_C(\sigma)}\alpha^{|\sigma_x|}
      \leq \alpha^{|\sigma_z|}+(1+\delta)\left(\alpha^{|\sigma_{w_1}|}+\alpha^{|\sigma_{w_2}|}+\alpha^{|\sigma_{w_3}|}+\alpha^{|\sigma_{w_4}|}\right)
      \leq (1+2\delta)\alpha^{|\sigma_z|}.
      \label{eq:largeverts}
   \end{align}
   For the small edges, note that an edge of length $\ell$ crosses at most $\ell-1$ squares of $\mathbb{S}$; recall the definition of $\mathbb{S}$ from the paragraph 
   preceding~\eqref{eq:defs}. 
   Proposition~\ref{pro:tautree}\ref{it:treesquare} gives that for any $w\in\tau(\sigma,z)$ the descendants of $w$ are contained in $S(\sigma_w)$.
   Therefore, all descendants of $w$ in the tree $\tau(\sigma,z)$ must amount to 
   at most $5S(\sigma_w)$ midpoints since each square of $\mathbb{S}$ has $5$ midpoints of $\Lambda$. 
   Let $R\subset \tau(\sigma,z)$ be the set of midpoints in $\tau(\sigma,z)$ whose edge has length smaller than $C$ 
   and whose parent has length larger than $C$;
   if no such midpoint of $\tau(\sigma,z)$ satisfies this condition, set $R=\{z\}$.
   By definition, no midpoint of $R$ can be a descendant of another midpoint of $R$. Therefore, 
   $$
      \sum_{w\in R}|\sigma_w|\leq \sum_{w\in \tau_\leaves(\sigma,z)}|\sigma_w|\leq 2|\sigma_z|, 
   $$
   implying that the cardinality of $R$ is at most $2|\sigma_z|$.
%    Using this, if we chop each branch of the tree 
%    $\tau(\sigma,z)$ as soon as we find an edge of size smaller than $C$ (keeping that edge on the tree), 
%    there will remain at most $2|\sigma_z|$ edges of size smaller than $C$ in the tree. 
   Using this and the fact that the descendants of any $w\in R$ have length at most $|\sigma_w|\leq C$ we obtain
   $$
      \sum_{x\in \tau(\sigma,z)\cap F_C(\sigma)}\alpha^{|\sigma_x|}
      \leq \sum_{w\in R} 5 S(\sigma_w)\alpha^{|\sigma_w|}
      \leq \sum_{w\in R} 5(C-1)\alpha^C 
      \leq 10(C-1)|\sigma_z|\alpha^{C}.
   $$
   In addition, if $|\sigma_z|\geq C'$, we obtain
   $$
      \sum_{x\in \tau(\sigma,z)\cap F_C(\sigma)}\alpha^{|\sigma_x|}
      \leq 10(C-1)\alpha^{-C} |\sigma_z|\alpha^{2C}
      \leq \frac{10(C-1)\alpha^{-C}}{4} \alpha^{|\sigma_z|}
      \leq \alpha^{|\sigma_z|},
   $$
   where we used the condition on $C'$ from~\eqref{eq:condcp} in the second inequality, and the condition on $C$ from~\eqref{eq:condclarge} in the last inequality.
   Together with~\eqref{eq:largeverts}, this establishes the lemma.
\end{proof}

\begin{lemma}\label{lem:decreasingdominate}
   For any boundary condition $\xi\in\Xi(\Lambda^0)$, any triangulation $\sigma\in\Omega^\xi$, and any ground state edge $g\in\mathbb{G}^\xi$, 
   if $w$ is the midpoint of $g$, then 
   $$
      \Psi_g(\sigma) 
      \leq \ind{\sigma_w\in\mathbb{G}^\xi}
         + \frac{c{C'}^2\alpha^{C'+2}}{\alpha^2-1}
         + \left(\frac{2\alpha^3}{(\alpha-1)^2(\alpha+1)}\right)\sum_{z\in \fdec^g(\sigma) \setminus F_{|g|+C'}(\sigma)} \alpha^{|\sigma_z|-|g|},
   $$
   where $c$ is the constant in Proposition~\ref{pro:smalltrees}\ref{it:alltriang}.
\end{lemma}
\begin{proof}
   Since a decreasing flip decreases the length of an edge by at least $2$, we have for any given midpoint $x\in\Lambda$ that
   $$
      \sum_{e \in E_x^\xi(\sigma,g)} \alpha^{|e|}
      \leq \sum_{j=0}^{|\sigma_x|/2} \alpha^{|\sigma_x|-2j}
      \leq \left(\frac{\alpha^2}{\alpha^2-1}\right)\alpha^{|\sigma_x|}.
   $$
   (Recall the definition of $E_x^\xi(\sigma,g)$ from~\eqref{eq:defe}.)
   We decompose $\Psi_g(\sigma)$ using Proposition~\ref{pro:squaregeom}, which gives that all edges intersecting $g$ 
   must either be in ground state or be in trees rooted at edges that also intersect $g$.
   Letting $w$ be the midpoint of $g$, we obtain
   \begin{align}
     \Psi_g(\sigma) 
     &\leq \ind{\sigma_w\in\mathbb{G}^\xi}
         + \sum_{z\in \fdec^g(\sigma)}\sum_{x\in \tau(\sigma,z)} \sum_{e\in E_x^\xi(\sigma,g)}\alpha^{|e|-|g|}\nonumber\\
     &\leq \ind{\sigma_w\in\mathbb{G}^\xi}
         + \frac{\alpha^2}{\alpha^2-1}
              \sum_{z\in \fdec^g(\sigma)}\sum_{x\in \tau(\sigma,z) \colon \sigma_x \cap g \neq \emptyset} \alpha^{|\sigma_x|-|g|}.
      %&= \ind{\sigma_w\in\mathbb{G}^\xi}
      %   + \frac{\alpha^2}{\alpha^2-1}
      %        \sum_{z\in \fdec^g(\sigma)\setminus F_{|g|+C'}(\sigma)}\sum_{x\in \tau(\sigma,z) \colon \sigma_x \cap g \neq \emptyset} \alpha^{|\sigma_x|-|g|}\nonumber\\
      %&\quad+ \frac{\alpha^2}{\alpha^2-1}
      %          \sum_{z\in \fdec^g(\sigma)\cap F_{|g|+C'}(\sigma)}\sum_{x\in \tau(\sigma,z) \colon \sigma_x \cap g \neq \emptyset} \alpha^{|\sigma_x|-|g|}.
      \label{eq:twoterms}
   \end{align}
   We consider two cases, first when $z\in \fdec^g(\sigma)\cap F_{|g|+C'}(\sigma)$ and then when $z\in \fdec^g(\sigma)\setminus F_{|g|+C'}(\sigma)$.
   For the first case, 
   since all edges $\sigma_x$ in that sum are smaller than $|g|+C'$ and intersect $g$, Proposition~\ref{pro:smalltrees}\ref{it:alltriang} yields
   \begin{equation}
      \sum_{z\in \fdec^g(\sigma)\cap F_{|g|+C'}(\sigma)}\sum_{x\in \tau(\sigma,z) \colon \sigma_x \cap g \neq \emptyset} \alpha^{|\sigma_x|-|g|}
      \leq c{C'}^2\alpha^{C'}.
      \label{eq:notecb}
   \end{equation}
   For the second case, we apply Lemma~\ref{lem:treeroot} to obtain
   \begin{equation}
      \sum_{z\in \fdec^g(\sigma)\setminus F_{|g|+C'}(\sigma)}\sum_{x\in \tau(\sigma,z) \colon \sigma_x \cap g \neq \emptyset} \alpha^{|\sigma_x|-|g|}
      \leq  \frac{2\alpha}{\alpha-1}\sum_{z\in \fdec^g(\sigma)\setminus F_{|g|+C'}(\sigma)}\alpha^{|\sigma_z|-|g|}.
      \label{eq:withecb}
   \end{equation}
   Plugging~\eqref{eq:notecb} and~\eqref{eq:withecb} into~\eqref{eq:twoterms} concludes the proof.
\end{proof}

%############################################################################################
\subsection{Finishing the proof}\label{sec:finish}
\begin{proof}[Proof of Theorem~\ref{thm:lyapunov}]
   Using Lemma~\ref{lem:decreasingdominate} we have 
   \begin{equation}
      \sum_{z\in \fdec^g(\sigma)\setminus F_{|g|+C'}(\sigma)}\alpha^{|\sigma_z|-|g|}
      \geq \left(\frac{(\alpha-1)^2(\alpha+1)}{2\alpha^3}\right)
         \left(\Psi_g(\sigma)-\ind{\sigma_w\in\mathbb{G}} - \frac{c{C'}^2\alpha^{C'+2}}{\alpha^2-1}\right).
      \label{eq:boundrewritten}
   \end{equation}
   Using Lemma~\ref{lem:rootdominates}, since $\rho_g(\sigma,z) \leq -\alpha^{|\sigma_z|-|g|}/2$ we have that if 
   $$
      \left(\frac{(\alpha-1)^2(\alpha+1)}{2\alpha^3}\right)
         \left(\Psi_g(\sigma)-\ind{\sigma_w\in\mathbb{G}} - \frac{c{C'}^2\alpha^{C'+2}}{\alpha^2-1}\right)
      \geq \frac{8\alpha^2c_1}{\alpha^2-1},
   $$
   then there is a contraction on the expected change of $\Psi_g(\sigma)$, where $c_1$ is the constant in Lemma~\ref{lem:rootdominates}.
   We can set a constant $\psi_0>1$ such that 
   $$
      \psi_0 > 1+\frac{c{C'}^2\alpha^{C'+2}}{\alpha^2-1}+\frac{16\alpha^5c_1}{(\alpha-1)^3(\alpha+1)^2}
   $$
   and 
   $$
      \Psi_g(\sigma)-\ind{\sigma_w\in\mathbb{G}} - \frac{c{C'}^2\alpha^{C'+2}}{\alpha^2-1}\geq \frac{\Psi_g(\sigma)}{2}
      \quad\text{for all $\Psi_g(\sigma)\geq\psi_0$}.
   $$
   Then, whenever $\Psi_g(\sigma)\geq \psi_0$ there is a contraction in the Lyapunov function, and 
   using the second statement in Lemma~\ref{lem:rootdominates} and~\eqref{eq:boundrewritten} we have
   \begin{align*}
      \E_\sigma\big(\Psi_g(\sigma')-\Psi_g(\sigma)\big)
      &\leq \left(\frac{\alpha^2-1}{4\alpha^2|\Lambda|}\right) \sum_{z\in\fdec^g(\sigma)\setminus F_{|g|+C'}(\sigma)}\rho_g(\sigma,z)\\
      &\leq - \left(\frac{\alpha^2-1}{4\alpha^2|\Lambda|}\right) \sum_{z\in\fdec^g(\sigma)\setminus F_{|g|+C'}(\sigma)}\alpha^{|\sigma_z|-|g|}/2\\
      &\leq - \left(\frac{\alpha^2-1}{4\alpha^2|\Lambda|}\right) \left(\frac{(\alpha-1)^2(\alpha+1)}{4\alpha^3}\right)\frac{\Psi_g(\sigma)}{2}.
   \end{align*}
   Setting $\epsilon=\frac{(\alpha-1)^3(\alpha+1)^2}{32\alpha^5}$ concludes the proof.
\end{proof}

%############################################################################################
%############################################################################################
%############################################################################################
\section{Direct consequences of the Lyapunov function}\label{sec:consequence}

Throughout this section, we fix an arbitrary boundary condition $\xi\in\Xi(\Lambda^0)$, an arbitrary triangulation $\sigma\in\Omega^\xi$ and an arbitrary
ground state edge $g\in\mathbb{G}^\xi$, and we let $\alpha$, $\psi_0$ and $\epsilon$ refer to the constants in Theorem~\ref{thm:lyapunov}.
Since $\alpha$, $\psi_0$ and $\epsilon$ all depend on $\lambda$, in the results below we omit dependences on 
$\alpha$, $\psi_0$ and $\epsilon$ and highlight only dependences on $\lambda$.
% For simplicity we will drop the dependence on $\xi$ and $g$ from the notation.
Let 
$\sigma=\sigma^0,\sigma^1,\sigma^2,\ldots$ be a sequence of triangulations obtained from the Markov chain $\mathcal{M}_\sigma^\lambda(\Omega^\xi)$, 
the edge-flipping Glauber dynamics with parameter $\lambda$,
state space $\Omega^\xi$, and initial configuration $\sigma$.
Define 
$$
   \Omega_\good
   = \Omega_\good^\xi
   = \left\{\eta\in\Omega^\xi \colon \Psi_g(\eta)\leq \psi_0\right\}.
$$
Theorem~\ref{thm:lyapunov} establishes that $\Psi_g(\eta)$ contracts in expectation for all $\eta\not\in\Omega_\good$.

We denote by $\pi=\pi^\xi$ the stationary measure of $\mathcal{M}^\lambda(\Omega^\xi)$, see~\eqref{eq:defpi}. For any function 
$f\colon \Omega^\xi\to\mathbb{R}$, we denote by $\pi(f)$ the expectation of $f$ with respect to $\pi$.
The first proposition establishes that if the initial configuration $\sigma$ does not belong to $\Omega_\good$, then very quickly the Glauber dynamics 
enters the set $\Omega_\good$.
\begin{proposition}
   Fix any boundary condition $\xi\in\Xi(\Lambda^0)$, any initial triangulation $\sigma\in\Omega^\xi$ and any ground state edge $g\in\mathbb{G}^\xi$.
   If $T=\min\{t\geq 0 \colon \sigma^t \in \Omega_\good\}$, then 
   \begin{equation}
      \E_\sigma\big((1+\epsilon/|\Lambda|)^T\big) \leq \Psi_g(\sigma).
      \label{eq:momentt}
   \end{equation}
   Consequently, there exist a constant $\ell_0=\ell_0(\lambda)>0$ so that for any $\ell\geq \ell_0$, we have
   \begin{equation}
      \PR_\sigma\left(T \geq \ell |\Lambda| +\ell_0|\Lambda|\log\left(\Psi_g(\sigma)\right)\right) 
      \leq \exp\left(-\ell/\ell_0\right).
      \label{eq:tailt}
   \end{equation}
\end{proposition}
\begin{proof}
   For all $t\geq 0$, define the random variable 
   $$
      X_t = \Psi_g(\sigma^{t\land T})(1+\epsilon/|\Lambda|)^{t\land T}.
   $$ 
   Letting $\mathcal{F}_t$ be the $\sigma$-algebra generated by $\sigma_0,\sigma_1,\ldots,\sigma_t$ , we have for all $t\geq 1$ for which 
   $T\geq t$ that
   $$
      \E_\sigma\big(X_t \mid \mathcal{F}_{t-1}\big) 
      \leq \left(1-\epsilon/|\Lambda|\right)\Psi_g(\sigma^{t-1})(1+\epsilon/|\Lambda|)^{t}
      \leq (1+\epsilon/|\Lambda|)^{t-1}\Psi_g(\sigma^{t-1})
      = X_{t-1}.
   $$
   Consequently, $X_t$ is a supermartingale, which gives that 
   \begin{equation}
      \E_\sigma\big(X_t\big) \leq X_0 = \Psi_g(\sigma) \quad\text{for all $t\geq 0$}
      \label{eq:xta}
   \end{equation}
   and 
   \begin{equation}
      \E_\sigma\big(X_t\big) \geq \E_\sigma\big((1+\epsilon/|\Lambda|)^{t\land T}\big)
      \label{eq:xtb}
   \end{equation}
   since $\Psi_g(\sigma)\geq 1$ for all $\sigma\in\Omega^\xi$. 
   Plugging the bound on $\E_\sigma(X_t)$ from~\eqref{eq:xta} into~\eqref{eq:xtb}, and taking the limit as $t\to\infty$ in~\eqref{eq:xtb} 
   establishes~\eqref{eq:momentt}.
   
   The statement in~\eqref{eq:tailt} is a simple application of Chernoff's inequality using~\eqref{eq:momentt}, which gives
   \begin{align*}
      \PR_\sigma\left(T \geq \ell |\Lambda| +\ell_0|\Lambda|\log\left(\Psi_g(\sigma)\right)\right) 
      &\leq  \frac{\E_\sigma\big((1+\epsilon/|\Lambda|)^T\big)}
                  {(1+\epsilon/|\Lambda|)^{\ell |\Lambda| +\ell_0|\Lambda|\log\left(\Psi_g(\sigma)\right)}}\\
      &\leq \frac{\Psi_g(\sigma)}
                 {(1+\epsilon/|\Lambda|)^{\ell |\Lambda| +\ell_0|\Lambda|\log\left(\Psi_g(\sigma)\right)}}.
   \end{align*}
   The result follows for all $\ell\geq\ell_0$ where $\ell_0$ is the smallest positive number such that 
   $(1+\epsilon/x)^{\ell_0 x}\geq e$ for all $x\geq 1$.
\end{proof}

\begin{proposition}\label{pro:expectation}
   Fix any boundary condition $\xi\in\Xi(\Lambda^0)$, any initial triangulation $\sigma\in\Omega^\xi$ and any ground state edge $g\in\mathbb{G}^\xi$.
   There exists a constant $c=c(\lambda)>0$ such that for any $t\geq 1$
   we have
   $$
      \E_\sigma\big(\Psi_g(\sigma^t)\big) 
      \leq \max\left\{\Psi_g(\sigma)\left(1-\frac{\epsilon}{|\Lambda|}\right)^t, c\right\}.
   $$
   Consequently,
   $
      \pi(\Psi_g) \leq c.
   $
\end{proposition}
\begin{proof}
   Let $\mu$ be any distribution on triangulations, and let $\mu'$ be the distribution after one step of the Markov chain starting from a random triangulation 
   distributed according to $\mu$.
   For any function $f\colon \Omega^\xi\to\mathbb{R}$ and any $\Omega'\subset\Omega^\xi$, we denote by $\mu(f; \Omega')$ the expectation of $f$ with respect to 
   $\mu$ under the set $\Omega'$; formally, $\mu(f; \Omega')=\sum_{\eta\in\Omega'}f(\eta)\mu(\eta)$.
   Using this notation, we write
   \begin{align}
      \mu(\Psi_g) 
      = \mu\big(\Psi_g; \Omega_\good\big) + \mu\big(\Psi_g; \Omega_\good^\compl\big)
      \leq \psi_0\mu(\Omega_\good) + \mu\big(\Psi_g; \Omega_\good^\compl\big).
      \label{eq:fstep}
   \end{align}
   Then, for $\eta,\eta'\in\Omega^\xi$, letting $p(\eta,\eta')$ be the probability that the Markov chain moves from $\eta$ to $\eta'$ in one transition, we write
   \begin{align*}
      \mu'(\Psi_g) 
      = \sum_{\eta,\eta'} \mu(\eta)p(\eta,\eta')\Psi_g(\eta')
      \leq \sum_{\eta\in \Omega_\good}\mu(\eta) \left( \Psi_g(\eta) + \frac{c_1}{|\Lambda|}\right)
         + \sum_{\eta\in \Omega_\good^\compl}\mu(\eta)\left(1-\frac{\epsilon}{|\Lambda|}\right) \Psi_g(\eta).
   \end{align*}
   where $c_1$ is the constant from Lemma~\ref{lem:rootdominates} (see Remark~\ref{rem:rootdominates}). 
   Hence,
   \begin{align*}
      \mu'(\Psi_g) 
      &\leq \mu(\Psi_g)+ \frac{c_1}{|\Lambda|}\mu\big(\Omega_\good\big)
         -\frac{\epsilon}{|\Lambda|}\mu\big(\Psi_g; \Omega_\good^\compl\big).
   \end{align*}
   Applying the lower bound
   on $\mu\big(\Psi_g; \Omega_\good^\compl\big)$ from~\eqref{eq:fstep} we obtain
   $$
      \mu'(\Psi_g) 
      \leq \left(1-\frac{\epsilon}{|\Lambda|}\right)\mu(\Psi_g)
         +\left(\frac{c_1+\epsilon \psi_0}{|\Lambda|}\right)\mu(\Omega_\good).
   $$
   Fix any initial triangulation $\sigma^0=\sigma\in\Omega^\xi$, and consider the sequence $\{X_t\}_t$ where $X_t=\E_\sigma\big(\Psi_g(\sigma^t)\big)$. 
   Clearly, $X_t$ is deterministic given $\sigma$, and the equation above gives that
   \begin{align*}
      X_t 
      &\leq \left(1-\frac{\epsilon}{|\Lambda|}\right)X_{t-1}+\frac{c_1+\epsilon \psi_0}{|\Lambda|}\\
      &\leq \left(1-\frac{\epsilon}{|\Lambda|}\right)^t X_{0} +\frac{c_1+\epsilon \psi_0}{|\Lambda|}\sum_{i=0}^{t-1}\left(1-\frac{\epsilon}{|\Lambda|}\right)^i
      \leq \left(1-\frac{\epsilon}{|\Lambda|}\right)^t X_{0} +\frac{c_1+\epsilon \psi_0}{\epsilon},
   \end{align*}
   for all $t$, which implies the proposition.
\end{proof}

The following two simple propositions establish that if a triangulation $\sigma$ is such that $\Psi_g(\sigma)$ is small, 
then the largest edge of $\sigma$ intersecting $g$ and the number of edges of $\sigma$ intersecting 
$g$ are both small.
\begin{proposition}[Largest intersection]\label{pro:largestedge}
   Given any boundary condition $\xi\in\Xi(\Lambda^0)$, any triangulation $\sigma\in\Omega^\xi$ and any ground state edge $g\in\mathbb{G}^\xi$, 
   the largest edge of $\sigma$ that intersects $g$ has length at most 
   $|g|+\frac{\log \Psi_g(\sigma)}{\log \alpha}$.
\end{proposition}
\begin{proof}
   If an edge $\sigma_x$ intersects $g$, then $\Psi_g(\sigma)\geq \alpha^{|\sigma_x|-|g|}$, which establishes the lemma.
\end{proof}

\begin{proposition}[Number of intersections]\label{pro:numbercross}
   Given any boundary condition $\xi\in\Xi(\Lambda^0)$, any triangulation $\sigma\in\Omega^\xi$ and any ground state edge $g\in\mathbb{G}^\xi$, 
   the number of edges of $\sigma$ that intersect $g$ is at most 
   $\Psi_g(\sigma)$.
\end{proposition}
\begin{proof}
   If $\Gamma\subset\Lambda$ are the midpoint of the edges of $\sigma$ intersecting $g$, we obtain
   $
      \Psi_g(\sigma) 
      \geq \sum_{x\in \Gamma} \alpha^{|\sigma_x|-|g|} 
      \geq |\Gamma|,
   $
   where the last step follows since, by Proposition~\ref{pro:monotonicity}, if $\sigma_x$ intersects $g$ then $|\sigma_x|\geq |g|$.
\end{proof}

%############################################################################################
%############################################################################################
%############################################################################################
\section{Applications of the Lyapunov function}\label{sec:application}
%############################################################################################
\subsection{Tightness of local measures}\label{sec:tightness}
To avoid a cumbersome statement of the theorem below, in this section we only consider the special case where 
$\Lambda^0$ is the $n\times n$ rectangle $[-n/2,n/2]^2\cap \mathbb{Z}^2$, which we denote by $\Lambda^0_n$.
For $\Lambda^0_n$, 
let $\Omega_{n}$ denote the set of triangulations with vertices in $\Lambda^0_n$, 
let $\Lambda_{n}$ be the set of midpoints of the edges of some triangulation in $\Omega_n$, 
and let $\pi_n$ be the stationary measure over triangulations in $\Omega_n$ with parameter $\lambda$.
Consider that $\xi$ is the \emph{free boundary condition}; i.e., $\xi$ only contains the horizontal and vertical edges that form 
the boundary of the $n\times n$ rectangle. To emphasize this, we will drop $\xi$ from the notation.

Here we want to study how the configuration of edges inside a fixed neighborhood around the origin behaves as $n$ goes to infinity.
In particular, does the measure over such local configurations converge as $n\to\infty$? 
Here we will show via the Lyapunov function that these local measures 
are tight as $n\to\infty$.

For any $k>0$, let $\Upsilon_k=[-k/2,k/2]^2\cap\Lambda_n$ be the set of midpoints inside $[-k/2,k/2]^2$.
Let $\Gamma_k$ be the set of configurations of disjoint edges of midpoint in $\Upsilon_k$ such that for any 
$\gamma\in\Gamma_k$ there exists at 
least one $n$ and one triangulation $\sigma\in \Omega_n$ such that the edges in $\gamma$ are the edges of $\sigma_{\Upsilon_k}$, the edges of 
$\sigma$ whose midpoints lie in $\Upsilon_k$.
More formally, 
$$
   \Gamma_k = \bigcup_{n\geq k} \{\sigma_{\Upsilon_k} \colon \sigma\in\Omega_n\}.
$$
Finally, let $\pi_n^k$ be the stationary measure over triangulations of $\Omega_n$ of the edges of midpoint in $\Upsilon_k$. 
More precisely, for any $\gamma\in\Gamma_k$, we have
$$
   \pi_n^k(\gamma) = \frac{1}{Z_n^k}\sum_{\sigma\in\Omega_n\colon \sigma_{\Upsilon_k}=\gamma} \pi_n(\sigma),
$$
where $Z_n^k$ is a normalizing constant to make $\pi_n^k$ a probability measure over $\Gamma_k$. 

\begin{theorem}\label{thm:tightness}
   For any $\lambda\in(0,1)$, $\pi_n^k$ is a tight measure. 
\end{theorem}
\begin{proof}
   Let $H\subset\mathbb{G}$ be the set of vertical and horizontal ground state edges forming the (outer) boundary of 
   $\Upsilon_k$. 
   Using Proposition~\ref{pro:largestedge}, the largest 
   edge of a triangulation $\sigma\in\Omega_{n}$ that intersects the boundary of $[-k/2,k/2]^2$ has length at most 
   $$
      \max\left\{1+\frac{\log \Psi_g(\sigma)}{\log \alpha} \colon g\in H\right\}
      \leq \sum_{g\in H}\left(1+\frac{\log \Psi_g(\sigma)}{\log \alpha}\right).
   $$
   Therefore, taking expectation over $\sigma\in\Omega_{n}$ according 
   to the stationary measure $\pi_n$, we have that the expected value for the largest edge crossing an edge of $H$ is at most 
   $$
      \sum_{g\in H} \left(1+\frac{\pi\big(\log \Psi_g(\sigma)\big)}{\log \alpha}\right)
      \leq \sum_{g\in H} \left(1+\frac{\log c}{\log \alpha}\right)
      \leq 8(k+1)\left(1+\frac{\log c}{\log \alpha}\right),
   $$
   where the inequality follows by Jensen's inequality and Proposition~\ref{pro:expectation}. 
   Since the bound above does not depend on $n$, 
   Markov's inequality gives that for any $\delta>0$ there exists $L$ such that with probability at least $1-\delta$, a triangulation 
   $\sigma$ distributed as $\pi_n$ is such that the edges in $\sigma_{\Upsilon_k}$ are contained inside $[-L/2,L/2]^2$. 
   Since $L$ does not depend on $n$, the tightness of $\pi_n^k$ is established.
\end{proof}

A consequence of the proposition above is that $\pi_n^k$ has subsequential limits. 
An interesting open problem is whether the limit is \emph{unique}.

%############################################################################################
\subsection{Ground state probability}
The theorem below establishes that the probability that the edge of a given midpoint is in ground state, 
given any boundary condition, is bounded away from zero by a constant independent of the boundary condition. 
\begin{theorem}\label{thm:surgery}
   Fix any boundary condition $\xi\in\Xi(\Lambda^0)$, and any ground state edge $g\in\mathbb{G}^\xi$.
   For any $\lambda\in(0,1)$, there exists a positive constant $\delta=\delta(\lambda)$ such that if $\sigma$ is a random triangulation 
   distributed according to $\pi^\xi$, we have
   $$
      \pi^\xi(g\in \sigma) \geq \delta.
   $$
\end{theorem}
\begin{proof}
   Let $x\in\Lambda$ be the midpoint of $g$. 
   Proposition~\ref{pro:expectation} gives that 
   $
      \pi^\xi(\Psi_{g}) \leq c_2,
   $
   for some constant $c_2$.
   Let $\tilde\Omega^\xi\subset\Omega^\xi$ be the set of triangulations such that 
   $\eta\in\Omega^\xi$ belongs to $\tilde\Omega^\xi$ if and only if $\Psi_{g}(\eta) \leq 2c_2$. 
   By Markov's inequality we have
   $$
      \pi^\xi(\tilde \Omega^\xi)\geq 1/2.
   $$
   Let 
   $\Omega_{g}^\xi\subset\tilde\Omega^\xi$ be the set of triangulations $\eta$ for which $\eta_x=g$.
   We define a mapping $\phi\colon\tilde\Omega^\xi\to\Omega_{g}^\xi$ as follows. From $\eta\in\tilde\Omega^\xi$, 
   construct $\phi(\eta)$ by using the sequence of 
   triangulations from Proposition~\ref{pro:producegroundstate}. Since this sequence is obtained by performing only non-increasing flips, we obtain that
   $$ 
      |\phi(\eta)_y|\leq |\eta_y|
      \text{ for all $y\in\Lambda$, and consequently, }
      \pi^\xi(\phi(\eta))\geq\pi^\xi(\eta). 
   $$
   Also, Proposition~\ref{pro:producegroundstate} gives that in this sequence we only flip edges that intersect $g$.
   
   Let $\ell_g(\eta)$ be the length of the largest edge of $\eta$ intersecting $g$.
   Proposition~\ref{pro:largestedge} gives that
   $$
      \ell_g(\eta) 
      \leq |g|+\frac{\log \Psi_g(\eta)}{\log \alpha}
      \leq |g| + \frac{\log (2c_2)}{\log \alpha}
      \quad \text{for all $\eta\in\tilde\Omega^\xi$}.
   $$
   Using Proposition~\ref{pro:smalltrees}\ref{it:thistriang} 
   we obtain that the midpoints of the edges crossing $g$ in any $\eta\in\tilde\Omega^\xi$ must belong to a ball centered at $x$ of radius 
   $\frac{2\log (2c_2)}{\log \alpha}$.
   Hence, there exists a constant $c_3=c_3(\lambda)$ such that the number of midpoints of $\Lambda$ inside this ball is at most $c_3$.
   This implies that the number of different triangulations $\eta\in\tilde\Omega^\xi$ that map to the same $\phi(\eta)$ is 
   (by Anclin's bound, Lemma~\ref{lem:anclin}) at most $2^{c_3}$.
   
   For each $\tau \in \Omega_{g}^\xi$, 
   denote by $\kappa(\tau)$ the number of triangulations $\eta\in\tilde\Omega^\xi$ for which $\phi(\eta)=\tau$.
   Using all that, we have
   \begin{align*}
      \pi^\xi(\Omega_{g}^\xi)
      = \sum_{\sigma\in\Omega_{g}^\xi} \pi^\xi(\sigma)
      \geq \sum_{\sigma\in\Omega_g^\xi} \sum_{\tau\in\tilde\Omega^\xi\colon \phi(\tau)=\sigma}\frac{\pi^\xi(\tau)}{\kappa(\sigma)}
      \geq \sum_{\tau\in\tilde\Omega^\xi} \frac{\pi^\xi(\tau)}{2^{c_3}}
      = 2^{-c_3}\pi^\xi(\tilde\Omega^\xi)
      \geq 2^{-c_3-1}.
   \end{align*}
   Since $\pi^\xi(g\in \sigma)\geq \pi^\xi(\Omega_g^\xi)$, the theorem follows.
\end{proof}

%############################################################################################
\subsection{Exponential decay of edge length}
\begin{theorem}[Tail of intersecting ground state]\label{thm:crosstail}
   Let $\xi\in\Xi(\Lambda^0)$ be any boundary condition, and $g\in\mathbb{G}^\xi$ be a ground state edge of midpoint $x\in\Lambda$.
   Fix any $\lambda\in(0,1)$.
   Let $\sigma\in\Omega^\xi$ be any initial triangulation, and $\sigma^0=\sigma,\sigma^1,\sigma^2,\ldots$ 
   be a sequence of triangulations obtained 
   by the Markov chain $\mathcal{M}_\sigma^\lambda(\Omega^\xi)$. 
   There exists a positive constant $c=c(\lambda)$ such that 
   for any $t\geq 1$ and any $\ell\geq 0$, we have 
   $$
      \PR_\sigma\left(\bigcup_{y\in \Lambda}\big\{\sigma^t_y \cap g\neq \emptyset\big\} \cap \big\{|\sigma^t_y| \geq |g|+\ell\big\}\right)
      \leq \left(\Psi_g(\sigma)\left(1-\frac{\epsilon}{|\Lambda|}\right)^t \lor c\right)\alpha^{-\ell}.
   $$
\end{theorem}
\begin{proof}
   The proof is a simple application of Markov's inequality to Proposition~\ref{pro:expectation}, which gives that the left-hand side above is at most
   \begin{align*}
      \PR_\sigma\big(\Psi_g(\sigma^t) \geq \alpha^{\ell}\big)
      \leq \E_\sigma\big(\Psi_g(\sigma^t)\big)\alpha^{-\ell}
      \leq \left(\Psi_g(\sigma)\left(1-\frac{\epsilon}{|\Lambda|}\right)^t \lor c\right)\alpha^{-\ell}.
   \end{align*}
\end{proof}

The following is a direct consequence of Theorem~\ref{thm:crosstail}
\begin{corollary}[Tail at a given time]\label{cor:tailedge}
   Let $\xi\in\Xi(\Lambda^0)$ be any boundary condition, $x\in\Lambda$ be any midpoint, and $g\in\mathbb{G}^\xi$ be a ground state edge of midpoint $x$.
   Fix any $\lambda\in(0,1)$.
   Let $\sigma\in\Omega^\xi$ be any initial triangulation, and $\sigma^0=\sigma,\sigma^1,\sigma^2,\ldots$ be a sequence of triangulations obtained 
   by the Markov chain $\mathcal{M}_\sigma^\lambda(\Omega^\xi)$. 
   There exists a positive constant $c=c(\lambda)$ such that 
   for any $t\geq 1$ and any $\ell\geq 0$, we have 
   $$
      \PR_\sigma\big(|\sigma^t_x| \geq |g|+\ell\big)
      \leq \left(\Psi_g(\sigma)\left(1-\frac{\epsilon}{|\Lambda|}\right)^t \lor c\right)\alpha^{-\ell}.
   $$
   Consequently, 
   taking the limit as $t\to\infty$, we obtain for a random triangulation $\sigma$ distributed according to 
   $\pi^\xi$ that
   $$
      \pi^\xi(|\sigma_x| \geq |g| + \ell) \leq c\alpha^{-\ell}.
   $$
\end{corollary}

\subsection{Crossings of small triangles}
To simplify the statement of the theorem below, consider that $\Lambda^0_n$ is the set of integer points 
inside $[-n/2,n/2]^2$. Let $\Lambda_n$ be the set of midpoints of the edges of a triangulation of $\Lambda^0_n$.
Let $\xi\in\Xi(\Lambda^0_n)$ be any boundary condition, $\Omega^\xi_n$ be the set of triangulations of $\Lambda^0_n$ consistent with $\xi$, and 
$\pi_n^\xi$ be the stationary measure of $\mathcal{M}^\lambda(\Omega^\xi_n)$.
Let $\sigma\in\Omega^\xi_n$ be any triangulation. 
We say that two triangles of $\sigma$ are \emph{adjacent} if they share an edge.

First consider the case of $\xi$ being the free boundary condition.
The theorem below establishes that, with high probability, a triangulation sampled from $\pi_n^\xi$ 
contains a left-to-right crossing
of adjacent triangles, where all edges in the triangles are smaller than some constant depending only on $\lambda$.
\begin{theorem}\label{thm:smalltriangles}
   Let $\xi$ be the free boundary condition, and fix any $\lambda\in(0,1)$.
   Let $\sigma$ be a random triangulation distributed according to $\pi^\xi_n$. 
   Let $C(\sigma,L)$ be the event that there exists a path of adjacent triangles in $\sigma$ that 
   intersects both the left and right boundaries of $\Lambda^0_n$ and such that 
   the edges of the triangles have length at most $L$. 
   Then, there exist positive constants $c=c(\lambda)$ and $L_0=L_0(\lambda)$
   such that for all $L\geq L_0$ we have
   $$
      \pi^\xi_n(C(\sigma,L)) \geq 1 - \exp\left(-c n\right).
   $$
\end{theorem}

The theorem above is a consequence of the following, more general theorem which establishes that crossings of small triangles 
occur even in thin slabs inside $\Lambda^0_n$.
\begin{theorem}\label{thm:smalltrianglesgen}
   Let $R$ be a $m\times k$ rectangle inside $\Lambda^0_n$, where $m=m(n)\geq k=k(n)$.
   Consider any boundary condition $\xi\in\Xi(\Lambda^0_n)$ that does not intersect $R$, and take any $\lambda\in(0,1)$.
   Let $\sigma$ be a random triangulation distributed according to $\pi^\xi_n$. 
   For any $L\geq 1$, let $C_R(\sigma,L)$ be the event that there exists a path of adjacent triangles in $\sigma$ that 
   intersect both the left and right boundaries of $R$ and such that 
   the edges of the triangles are contained inside $R$ and have length at most $L$. 
   Then, there exist positive constants $c=c(\lambda),c'=c'(\lambda)$, $L_0=L_0(\lambda)$ and $m_0=m_0(\lambda,L)$ 
   such that for all $m\geq m_0$, all $k\geq c'\log(m)$ and all $L\geq L_0$ we have
   $$
      \pi^\xi_n(C_R(\sigma,L)) \geq 1 - \exp\left(-ck\right).
   $$
\end{theorem}
\begin{proof}
   Let $W$ be a large enough constant depending on $\lambda$, but independent of $k,m,n$.
   Partition $R$ into squares of side length $W$, such that each midpoint of $\Lambda_n$ belongs to exactly one square; 
   in order to allow the squares to completely partition $R$, we let the squares close to the boundary of $R$ be rectangles with sides of length between $W$ and $2W$ (for simplicity we will continue to refer to them as squares).
   Consider an arbitrary order of the squares and let $Q_i$ denote the $i$th square.
   Let $G_i$ be the event that all edges of midpoint inside $Q_i$ have length at most $L$, and let 
   $G_i^\gs$ be the event that all edges of midpoint inside $Q_i$ are in ground state given $\xi$.
   For any event $F$ for which $F\cap G_i^\gs\neq\emptyset$, we obtain that 
   \begin{equation}
      \pi^\xi_n(G_i\mid F) 
      \geq 1 - \sum_{x\in \Lambda\cap Q_i}\pi^\xi_n(|\sigma_x|> L \mid F)
      \geq 1 - \sum_{x\in \Lambda\cap Q_i}c_1\alpha^{-L}
      \geq 1 - 5W^2c_1\alpha^{-L},
      \label{eq:gi}
   \end{equation}
   where the second inequality follows from the second part of Corollary~\ref{cor:tailedge} for some constant $c_1>0$, and $\alpha$ comes from
   Theorem~\ref{thm:lyapunov}.
   We set $L=L(W)$ as a function of $W$ so that $L\leq W/100$ but $5W^2c_1\alpha^{-L(W)}\to1$ as $W\to\infty$.
   Let $C_R'(\sigma,L)$ the event that there exists a path of adjacent squares such that $G_i$ holds for each square $Q_i$ in the path and 
   there are two squares in the path containing midpoint of the left and right boundaries of $R$. Clearly,
   $$
      \pi^\xi_n(C_R(\sigma,L)) \geq \pi^\xi_n(C_R'(\sigma,L)).
   $$
   
   By planar duality, if $C_R'(\sigma,L)$ does not hold, 
   then there must exist a dual path of $*$-adjacent squares such that $G_i$ does not hold for each square $Q_i$ in the 
   dual path and there are two squares in this path containing midpoints of the top and bottom boundaries of $R$. A $*$-adjacent path of squares is a 
   sequence of squares so that two consecutive squares in the sequence intersect in at least one point. Such a path must contain at least 
   $\lfloor\frac{k}{W}\rfloor$ squares. We say that a path of (not necessarily adjacent) 
   squares $Q_{j_1},Q_{j_2},Q_{j_3},\ldots$ is an \emph{invasive path} if, for all $i$,
   we have that $G_{j_i}$ does not hold and $Q_{j_{i+1}}$ is $*$-adjacent to a square 
   $Q_k$ for which there is an edge of 
   midpoint in $Q_{j_i}$ intersecting $Q_{k}$. We say that an invasive path is a \emph{top-to-bottom} invasive path in $R$ if it starts with a square
   containing a midpoint in the top boundary of $R$ and ends with a square containing a midpoint whose edge intersects the bottom boundary of $R$. 
   Note that the existence of a dual top-to-bottom path of $*$-adjacent squares as described in the beginning of the paragraph 
   implies the existence of a top-to-bottom invasive path in $R$.
   
   Letting 
   $I$ denote the set of indices $i$ such that $Q_i$ contains a midpoint in the top boundary of $R$, we have 
   \begin{align}
      \pi^\xi_n(C_R'(\sigma,L)) 
      &\geq 1 - \pi^\xi_n(\text{there is a top-to-bottom invasive path in $R$})\nonumber\\
      &\geq 1 - \sum_{i\in I}\pi^\xi_n(\text{there is a top-to-bottom invasive path in $R$ starting from $Q_i$}).
      \label{eq:startpath}
   \end{align}
   For two squares $Q_i,Q_j$ let $d(Q_i,Q_j)$ be the length of the shortest path of $*$-adjacent squares from $Q_i$ to $Q_j$.
   For each event $F$ such that $G_i^\gs\cap F\neq\emptyset$ the probability that an edge of midpoint $x\in\Lambda\cap Q_i$ with 
   $d(Q_i,Q_j)\geq \ell$ intersects $Q_j$ is 
   $$
      \pi^\xi_n(\sigma_x \cap Q_j \neq \emptyset \mid F) 
      \leq \pi^\xi_n(|\sigma_x| \geq  (\ell-1)W \mid F) 
      \leq c_1\alpha^{-(\ell-1)W},
   $$
   where the last inequality follows from the second part of Corollary~\ref{cor:tailedge}.
   Hence,
   $$
      \pi^\xi_n\left(\bigcup\nolimits_{x\in\Lambda\cap Q_i}\big\{\sigma_x \cap Q_j \neq \emptyset\big\} \mid F\right) 
      \leq 5 W^2 c_1 \alpha^{-(d(Q_i,Q_j)-1)W}.
   $$
%    If $d(Q_i,Q_j)=1$, then we use the fact that $G_j$ must not hold. 
   Let 
   $p_\ell = 5 W^2 c_1 \alpha^{-(\ell-1)W}$ for $\ell\geq2$. 
   For $\ell=1$, let $p_\ell=5W^2c_1\alpha^{-L}$, which is the bound in~\eqref{eq:gi} for the probability that $G_i$ does not hold.
   Therefore, the probability that a given top-to-bottom invasive path $Q_{j_1},Q_{j_2},Q_{j_3},\ldots,Q_{j_\kappa}$ exists 
   such that $\ell_i=d(Q_{j_i},Q_{j_{i-1}})$ is at most 
   $
      \prod_{i=2}^\kappa p_{\ell_i-1}.
   $
   Note that, by definition of invasive paths, $\ell_i\geq 2$ for all $i$.
   The number of invasive paths with a given sequence $\ell_2,\ell_3,\ldots,\ell_\kappa$ is at most
   $
      \prod_{i=2}^\kappa (2\ell_i+2)^2.
   $
   We put both expressions together, and 
   set $W$ large enough so that $(2\ell_i+2)^2 p_{\ell_i-1}\leq e^{-c_2 \ell_i}$ for some positive constant $c_2$ (that increases with $W$) and all $i$,
   which yields
   $$
      \prod_{i=2}^\kappa (2\ell_i+2)^2 p_{\ell_i-1}
      \leq \exp\left(-c_2\sum\nolimits_{i=1}^\kappa \ell_i\right).
   $$
   For the path to go from the top boundary to the bottom boundary of $R$ we need that $\sum_{i=2}^\kappa (\ell_i+1) \geq \lfloor\frac{k}{W}-1\rfloor$. 
   This implies that $\sum_{i=2}^\kappa \ell_i \geq \lfloor\frac{k}{2W}-1\rfloor$.
   Given $s=\sum_{i=2}^\kappa \ell_i$, there are at most $s$ possible values for $\kappa$ and, given $s$ and $\kappa$,
   there are at most 
   $\binom{s-1}{\kappa-1}$ possible ways to choose the values of $\ell_1,\ell_2,\ldots,\ell_\kappa$ so that $s=\sum_{i=2}^\kappa \ell_i$.
   Plugging
   everything into~\eqref{eq:startpath}, we have
   \begin{align*}
      \pi^\xi_n(C_R'(\sigma,L)) 
      &\geq 1 - \sum_{i\in I} \sum_{k,\ell_1,\ell_2,\ldots,\ell_k}\exp\left(-c_2\sum\nolimits_{i=1}^k \ell_i\right)\\
      &\geq 1 - \sum_{i\in I} \sum_{s\geq \frac{k}{2W}-2}\sum_{\kappa=\frac{k}{2W}}^s \binom{s-1}{\kappa-1}\exp\left(-c_2s\right)\\
      &\geq 1 - \sum_{i\in I} \sum_{s\geq \frac{k}{2W}-2} 2^{s-1}\exp\left(-c_2s\right).
   \end{align*}
   Setting $W$ large enough, which causes $c_2$ to be large enough, we obtain
   $$
      \pi^\xi_n(C_R'(\sigma,L)) 
      \geq 1 - \sum_{i\in I} \exp\left(-\frac{c_3 k}{4W}\right),
   $$
   for some constant $c_3>0$.
   Noting that the cardinality of $I$ is at most $m+1$ and that $k\geq c'\log m$ for some large enough $c$, depending
   on $W$,
   concludes the proof.
\end{proof}

%############################################################################################
%############################################################################################
%############################################################################################
\section{Applications to triangulations of thin triangles}\label{sec:thinrectangle}
In this section we apply our Lyapunov function to 
study asymptotic properties of triangulations of thin rectangles; i.e., triangulations of the integer points in 
$[-n/2,n/2] \times [0,k]$ as $n\to\infty$ while $k$ is kept fixed. 
Define $\Lambda^0_{n,k}$ to be the set of integer points in $[-n/2,n/2] \times [0,k]$ and $\Lambda_{n,k}$ to be the set of midpoints of a triangulation of $\Lambda^0_{n,k}$.
Given any set of constraint edges $\xi\in\Xi(\Lambda^0)$, let 
\begin{align*}
   \Omega^\xi_{n,k} \text{ be the set of triangulations of $\Lambda^0_{n,k}$ compatible with $\xi$.}
\end{align*}
We say that a triangulation $\sigma\in\Omega^\xi_{n,k}$ has a \emph{top-to-bottom crossing of unit verticals} if 
$\sigma$ contains $k$ unit vertical edges with the same horizontal coordinate.
As in~\eqref{eq:defpi}, for any $\lambda$, we denote by 
\begin{align*}
   \pi^\xi_{n,k} \text{ the stationary measure of $\mathcal{M}^\lambda(\Omega^\xi_{n,k})$.}
\end{align*}

%############################################################################################
\subsection{Vertical crossings}
Our first goal is to show that for any $\lambda\in(0,1)$ a typical triangulation 
contains many top-to-bottom crossings of edges that are unit verticals and that the stationary measure has decay of correlations.
\begin{theorem}\label{thm:toptobottom}
   Let $m=m(k)$ be large enough, and let $R$ be an $m \times k$ rectangle inside $[-n/2,n/2]\times [0,k]$. 
   Consider an arbitrary boundary condition $\xi\in\Xi(\Lambda^0_{n,k})$ such that no edge of $\xi$ intersects $R$, and take any $\lambda\in(0,1)$.
   For any triangulation $\eta\in\Omega^\xi_{n,k}$, 
   let $C_R(\eta)$ be the number of disjoint top to bottom crossings of unit verticals of $\eta$ that are inside $R$.
   Let $\sigma$ be a random triangulation distributed according to $\pi^\xi_{n,k}$.
   Then there exist positive constants $c=c(\lambda,k)$ and $\delta=\delta(\lambda,k)\in(0,1)$ such that, for any large enough $m$, we have
   $$
     \pi^\xi_{n,k}(C_R(\sigma) \leq \delta m)\leq e^{-c m}.
   $$
%    Furthermore, let $\xi,\xi'\in\Xi(\Lambda^0_{n,k})$ be two boundary conditions that do not intersect $R$.
%    Let $\sigma$ and $\sigma'$ be two random triangulations distributed according to $\pi^\xi_{n,k}$ and $\pi^{\xi'}_{n,k}$, respectively. 
%    Then we can couple $\sigma$ and $\sigma'$ such that the probability that they have less than $\delta m$ equal top to bottom crossings of unit verticals 
%    is at most $e^{-c m}$.
\end{theorem}
\begin{proof}
   Let $W>0$ be a large enough constant (independent of $m$).
   Partition $R$ into slabs of width $W$; i.e., each such slab is a translate of $[0,W]\times [0,k]$. 
   For the $i$th slab, let $\Gamma_i$ be the set of midpoints of the $i$th slab with the smallest horizontal coordinate; 
   we assume that $W$ is set in such a way that $\Gamma_i$ only contains
   midpoints for which the ground state edge is a unit vertical.
   We will sample the edges of $\Gamma_i$ for each $i$, in order. We may need to skip some values of $i$ if 
   edges in previous slabs turn out to be long edges.
   In order to do this, define the random variable $A_i\geq1$, such that $A_i=j$ iff slab $i+j$ is the first slab to the right of $\Gamma_i$ not to be intersected by an edge of midpoint in $\Gamma_i$. 
   Let $\mathcal{F}_i$ be the $\sigma$-algebra generated by $\{\sigma_x\}_{x \in \bigcup_{j=1}^i\Gamma_j}$. 
   Let $V_i$ be the event that the edges of midpoint in $\Gamma_i$ are all unit verticals. 
   We will look at events $F\in\mathcal{F}_{i-1}$ for which $F\cap V_i\neq\emptyset$, this means that under $F$ the ground state configuration of $\Gamma_i$ is unit verticals.
   For any $x\in \Gamma_i$ and any $i$, denote by $\bar\sigma_x$ the (ground state) unit vertical edge whose midpoint is $x$.
   Thus, for any event $F\in\mathcal{F}_{i-1}$ for which $F\cap V_i\neq\emptyset$ and any $j\geq2$, we have that 
   $$
      \pi^\xi_{n,k}\big(A_i \geq j \mid F\big) 
      \leq \pi^\xi_{n,k}\left(\bigcup\nolimits_{x\in \Gamma_i}\big\{\Psi_{\bar\sigma_x}(\sigma)\geq \alpha^{(j-1)W-1}\big\} \mid F\right)
      \leq \sum_{x\in \Gamma_i}\pi^\xi_{n,k}\left(\Psi_{\bar\sigma_x}(\sigma)\geq \alpha^{(j-1)W-1} \mid F\right).
   $$
   Using Markov's inequality, we obtain
   \begin{equation}
      \pi^\xi_{n,k}\big(A_i \geq j \,\mid\, F\big) 
      \leq \sum_{x\in \Gamma_i}\pi^\xi_{n,k}\left(\Psi_{\bar\sigma_x} \mid F\right)\alpha^{-(j-1)W+1}
      \leq \sum_{x\in \Gamma_i}c\alpha^{-(j-1)W+1}
      \leq kc\alpha^{-(j-1)W+1},
      \label{eq:ai}
   \end{equation}
   where $c$ does not depend on $W$. We set $W$ large enough so that we can find a $\beta=\beta(W)\in(0,1)$ for which 
   \begin{equation}
      \pi^\xi_{n,k}\big(A_i \geq j \,\mid\, F\big)
      \leq (1-\beta)^{j-1} \text{ for all $j\geq 2$}.
      \label{eq:tailb}
   \end{equation}
   Consequently, $\pi^\xi_{n,k}\big(A_i=1 \,\mid\, F\big)\geq \beta$. In other words, $A_i$ is stochastically dominated by a geometric random variable of parameter $\beta$, uniformly over 
   $F\in\mathcal{F}_{i-1}$ for which $F\cap V_i\neq\emptyset$. 
   Using Theorem~\ref{thm:surgery} for each $x\in\Gamma_i$ we obtain a constant $c_1>0$ that is independent of $m$ and $k$ so that 
   \begin{equation}
      \pi^\xi_{n,k}\big(V_i \,\mid\, F\big)
      \geq e^{-c_1 k} \quad\text{for all event $F\in\mathcal{F}_{i-1}$ with $F\cap V_i\neq \emptyset$}.
      \label{eq:c}
   \end{equation}

   Let $A_1',A_2',\ldots$ be i.i.d.\ geometric random variables of parameter $\beta$. 
   We define a sequence of random variables $k_1,k_2,\ldots$ inductively as follows.
   Let $k_1=1$. Assume that $k_j$ has been defined. 
   Sample $A_{k_j}'$ and $A_{k_j}$ in a coupled way so that $A_{k_j}\leq A_{k_j}'$. Given $A_{k_j}$, sample the edges of midpoint in $\Gamma_{k_j}$.
   Set $k_{j+1}=k_j+A_{k_j}'$, and iterate.
   Define the stopping time 
   $$
      \tau=\min\big\{j\geq 1 \colon k_j \geq \tfrac{m}{W}-1\big\}.
   $$
   In other words, $k_\tau$ is the first value of $k$ falling outside $R$.
   Let $S_{n,p}$ be a binomial random variable of parameters $n$ and $p$.
   Then, setting $\delta=e^{-c_1 k}\beta/(2W)$, we have
   \begin{align*}
      \pi^\xi_{n,k}\big(C_R(\sigma) \leq \delta m\big)
      &\leq \pi^\xi_{n,k}\big(\big\{\tau < 2\beta m /(3W)\big\} \cup \big\{S_{2\beta m/(3W),e^{-c_1k}} \leq \delta m\big\}\big)\\
      &\leq \pi^\xi_{n,k}\big(\tau < 2\beta m /(3W)\big)+ \pi^\xi_{n,k}\big(S_{2\beta m/(3W),e^{-c_1k}} \leq \delta\big).
   \end{align*}
   The first term is the probability that a sum of $2\beta m/(3W)$ i.i.d.\ geometric random variables of success probability $\beta$ 
   is larger than $\frac{m}{W}-1$, which using Lemma~\ref{lem:cbgeometric} with $\epsilon=\frac{1}{2}-\frac{3W}{2m}\in(1/3,1/2)$ gives
   $$
      \pi^\xi_{n,k}\big(\tau < 2\beta m /(3W)\big)
      \leq \exp\left(-\frac{(1/3)^2}{2(1+1/2)}\,\frac{2\beta m}{3W} \right)
      = \exp\left(-\frac{2\beta m}{81 W} \right).
   $$
   The second term is the probability that a Binomial random variable is smaller than $3/4$ of its expectation, which can be bounded above using Lemma~\ref{lem:cbbinomial}, yielding
   $$
      \pi^\xi_{n,k}\big(S_{2\beta m/(3W),e^{-c_1k}} \leq \delta m\big)
      \leq \exp\left(-\frac{1}{2\, 4^2}\,\frac{2\beta m e^{-c_1k}}{3W}\right)
      =\exp\left(-\frac{\delta m}{24}\right).
   $$
\end{proof}

A similar proof establishes that we can couple two triangulations so that they have the same vertical crossing.
\begin{theorem}\label{thm:couple}
   Let $m=m(k)$ be large enough, and let $R$ be an $m \times k$ rectangle inside $[0,n]\times [0,k]$. 
   Consider two arbitrary boundary conditions $\xi,\xi'\in\Xi(\Lambda^0_{n,k})$ such that no edge of $\xi\cup\xi'$ intersects $R$. 
   Take any $\lambda\in(0,1)$.
   Let $\sigma$ and $\sigma'$ be two random triangulations distributed according to $\pi^\xi_{n,k}$ and $\pi^{\xi'}_{n,k}$, respectively. 
   Then there exist positive constants $c=c(\lambda,k)$ and $\delta=\delta(\lambda,k)\in(0,1)$ such that, for any large enough $m$, 
   we can couple $\sigma$ and $\sigma'$ such that the probability that $\sigma,\sigma'$  
   have less than $\delta m$ equal top to bottom crossings of unit verticals 
   in $R$ is at most $e^{-c m}$.
\end{theorem}
\begin{proof}
   The proof is similar to the proof of Theorem~\ref{thm:toptobottom}.
   We only need to define the $A_i$ so that $A_i=j$ iff slab $i+j$ is the first slab to the right of
   $\Gamma_i$ not to be intersected by any edge of $\sigma$ and $\sigma'$ with midpoint in $\Gamma_i$, and 
   also need to define $V_i$ as the event that both $\sigma$ and $\sigma'$ 
   have unit verticals at edges of midpoint in $\Gamma_i$. 
   Then~\eqref{eq:ai} translates to the bound $\mathbf{P}\big(A_i\geq j\,\mid\,F\big)\leq 2kc\alpha^{-(j-1)W+1}$,
   whereas~\eqref{eq:c} holds with no change since the bound there holds uniformly on $F \cap V_i\neq\emptyset$. 
\end{proof}

%############################################################################################
\subsection{Decay of correlations}
The theorem below establishes decay of correlations for lattice triangulations in thin rectangles uniformly on the 
boundary conditions. We denote by the \emph{horizontal distance} between two points the distance between their horizontal coordinates 
(ignoring their vertical coordinates).
\begin{theorem}[Decay of correlations]\label{thm:correlations}
   Consider any region $Q\subset [-n/2,n/2]\times [0,k]$ and let $\Upsilon=\Upsilon'\cap \Lambda_{n,k}$ be the set of midpoints in $\Upsilon'$.
   Take any $\lambda\in(0,1)$.
   Let $\xi,\xi'\in\Xi(\Lambda^0_{n,k})$ be two boundary conditions that do not intersect $Q$, and 
   let $m$ be the horizontal distance between $\Upsilon$ and the edges of $\xi\cup\xi'$. 
   Let $\sigma$ and $\sigma'$ be two random triangulations distributed according to $\pi^\xi_{n,k}$ and $\pi^{\xi'}_{n,k}$, respectively, and let 
   $\sigma_\Upsilon$ and $\sigma'_\Upsilon$ denote the configuration of the edges of midpoints in $\Upsilon$ in $\sigma$ and $\sigma'$, respectively. 
   Then there exist a positive constant $c=c(\lambda,k)$ 
   and a coupling $\mathbf{P}$
   between $\pi^\xi_{n,k}$ and $\pi^{\xi'}_{n,k}$ such that 
   $
      \mathbf{P}\big(\sigma_\Upsilon = \sigma'_\Upsilon\big)
      \geq 1-e^{-cm}.
   $
\end{theorem}
\begin{proof}
   The proof uses Theorem~\ref{thm:couple}. 
   Let 
   $R$ be the rectangle of height $k$ between the leftmost point of $Q$ and the rightmost point of $\xi\cup\xi'$ that is to the left of $Q$.
   Note that $R$ is a $m'\times k$ rectangle with $m'\geq m$.
   We can apply Theorem~\ref{thm:couple} with this choice of $R$ to show that if we sample the edges with midpoint in
   $R$ from left to right, with probability at least $1-e^{-c_1 m'}$ for some constant $c_1>0$, 
   at some time we sample the same top-to-bottom crossing of unit verticals in both $\sigma$ and $\sigma'$. 
   Then we repeat the same argument with $R'$ being the rectangle between the rightmost point of $Q$ and 
   the leftmost point of $\xi\cup\xi'$ that is to the right of $Q$. 
   If it turns out that, during the construction described in the coupling above, 
   $\sigma$ and $\sigma'$ at some moment sample the same top-to-bottom crossing of unit verticals inside $R$, 
   and similarly inside $R'$, then letting 
   $\Upsilon'$ denote the set of midpoints between 
   these two crossings we obtain that we can couple $\sigma,\sigma'$ so that they coincide in $\Upsilon'$. 
   Since $\Upsilon'\supseteq \Upsilon$, the theorem is established.
\end{proof}

%############################################################################################
\subsection{Local limits}
In Theorem~\ref{thm:tightness} we showed that the measure on local configurations of a sequence of lattice triangulations for any $\lambda\in(0,1)$ is tight. 
In the case of triangulations of thin triangles, taking advantage of decay of correlations (cf.\ Theorem~\ref{thm:correlations}),
we can establish the existence of a unique local limit.

We adapt the definitions from Section~\ref{sec:tightness}. 
For any $\ell$ such that $0<\ell <n$, let $\Upsilon_\ell=[-\ell/2,\ell/2]\times[0,k]\cap\Lambda_n$ be the midpoints inside 
$[-\ell/2,\ell/2]\times[0,k]$, and 
let $\Gamma_\ell$ be the set of configurations of disjoint edges with midpoint in $\Upsilon_\ell$ such that for any 
$\gamma\in\Gamma_\ell$ there exists at 
least one $n$ and one triangulation $\sigma\in \Omega_{n,k}$ such that $\gamma=\sigma_{\Upsilon_\ell}$, where we recall that
$\sigma_{\Upsilon_\ell}$ denotes the set of edges of $\sigma$ whose midpoints belong to $\Upsilon_\ell$.
Let $\pi_{n,k}^\ell$ be the stationary measure over triangulations of $\Omega_{n,k}$ of the edges with midpoint in $\Upsilon_\ell$. 
More precisely, for any $\gamma\in\Gamma_\ell$, we have 
\begin{equation}
   \pi_{n,k}^\ell(\gamma) = 
   \frac{1}{Z_{n,k}^\ell}\sum_{\sigma\in\Omega_{n,k}\colon \sigma_{\Upsilon_\ell}=\gamma} \pi_{n,k}(\sigma),
   \label{eq:pink}
\end{equation}
where $Z_{n,k}^\ell$ is a normalizing constant to make $\pi_{n,k}^\ell$ a probability measure over $\Gamma_\ell$. 

\begin{theorem}\label{thm:locallimits}
   Let $k,\ell$ be fixed integers.
   For any $\lambda\in(0,1)$, 
   there exists a unique probability measure $\pi_{\infty,k}^\ell$ over $\Gamma_\ell$ so that $\pi_{n,k}^\ell$ converges to $\pi_{\infty,k}^\ell$ as 
   $n\to\infty$.
\end{theorem}
\begin{proof}
    Let $\delta>0$ and $\ell\in\mathbb{Z}_+$ be arbitrary.
    Let $n_0=n_0(\delta,\ell)$ be a large enough integer.
    We will show that, for any $\gamma\in\Gamma_\ell$ and any $n\geq n_0$ we have 
    $
       |\pi_{n,k}^\ell(\gamma)-\pi_{n_0,k}^\ell(\gamma)| <\delta.
    $
    Note that a triangulation of $\Lambda_{n_0,k}^0$ can be seen as a triangulation of 
    $\Lambda_{n,k}^0$ with boundary condition containing the unit vertical edges of horizontal coordinates $-n_0/2$ and $n_0/2$.
    Then Theorem~\ref{thm:correlations} gives that there exists a coupling $\mathbf{P}$ between $\pi_{n,k}$ with $\pi_{n_0,k}$ so that,
    if $\sigma,\sigma'$ are triangulations distributed as $\pi_{n,k}$ and $\pi_{n_0,k}$, respectively,
    then
    $$
       \mathbf{P}(\sigma_{\Upsilon_\ell}\neq \sigma_{\Upsilon_\ell}')\leq \exp\left(-c \left(\frac{n_0-\ell}{2}\right)\right)<\frac{\delta}{2},
    $$
    for some constant $c>0$, where the last step follows by having $n_0$ large enough with respect to $\delta$ and $\ell$.
    The theorem follows since $\frac{1}{2}\sum_{\gamma\in\Gamma_\ell}|\pi_{n,k}^\ell(\gamma)-\pi_{n,k}^\ell(\gamma)| 
    \leq \mathbf{P}(\sigma_{\Upsilon_\ell}\neq \sigma_{\Upsilon_\ell}')$.
\end{proof}

%############################################################################################
\subsection{Distributional limits of induced graph}
For each $\sigma\in\Omega_{n,k}$ we obtain a labelled, planar graph $G$ with vertex set $\Lambda_{n,k}^0$ and 
edges given by the edges of the triangulation $\sigma$. We call $G$ the induced graph of $\sigma$.
Let $G_{n,k}$ be the (random) induced graph of a random triangulation distributed according to $\pi_{n,k}$. 
In a seminar work, Benjamini and Schramm~\cite{BS} introduced the distributional limit of finite graphs. 
Given a sequence of graphs $\{H_{n}\}_n$, 
the distributional limit of $\{H_{n}\}_n$ is a random infinite graph $\mathcal{H}$, rooted at a (possibly random) vertex $\rho$,
with the property that finite neighborhoods of $H_n$ around a random vertex converge in distribution to neighborhoods of $\mathcal{H}$ around $\rho$.
Below we show that the distributional limit of $G_{n,k}$ exists as $n\to\infty$ and $k$ remains fixed.
Let $\Lambda_{\infty,k}^0$ be the integers points inside $\mathbb{Z}\times \{0,1,\ldots,k\}$.
\begin{theorem}\label{thm:distrlimit}
   Let $k\geq 1$ and $\lambda\in(0,1)$ be fixed.
   There exists a distribution $\mathcal{G}_k$ on infinite graphs with vertex set $\Lambda_{\infty,k}^0$ such that 
   if $\rho$ is a vertex uniformly distributed 
   on $\{0\}\times \{0,1,2,3,\ldots,k\}$ we have that 
   $\mathcal{G}_k$ rooted at $\rho$ is 
   the distributional limit of $\{G_{n,k}\}_n$.
\end{theorem}
\begin{proof}
   First we sample $v_n=(i,j)$ uniformly at random from $\Lambda_{n,k}^0$. 
   We need to show that, for any fixed integer $\ell$, the $\ell$-th neighborhood of $v_n$ 
   in $G_{n,k}$ converges in distribution to some measure.
   Let $\delta>0$ be an arbitrary number.
   Note that with probability at least $1-\frac{1}{\sqrt{n}}$, $v_n$ is at distance 
   at least $\frac{\sqrt{n}}{3}$ from the boundary of $\Lambda_{n,k}^0$. 
   Given any $c_1>1$, let $c_2=c_2(c_1,\delta)$ be large enough so that Theorem~\ref{thm:crosstail} gives that 
   with probability at least $1-\delta/2$, all edges of midpoint in $A_n=v_n+[-c_1\ell^2,c_1\ell^2]\times[0,k]$ have size at most 
   $c_2\log (c_1\ell^2 k)$. Now set $c_1$ large enough so that $\frac{c_1\ell^2}{c_2 \log (c_1\ell^2 k)}\geq 3\ell$. 
   Therefore, under this event, the $\ell$-th neighborhood around $v_n$ in $G_{n,k}$ has only edges of midpoints inside 
   $A_n$. But the set of edges of midpoints inside $A_n$ converge as $n\to\infty$ by Theorem~\ref{thm:locallimits}, concluding the proof.
\end{proof}

We call the random graph $\mathcal{G}_k$ the \emph{infinite lattice triangulation on slabs}.
We can show that the degree of a given vertex of the infinite lattice triangulation on slabs has an exponential tail.
\begin{theorem}\label{thm:exptail}
   For any given $\rho\in\Lambda_{\infty,k}^0$, let $d_{\mathcal{G}_k}(\rho)$ be a random variable denoting the degree of $\rho$ in a graph distributed 
   according to $\mathcal{G}_k$. Then there exists a positive constant $c$ such that for any $\ell\geq 0$ we have that 
   $
      \PR\big(d_{\mathcal{G}_k}(\rho)\geq \ell\big)\leq \exp(-c\ell).
   $
\end{theorem}
\begin{proof}
   Take $n$ large enough so that $\rho\in\Lambda_{n,k}$. Let $\Gamma$ be the set of ground state horizontal and vertical edges forming the $2\times 2$ square 
   centered at $\rho$. Then the degree of $\rho$ is at most 8 plus the number of edges intersecting edges of $\Gamma$. Since this last random
   variable has an exponential tail for all large enough $n$ by Theorem~\ref{thm:crosstail} and Proposition~\ref{pro:numbercross}, the proof is completed.
\end{proof}

It was shown by Gurel-Gurevich and Nachmias~\cite{GGN} that any distributional limit of graphs where the degrees have an exponential 
tail is a recurrent graph almost surely. Therefore, the two theorems above imply the following result.
\begin{corollary}\label{cor:recurrence}
   For any integer $k\geq 1$ and any real number $\lambda\in(0,1)$ the infinite lattice triangulation with parameter $\lambda$ 
   on $k$-slabs is almost surely recurrent.
\end{corollary}

%############################################################################################
%############################################################################################
%############################################################################################
\section{Open problems}\label{sec:problems}
\begin{itemize}
   \item In Theorem~\ref{thm:locallimits} we show that the local limit of random triangulations on thin rectangles exist. 
      In the case of $n\times n$ triangulations, our results only give that \emph{subsequential} limits exists (cf.\ Theorem~\ref{thm:tightness}).
      An interesting open problem is to establish whether for any $\lambda\in(0,1)$ 
      there exists a unique measure $\pi_\infty^k$ over $\Gamma_k$ so that $\pi_n^k$ converges to $\pi_\infty^k$ as $n\to\infty$?
   
   \item In the context of Theorems~\ref{thm:smalltriangles} and~\ref{thm:smalltrianglesgen}, 
      establish the existence of crossings of small triangles in the presence of arbitrary boundary conditions, where small refers to 
      the difference between the length of the edges in the triangles and their ground state.
   
   \item Establish decay of correlations (in the context of Theorem~\ref{thm:correlations}) 
      and the corresponding results for local limits and distributional limits of the induced graph (Theorems~\ref{thm:locallimits}--\ref{thm:exptail} and 
      Corollary~\ref{cor:recurrence}) for triangulations of $n\times n$. 
\end{itemize}

%############################################################################################
%############################################################################################
%############################################################################################
\section*{Acknowledgments}
I am grateful to Pietro Caputo, Fabio Martinelli and Alistair Sinclair for several useful discussions,
and to Gr\'egory Miermont for suggesting the problem of local limits of lattice triangulations.

%############################################################################################
%############################################################################################
%############################################################################################
\bibliographystyle{plain}
\bibliography{lyapunov}

\appendix

%############################################################################################
\section{Standard large deviation results}

We use the following standard Chernoff bounds and large deviation results.
\begin{lemma}[Chernoff bound for binomial~{\cite[Corollary~A.1.10]{AlonSpencer}}]\label{lem:cbbinomial}
   Let $X$ be the sum of $n$ i.i.d.\ Bernoulli random variables of mean $p$.
%    Then,
%    $\pr{X \geq np+a} \leq \exp\left(a-(pn+a)\log\left(1+\frac{a}{pn}\right)\right)$. In the other direction,
   For any $\epsilon\in(0,1)$, we have
   $\pr{X \leq (1-\epsilon)np} \leq \exp\left(-\frac{\epsilon^2np}{2}\right)$.
\end{lemma}

\begin{lemma}[Chernoff bound for geometric]\label{lem:cbgeometric}
   Let $X$ be the sum of $n$ i.i.d.\ geometric random variables with mean $p$.
   Then, for any $\epsilon > 0$,
   $
     \pr{X \geq (1+\epsilon) \frac{n}{p} } \leq \exp\left(-\frac{\epsilon^2}{2(1+\epsilon)}n \right).
   $
\end{lemma}
\begin{proof}
   Enumerate each trial of the $n$ geometric random variables as $Z_1,Z_2,\ldots$. Then the event 
   $X \geq (1+\epsilon) \frac{n}{p}$ implies that $\sum_{i=1}^{(1+\epsilon) \frac{n}{p}} Z_i \leq n$.
   Since $\sum_{i=1}^{(1+\epsilon) \frac{n}{p}} Z_i$ is a binomial random variable with parameters 
   $(1+\epsilon) \frac{n}{p}$ and $p$, the result follows from Lemma~\ref{lem:cbbinomial}.
\end{proof}

%############################################################################################
\section{Absence of FKG}~\label{sec:nofkg}
Given a set of constraint $\xi$, recall that $\bar\sigma_x$ denotes the ground state edge of midpoint $x$ given $\xi$, with an arbitrary
choice among unit diagonals.
\begin{lemma}
   There exist a set of constraints $\xi$ (e.g., Figure~\ref{fig:nofkg}) such that the following holds for two midpoints $x,y\in\Lambda$:
   $
      \pi^\xi\big(\sigma_x = \bar\sigma_x \mid \sigma_y \neq \bar\sigma_y\big)
      > \pi^\xi\big(\sigma_x = \bar\sigma_x \mid \sigma_y=\bar\sigma_y\big)
   $
   and
   $
      \pi^\xi\big(\sigma_x = \bar\sigma_x\big) >
      \pi^\xi\big(\sigma_x = \bar\sigma_x \mid \sigma_y=\bar\sigma_y\big).
   $
   As a consequence, the FKG inequality does not hold.
\end{lemma}
\begin{proof}
   Refer to Figure~\ref{fig:nofkg}. 
   \begin{figure}[htbp]
      \begin{center}
         \includegraphics[scale=.9]{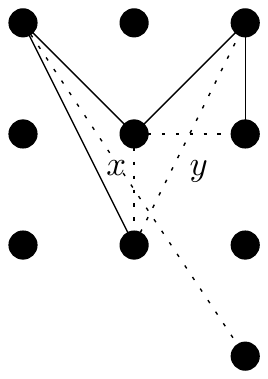}
      \end{center}\vspace{-.5cm}
      \caption{An example violating positive correlation.}
      \label{fig:nofkg}
   \end{figure}
   Solid edges are constraint edges. Dotted edges are the possible configuration of edges with midpoints in $x$ and $y$. 
   Note that if $\sigma_y$ is in ground state given the constraints, then there are two choices for the edge $\sigma_x$, therefore
   $\pi^\xi\big(\sigma_x = \bar\sigma_x \mid \sigma_y=\bar\sigma_y\big) <1$. 
   On the other hand, if $\sigma_y$ is not in ground state, then $\sigma_x$ can only be in ground state,
   giving that 
   $$
      \pi^\xi\big(\sigma_x = \bar\sigma_x \mid \sigma_y \neq \bar\sigma_y\big)
      =1
      >\pi^\xi\big(\sigma_x = \bar\sigma_x \mid \sigma_y=\bar\sigma_y\big),
   $$
   which establishes the first inequality of the lemma.
   For the second inequality, note that 
   \begin{align*}
      \pi^\xi\big(\sigma_x =\bar\sigma_x\big)
      &= \pi^\xi\big(\sigma_x =\bar\sigma_x \mid \sigma_y   =  \bar\sigma_y\big)\pi^\xi\big(\sigma_y   =  \bar\sigma_y\big)
       + \pi^\xi\big(\sigma_x =\bar\sigma_x \mid \sigma_y \neq \bar\sigma_y\big)\pi^\xi\big(\sigma_y \neq \bar\sigma_y\big)\\
      &= \pi^\xi\big(\sigma_x =\bar\sigma_x \mid \sigma_y   =  \bar\sigma_y\big)\pi^\xi\big(\sigma_y   =  \bar\sigma_y\big)
       + \pi^\xi\big(\sigma_y \neq\bar\sigma_y\big).
   \end{align*}
   Since $\pi^\xi\big(\sigma_y =\bar\sigma_y\big)\in(0,1)$, we have that $\pi^\xi\big(\sigma_x =\bar\sigma_x\big)$ is strictly between 
   $\pi^\xi\big(\sigma_x =\bar\sigma_x \mid \sigma_y=\bar\sigma_y\big)$ and $1$.
\end{proof}

\end{document}